\newif\ifpaper
\paperfalse 

\ifpaper
\documentclass{article}
\usepackage[final]{neurips_2019}
\else
\documentclass[11pt]{article}
\fi

\usepackage[utf8]{inputenc} 
\usepackage[T1]{fontenc}    
\usepackage{hyperref}       
\usepackage{url}            
\usepackage{booktabs}       
\usepackage{amsfonts}       
\usepackage{nicefrac}       
\usepackage{microtype}      

\usepackage{amsmath,amssymb,amsthm}
\usepackage{graphicx,enumerate}
\usepackage{comment,url,epstopdf}
\usepackage[x11names,usenames,dvipsnames,svgnames,table,rgb]{xcolor}
\usepackage{pgf,tikz,pgfplots}
\usetikzlibrary{arrows}
\usepackage{subcaption}
\numberwithin{equation}{section}
\ifpaper
\else
\usepackage[top=0.5in, bottom=0.8in, left=0.9in, right=0.9in, marginparwidth=1.5in]{geometry}

\fi

\newtheorem{theorem}{Theorem}[section]

\newtheorem{proposition}{Proposition}[section]
\newtheorem{lemma}{Lemma}[section]
\theoremstyle{definition}

\newtheorem{definition}{Definition}[section]

\newtheorem{assumption}{Assumption}

\newcommand{\nn}{\mathbb{N}} 
\newcommand{\norm}[1]{\left\Vert {#1} \right\Vert} 
\newcommand{\erl}{\left(-\infty , +\infty\right]} 
\newcommand{\dom}[1]{\mathrm{dom}\,{#1}} 
\newcommand{\idom}[1]{\mathrm{int\,dom}\,{#1}} 
\newcommand{\cdom}[1]{\overline{\mathrm{dom}\,{#1}}} 

\newcommand{\act}[1]{\left\langle {#1} \right\rangle} 

\newcommand{\argmin}{\mathrm{argmin}}

\newcommand{\sgn}{\mathrm{sgn}}

\newcommand{\barc}{\overline{C}}

\newcommand{\HHH}{\mathcal{H}}

\newcommand{\PPP}{\mathcal{P}}
\newcommand{\SSS}{\mathcal{S}}
\newcommand{\R}{\mathbb{R}}

\newcommand{\real}{\mathbb{R}} 

\graphicspath{{./figures/}}

\title{\ifpaper\else\vspace{-1.5em}\fi{\bf Beyond Alternating Updates for Matrix Factorization\\
 with Inertial Bregman Proximal Gradient Algorithms}}
\author{%
  David S.~Hippocampus\thanks{Use footnote for providing further information
    about author (webpage, alternative address)---\emph{not} for acknowledging
    funding agencies.} \\
  Department of Computer Science\\
  Cranberry-Lemon University\\
  Pittsburgh, PA 15213 \\
  \texttt{hippo@cs.cranberry-lemon.edu} \\
}
\ifpaper
\author{Mahesh Chandra Mukkamala\\
Mathematical Optimization Group\\
 Saarland University, Germany\\
 \texttt{mukkamala@math.uni-sb.de}\\
 \And
 Peter Ochs\\
 Mathematical Optimization Group\\
 Saarland University, Germany\\
 \texttt{ochs@math.uni-sb.de}}
\else
\author{Mahesh Chandra Mukkamala\thanks{Mathematical Optimization Group, Saarland University, Germany, E-mail: \texttt{mukkamala@math.uni-sb.de}}  \quad\quad  Peter Ochs\thanks{Mathematical Optimization Group, Saarland University, Germany,  E-mail: \texttt{ochs@math.uni-sb.de}}}
\fi
\date{}

\begin{document}

\maketitle
	\ifpaper\else\vspace{-2em}\fi
	\begin{abstract}
	Matrix Factorization is a popular non-convex optimization problem,  for which alternating minimization schemes are mostly used. They usually suffer from the major drawback that the solution is biased towards one of the optimization variables. A remedy is non-alternating schemes. However, due to a lack of Lipschitz continuity of the gradient in matrix factorization problems, convergence cannot be guaranteed. A recently developed approach relies on the concept of Bregman distances, which generalizes the standard Euclidean distance. We exploit this theory by proposing a novel Bregman distance for matrix factorization problems, which, at the same time, allows for simple/closed form update steps. Therefore,  for non-alternating schemes, such as the recently introduced Bregman Proximal Gradient (BPG) method and an inertial variant Convex--Concave Inertial BPG (CoCaIn BPG), convergence of the whole sequence to a stationary point is proved for Matrix Factorization. In several experiments, we observe a superior performance of our non-alternating schemes in terms of speed and objective value at the limit point.
	\end{abstract}
\ifpaper
\else
	\smallskip

\fi
\section{Introduction} \label{Sec:Intro}
	Matrix factorization  has numerous applications in Machine Learning \cite{MS2008,SRJ2005}, Computer Vision \cite{CVR2016,SMF2010,XLYZ2017,HV2019}, Bio-informatics \cite{SD2006,BTGM2004} and many others. Given a matrix ${\bf A} \in \R^{M \times N}$, one is interested in the factors ${\bf U} \in \R^{M \times K}$ and ${\bf Z} \in \R^{K \times N}$ such that ${\bf A} \approx {\bf U}{\bf Z}$ holds. This is usually cast into the following non-convex optimization problem 
	\begin{equation}\label{eq:prob-2}
	\min_{{\bf U}\in \mathcal U, {\bf Z}\in \mathcal Z}\, \left\{ \Psi \equiv  \frac 12\norm{{\bf A} - {\bf U}{\bf Z}}^2_{F} +  {\mathcal R}_1({\bf U}) +  {\mathcal R}_2({\bf Z}) \right\}\,,
	\end{equation}
	where $\mathcal U, \mathcal Z$ are constraint sets and ${\mathcal R}_1, {\mathcal R}_2$ are regularization terms. The most frequently used techniques for solving matrix factorization problems involve alternating updates (Gauss--Seidel type methods \cite{GL2012})  like PALM \cite{BST2014}, iPALM \cite{PS2016}, BCD \cite{XY2013}, BC-VMFB \cite{CPR2016}, HALS \cite{CZA2007} and many others. A common disadvantage of these schemes is their bias towards one of the optimization variables. Such alternating schemes involve fixing a subset of variables to do the updates. In order to guarantee convergence to a stationary point, alternating schemes require the first term in \eqref{eq:prob-2} to have a Lipschitz continuous gradient only with respect to each subset of variables. However, in general Lipschitz continuity of the gradient fails to hold for all variables. The same problem appears in various practical applications such as Quadratic Inverse Problems, Poisson Linear Inverse Problems, Cubic Regularized Non-convex  Quadratic Problems and Robust Denoising Problems with Non-convex Total Variation Regularization \cite{MOPS2019,BSTV2018,BBT2016}. They belong to the following broad class of non-convex additive composite minimization problems 
	\begin{equation}\label{eq:prob-1}
			(\PPP) \qquad \inf \left\{ \Psi \equiv f\left(x\right) + g\left(x\right) : \; x \in \barc \right\},
	\end{equation}
	where $f$ is potentially a non-convex extended real valued function, $g$ is a smooth (possibly non-convex) function and $\barc$ is a nonempty, closed, convex set in $\real^{d}$. In order to use non-alternating schemes for \eqref{eq:prob-2}, the gradient Lipschitz continuity must be generalized. Such a generalization was initially proposed by \cite{BDX2011} and popularized by \cite{BBT2016} in convex setting and for non-convex problems in \cite{BSTV2018}. They are based on a generalized proximity measure known as Bregman distance and have recently led to new algorithms to solve \eqref{eq:prob-1}: Bregman Proximal Gradient (BPG) method \cite{BSTV2018} and its inertial variant Convex--Concave Inertial BPG (CoCaIn BPG) \cite{MOPS2019}. \ifpaper\else\\\fi

	BPG generalizes the proximal gradient method from Euclidean distances to Bregman distances as  proximity measures. Its convergence theory relies on the generalized Lipschitz assumption, discussed above, called $L$-smad property \cite{BSTV2018}. It involves an upper bound and a lower bound, where the upper bound involves a convex majorant to control the step-size of BPG. However, the significance of lower bounds for BPG was not clear. In non-convex optimization literature, the lower bounds which involve concave minorants were largely ignored. Recently, extending on  \cite{WCP2017,Ochs18}, CoCaIn BPG changed this trend by justifying the usage of lower bounds to incorporate inertia  for faster convergence \cite{MOPS2019}. Moreover, the generated inertia is adaptive, in the sense that it changes according to the function behavior, i.e.,  CoCaIn BPG does not use an inertial parameter depending on the iteration counter  unlike  Nesterov Accelerated Gradient (NAG) method \cite{N1983} (also FISTA \cite{BT2009}) in the convex setting. \ifpaper\else\\\fi

	In this paper we ask the question: \textit{"Can we apply BPG and CoCaIn BPG efficiently for Matrix Factorization problems?''}. This question is significant, since convergence of the Bregman minimization variants BPG and CoCaIn BPG relies on the $L$-smad property, which is non-trivial to verify and an open problem for Matrix Factorization. Another crucial issue is the efficient computability of the algorithm's update steps, which is particularly hard due to the coupling between two subsets of variables. We successfully solve these challenges.\ifpaper\else\\\fi

	\textbf{Contributions.} We make recently introduced powerful Bregman minimization based algorithms BPG \cite{BSTV2018} and CoCaIn BPG \cite{MOPS2019} and the corresponding convergence results applicable to the matrix factorization problems. Experiments show a significant advantage of BPG and CoCaIn BPG which are non-alternating by construction, compared to popular alternating minimization schemes in particular PALM \cite{BST2014} and iPALM \cite{PS2016}. The proposed algorithms require the following non-trivial contributions:
	\begin{itemize}
	\item  We propose a novel Bregman distance for Matrix Factorization with the following auxiliary function (called kernel generating distance) with certain $c_1,c_2>0$:
	\[
	h({\bf U}, {\bf Z}) = c_1 \left(\frac{\norm{\bf U}_F^2 +\norm{\bf Z}_F^2 }{2} \right)^2 + c_2\left( \frac{\norm{\bf U}_F^2 +\norm{\bf Z}_F^2 }{2}\right)\,.
	\]
	The generated Bregman distance embeds the crucial coupling between the variables ${\bf U}$ and ${\bf Z}$. We prove the L-smad property with such a kernel generating distance and infer  convergence of BPG and CoCaIn BPG to a stationary point.
	\item  We compute the analytic solution for subproblems of the proposed variants of BPG, for which the usual analytic solutions based on Euclidean distances cannot be used.\ifpaper\else\\\fi
	\end{itemize}


	\textbf{Simple Illustration of BPG for Matrix Factorization.} Consider the following simple matrix factorization optimization problem, where we set $\mathcal R_1:=0$ and $\mathcal R_2 :=0$ in \eqref{eq:prob-2}
	\begin{equation}\label{eq:mat-fac-ex-1}
	\min_{{\bf U}\in \R^{M \times K}, {\bf Z}\in \R^{K \times N}}\, \left\{ \Psi({\bf U},{\bf Z}) =  \frac 12\norm{{\bf A} - {\bf U}{\bf Z}}^2_{F}  \right\}\,.
	\end{equation}
	For this problem, the update steps of \textbf{Bregman Proximal Gradient for Matrix Factorization (BPG-MF)} given in Section~\ref{ssec:bpg-mf-main} (also see Section~\ref{ssec:closed-form}) with a chosen $\lambda \in (0,1)$   are the following:  \ifpaper\else\\\fi

	{\centering
		\fcolorbox{black}{Orange!10}{\parbox{0.98\textwidth}{ In each iteration, compute  $t_k = 3(\norm{{\bf U^{k}}}_F^2 + \norm{{\bf Z^{k}}}_F^2) + \norm{{\bf A}}_F$ and perform the intermediary gradient descent  steps (non-alternating) for ${\bf U}$ and ${\bf Z}$ independently with step-size $\frac{\lambda}{t_k}$:
				\begin{align*}
				{\bf P^k} &= {\bf U^k} - \frac{\lambda}{t_k}\left[({\bf U^{k}}{\bf Z^{k}}-{\bf A})({\bf Z^{k}})^T\right] \,,\quad {\bf Q^k} = {\bf Z}^{k} - \frac{\lambda}{t_k}\left[({\bf U^{k}})^T({\bf U^{k}}{\bf Z^{k}}-{\bf A})\right]\,.
				\end{align*}
			Then, the additional scaling steps ${\bf U^{k+1}} = r t_k{\bf P^k}$ and ${\bf Z^{k+1}} = r t_k{\bf Q^k}$ are required, where
			 the scaling factor $r\geq 0$ satisfies a cubic equation:
			 $3t_k^2\left(\norm{{\bf P^k}}_F^2 + \norm{{\bf Q^k}}_F^2 \right)r^3 + \norm{{\bf A}}_Fr-1 = 0$.\vspace{-0.2em}
		}}
		\label{alg:bpg-simple-1}
	}\ifpaper\else\\\fi


\subsection{Related Work}

\textbf{Alternating Minimization} is the go-to strategy for matrix factorization problems  due to  coupling between two subsets of variables \cite{G2014,AFWGF2019,YPC2018}. In the context of non-convex and non-smooth optimization, recently PALM \cite{BST2014} was proposed and convergence to stationary point was proved. An inertial variant, iPALM was proposed in \cite{PS2016}. However, such methods require a subset of variables to be fixed.  We remove such a restriction here and take the contrary view by proposing non-alternating schemes based on a powerful Bregman proximal minimization framework, which we review below. \ifpaper\else\\\fi

\textbf{Bregman Proximal Minimization}  extends upon the standard proximal minimization, where Bregman distances are used as proximity measures. Based on initial works in \cite{BDX2011,BBT2016,BSTV2018}, related inertial variants were proposed in \cite{MOPS2019,ZBMJC2019}. Related line-search methods were proposed in \cite{OFB2019} based on \cite{BLPP2016,BLPPR2017}. More related works in convex optimization include \cite{N2017,LFN2018,MPTD2019}. Recently, the symmetric non-negative matrix factorization  problem was solved with a non-alternating Bregman proximal minimization  scheme \cite{DBA2019} with the following kernel generating distance 
\[
	h({\bf U}) = \frac{\norm{\bf U}_F^4}{4} +  \frac{\norm{\bf U}_F^2}{2}\,.
\]
However for the following applications, such a $h$ is not suitable, unlike our Bregman distance. \ifpaper\else\\\fi

\textbf{Non-negative Matrix Factorization (NMF)} is a variant of the matrix factorization problem which requires the factors to have non-negative entries  \cite{GV2014,LS2001}. Some applications are hyperspectral unmixing, clustering and  others \cite{G2014,EGB2019}.  The non-negativity constraints pose new  challenges  \cite{LS2001} and only convergence to a stationary point \cite{G2014,HD2011} is guaranteed,  as NMF is NP-hard in general. Under certain restrictions, NMF can be solved exactly \cite{AGKM2012,M2016} but such methods are computationally infeasible. We give efficient algorithms for NMF and show the superior performance empirically. \ifpaper\else\\\fi

\textbf{Matrix Completion} is another variant of Matrix Factorization arising in recommender systems \cite{KRV2009} and bio-informatics \cite{LYLWLPLW2018,TYALS2018}, which is an active research topic due to the hard non-convex optimization problem \cite{CR2009,FZSH2017}. The state-of-the-art methods were proposed in \cite{JM2018,YK2018} and other recent methods include \cite{YMYCS2014}. Here, our algorithms are either faster or competitive. \ifpaper\else\\\fi

Our algorithms are also applicable to Graph Regularized NMF (GNMF) \cite{CHHH2011}, Sparse NMF \cite{BST2014}, Nuclear Norm Regularized problems \cite{CCS2010,HO2014},  Symmetric NMF via non-symmetric extension \cite{ZLLL2018}.

\section{Matrix Factorization Problem Setting and Algorithms}\label{sec:matrix-factor-problem}
		\textbf{Notation.} We refer to \cite{RW1998-B} for standard notation, unless specified otherwise.\ifpaper\else\\\fi

		 Formally, in a matrix factorization problem, given a matrix ${\bf A} \in \R^{M \times N}$, we want to obtain the factors ${\bf U} \in \R^{M \times K}$ and ${\bf Z}^{K \times N}$ such that ${\bf A} \approx {\bf U}{\bf Z}$, which is captured by the following non-convex problem 
		\begin{equation}\label{eq:mat-fac}
		\min_{{\bf U}\in \mathcal U, {\bf Z}\in \mathcal Z}\, \left\{ \Psi({\bf U},{\bf Z}) :=  \frac 12\norm{{\bf A} - {\bf U}{\bf Z}}^2_{F} +  {\mathcal R}_1({\bf U}) +  {\mathcal R}_2({\bf Z}) \right\}\,,
		\end{equation}
		where ${\mathcal R}_1({\bf U}) +  {\mathcal R}_2({\bf Z})$ is the separable regularization term, $\frac 12\norm{{\bf A} - {\bf U}{\bf Z}}^2_{F}$ is the data-fitting term, and $\mathcal U, \mathcal Z$ are the constraint sets for ${\bf U}$ and ${\bf Z}$ respectively. Here, ${\mathcal R}_1({\bf U}) $ and ${\mathcal R}_2({\bf Z}) $ can be potentially non-convex extended real valued functions and possibly non-smooth. In this paper, we propose to make use of BPG and its inertial variant CoCaIn BPG to solve \eqref{eq:mat-fac}. The introduction of these algorithms requires the following preliminary considerations.
	\ifpaper
		\begin{definition}(Kernel Generating Distance \cite{BSTV2018}) \label{D:KernelGen}
				Let $C$ be a nonempty, convex and open subset of $\real^{d}$. Associated with $C$, a function $h : \real^{d} \rightarrow \erl$ is called a \textit{kernel generating distance} if it satisfies: $\rm{(i)}$ $h$ is proper, lower semicontinuous and convex, with $\dom h \subset \barc$ and $\dom \partial h = C$, and $\rm{(ii)}$ $h$ is $C^{1}$ on $\idom h \equiv C$. We denote the class of kernel generating distances by $\mathcal{G}(C)$.
		 \end{definition}
	\else
		\begin{definition}(Kernel Generating Distance \cite{BSTV2018}) \label{D:KernelGen}
				Let $C$ be a nonempty, convex and open subset of $\real^{d}$. Associated with $C$, a function $h : \real^{d} \rightarrow \erl$ is called a \textit{kernel generating distance} if it satisfies: 
				\begin{itemize}
    			\item[$\rm{(i)}$] $h$ is proper, lower semicontinuous and convex, with $\dom h \subset \barc$ and $\dom \partial h = C$ .
     			\item[$\rm{(ii)}$] $h$ is $C^{1}$ on $\idom h \equiv C$.
 			\end{itemize}
 			We denote the class of kernel generating distances by $\mathcal{G}(C)$.
		 \end{definition}
	\fi
	 	For every $h \in \mathcal{G}(C)$, the associated Bregman distance is given by $D_{h} : \dom h \times \idom h \to \real_{+}$:
	    \begin{equation*}
	   		D_{h}\left(x , y\right) := h\left(x\right) - \left[h\left(y\right) + \act{\nabla h\left(y\right) , x - y}\right].
	   	\end{equation*}
	   	For examples, consider the following kernel generating distances:
	   	\[
	   	h_0(x) = \frac{1}{2}\norm{x}^2\,, \quad h_1(x) = \frac{1}{4}\norm{x}^4 + \frac{1}{2}\norm{x}^2\, \quad\text{and}\quad h_2(x) = \frac{1}{3}\norm{x}^3 + \frac{1}{2}\norm{x}^2\,.
	   	\]
	   	The Bregman distances associated with $h_0(x)$ is the  Euclidean distance. The Bregman distances associated with $h_1$ and $h_2$ appear in the context of non-convex quadratic inverse problems \cite{BSTV2018,MOPS2019} and non-convex cubic regularized problems \cite{MOPS2019} respectively. For a review on the recent literature, we refer the reader to \cite{T2018} and for early work on Bregman distances to \cite{Censor1981}. 

	   	These distance measures are key for development of algorithms for the following class of non-convex additive composite problems 
		\begin{equation}\label{eq:main-problem}
			(\PPP) \qquad \inf \left\{ \Psi \equiv f\left(x\right) + g\left(x\right) : \; x \in \barc \right\},
		\end{equation}
		which is assumed to satisfy the following standard assumption \cite{BSTV2018}.
	\ifpaper 
		\begin{assumption} \label{A:AssumptionA}
			$\rm{(i)}$ $h \in \mathcal{G}(C)$ with $\barc = \cdom h$. $\rm{(ii)}$ $f : \real^{d} \rightarrow \erl$ is a proper and lower semicontinuous function (potentially non-convex) with $\dom f \cap C \neq \emptyset$. $\rm{(iii)}$ $g : \real^{d} \rightarrow \erl$ is a proper and lower semicontinuous function (potentially non-convex) with $\dom{h} \subset \dom{g}$, which is continuously differentiable on $C$. $\rm{(iv)}$ $v(\PPP) := \inf \left\{ \Psi\left(x\right) : \; x \in \barc \right\} > -\infty$.
		\end{assumption}
	\else 
		\begin{assumption} \label{A:AssumptionA}
			\begin{itemize}
				\item[$\rm{(i)}$] $h \in \mathcal{G}(C)$ with $\barc = \cdom h$. 
				\item[$\rm{(ii)}$] $f : \real^{d} \rightarrow \erl$ is a proper and lower semicontinuous function (potentially non-convex) with $\dom f \cap C \neq \emptyset$.
				\item[$\rm{(iii)}$] $g : \real^{d} \rightarrow \erl$ is a proper and lower semicontinuous function (potentially non-convex) with $\dom{h} \subset \dom{g}$, which is continuously differentiable on $C$. 
				\item[$\rm{(iv)}$]  $v(\PPP) := \inf \left\{ \Psi\left(x\right) : \; x \in \barc \right\} > -\infty$.
			\end{itemize}
		\end{assumption}
	\fi
	\paragraph*{Matrix Factorization Example.} A special case of \eqref{eq:main-problem} is the following problem,
	\begin{equation}\label{eq:mat-fac-1}
	\inf\, \left\{ \Psi({\bf U},{\bf Z}) :=  f_1({\bf U}) + f_2({\bf Z}) +  g({\bf U}, {\bf Z}): ({\bf U},{\bf Z}) \in \barc\right\}\,.
	\end{equation}
	We denote $f({\bf U}, {\bf Z}) = f_1({\bf U}) + f_2({\bf Z})$.	Many practical matrix factorization problems can be cast into the form of \eqref{eq:mat-fac}. The choice of $f$ and $g$ is dependent on the problem, for which we provide some examples in Section~\ref{sec:exps}. Here $f_1, f_2$ satisfy the assumptions of $f$ with dimensions chosen accordingly. Moreover by definition, $f$ is separable in ${\bf U}$ and ${\bf Z}$, which we assume only for practical reasons. Also, the choice of $f,g$ may not be unique. For example, in \eqref{eq:mat-fac} when  $\mathcal R_1({\bf U}) = \frac{\lambda_0}{2}\norm{{\bf U}}_F^2 \text{ and } \mathcal R_2({\bf Z}) = \frac{\lambda_0}{2}\norm{{\bf Z}}_F^2$
	the choice of $f$  as in \eqref{eq:mat-fac-1} can be $\mathcal R_1 + \mathcal R_2$ and $g=\frac 12\norm{{\bf A} - {\bf U}{\bf Z}}^2_{F}$. However, the other choice is to set $g=\Psi$ and $f:=0$.
\subsection{BPG-MF: Bregman Proximal Gradient for Matrix Factorization}\label{ssec:bpg-mf-main}
	We require the notion of Bregman Proximal Gradient Mapping \cite[Section 3.1]{BSTV2018} given by
		\begin{align}
			T_{\lambda}\left(x\right) 
			& =  \argmin \left\{ f\left(u\right) + \act{\nabla g\left(x\right) , u - x} + \frac{1}{\lambda} D_{h}\left(u , x\right) : \, u \in \barc \right\}\,. \label{D:OperT}
		\end{align}
	Then, the update step of Bregman Proximal Gradient (BPG) \cite{BSTV2018} for solving \eqref{eq:main-problem} is $x^{k+1} \in T_\lambda (x^k)$, for some $\lambda>0$ and $h\in \mathcal{G}(C)$. Convergence of BPG relies on a generalized notion of Lipschitz continuity, the so-called $L$-smad property (Defintion~\ref{D:l-smad}).\ifpaper\else\\\fi

	\textbf{Beyond Lipschitz continuity.} BPG extends upon the popular proximal gradient methods, for which convergence relies on Lipschitz continuity of the smooth part of the objective in \eqref{eq:main-problem}. However, such a notion of Lipschitz continuity is restrictive for many practical applications such as Poisson linear inverse problems \cite{BBT2016}, quadratic inverse problems \cite{BSTV2018,MOPS2019}, cubic regularized problems \cite{MOPS2019} and robust denoising problems with non-convex total variation regularization \cite{MOPS2019}.  The extensions for generalized notions of Lipschitz continuity of gradients is an active area of research \cite{BDX2011, BBT2016,LFN2018,BSTV2018}. We consider the following from \cite{BSTV2018}.
	\begin{definition}[$L$-smad property]\label{D:l-smad}
	The function $g$ is said to be $L$-smooth adaptable ($L$-smad)  on $C$ with respect to $h$, if and only if  $Lh -g$ and $Lh +g$ are convex on $C$.
	\end{definition}

	When $h(x) =\frac12 \norm{x}^2$, $L$-smad property is implied by Lipschitz continuous gradient. Consider the function $f(x)=x^4$, it is $L$-smad with respect to $h(x)=x^4$ and $L\geq 1$, however $\nabla f$ is not Lipschitz continuous. \ifpaper\else\medskip\fi

	Now, we are ready to present the BPG algorithm for Matrix Factorization.\ifpaper\else\\\fi

	{\centering
		\fcolorbox{black}{Orange!10}{\parbox{0.98\textwidth}{\textbf{BPG-MF: BPG for Matrix Factorization.} \\
				{\textbf{Input.}} Choose $h \in \mathcal{G}(C)$ with $C \equiv \idom h$ such that $g$ satisfies $L$-smad with respect to $h$ on $C$.\\
				{\textbf{Initialization.}} $({\bf U}^{{\bf 1}}, {\bf Z}^{{\bf 1}}) \in  \idom h$ and let $\lambda >0$. \\
				{\textbf{General Step.}} For $k = 1 , 2 , \ldots$, compute	
				\[
					{\bf P^k} = \lambda\nabla_{\bf U} g\left({\bf U}^{{\bf k}},{\bf Z}^{{\bf k}} \right) - \nabla_{\bf U} h({\bf U}^{{\bf k}},{\bf Z}^{{\bf k}})\,,\quad {\bf Q^k} =  \lambda\nabla_{\bf Z} g\left({\bf U}^{{\bf k}},{\bf Z}^{{\bf k}} \right) - \nabla_{\bf Z} h({\bf U}^{{\bf k}},{\bf Z}^{{\bf k}})\,,\vspace{-1em}
				\]
				\begin{equation}
				({\bf U}^{{\bf k+1}}, {\bf Z}^{{\bf k+1}}) \in \underset{({\bf U},{\bf Z}) \in \barc}{\argmin} \left\{  \lambda f({\bf U},{\bf Z})+ \act{ {\bf P^k}, {\bf U}}  + \act{ {\bf Q^k},  {\bf Z}} + h({\bf U},{\bf Z})  \right\}\,.
				\label{alg:bpg-mf-2}
				\end{equation}\vspace{-0.8em}
		}}
	}\ifpaper\else\\\fi

	Under Assumption~\ref{A:AssumptionA} and the following one (mostly satisfied in practice), BPG is well-defined \cite{BSTV2018}.
	\begin{assumption} \label{A:AssumptionB} The range of $T_\lambda$ lies in $C$  and, for all $\lambda>0$, the function $h+\lambda f$ is supercoercive.
	\end{assumption}
	The update step for BPG-MF is easy to derive from BPG, however convergence of BPG also relies on the ``right'' choice of kernel generating distance $h$ and the $L$-smad condition. Finding $h$ such that $L$-smad holds (also see Section~\ref{ssec:new-bregman-distance}) and that the update step can be given in closed form (also see Section~\ref{ssec:closed-form}) is our main contribution and allows us to invoke the convergence results from \cite{BSTV2018}. The convergence result states that the whole sequence of iterates generated by BPG-MF converges to a stationary point, precisely given in Theorem~\ref{thm:thm-main}. The result depends on the non-smooth KL-property (see \cite{BDLS2007,AB2009,BST2014}) which is a mild requirement and is satisfied by most practical objectives.  We provide below the convergence result in \cite[Theorem 4.1]{BSTV2018} adapted to BPG-MF.
	\begin{theorem}[Global Convergence of BPG-MF]\label{thm:thm-main} Let  Assumptions~\ref{A:AssumptionA} and \ref{A:AssumptionB} hold and let $g$ be $L$-smad with respect to $h$, where $h$ is assumed to be $\sigma$-strongly convex with full domain. Assume $\nabla g, \nabla h$ to be Lipschitz continuous on any bounded subset.  Let  $\left\{({\bf U^{k+1}}, {\bf Z}^{{\bf k+1}})\right\}_{k \in \nn}$ be a bounded sequence generated by BPG-MF with $0<\lambda L<1$, and suppose $\Psi$ satisfies the KL property, then, such a sequence has finite length, and converges to a critical point. 
	\end{theorem}
	


\subsection{New Bregman Distance for Matrix Factorization}\label{ssec:new-bregman-distance}
	We prove the $L$-smad property for the term $g({\bf U,\bf Z}) = \frac 12\norm{{\bf A} - {\bf U}{\bf Z}}^2_{F}$ of the matrix factorization problem in \eqref{eq:mat-fac}. The kernel generating distance is a linear combination of 
	\begin{equation}\label{eq:breg-funcs}
	h_1({\bf U},{\bf Z}) := \left(\frac{\norm{\bf U}_F^2 +\norm{\bf Z}_F^2 }{2} \right)^2\quad\text{and}\quad
	h_2({\bf U},{\bf Z}) :=  \frac{\norm{\bf U}_F^2 +\norm{\bf Z}_F^2 }{2} \,,
	\end{equation}
	and it is designed to also allow for closed form updates (see Section~\ref{ssec:closed-form}). 
	\begin{proposition}\label{prop:l-smad-1}
	Let $g,h_1,h_2$ be as defined above. Then, for $L\geq 1$, the function $g$ satisfies the $L$-smad property with respect to the following kernel generating distance
	\begin{equation}\label{eq:main-lsmad-h}
	 h_a({\bf U},{\bf Z}) =  3h_1({\bf U},{\bf Z}) + \norm{{\bf A}}_F h_2({\bf U},{\bf Z}) \,.
	\end{equation}
	\end{proposition}
	The proof is given in Section~\ref{ssec:main-lsmad-proof-sec} in the \ifpaper supplementary material\else appendix\fi.  The Bregman distances considered in previous works \cite{MOPS2019,BSTV2018} are separable and not applicable for matrix factorization problems. The inherent coupling between two subsets of variables ${\bf U},{\bf Z}$ is the main source of non-convexity in the objective $g$. The kernel generating distance (in particular $h_1$ in \eqref{eq:main-lsmad-h}) contains the interaction/coupling terms between ${\bf U}$ and ${\bf Z}$ which makes it amenable for matrix factorization problems.  

\subsection{CoCaIn BPG-MF: An Adaptive Inertial Bregman Proximal Gradient Method}\label{ssec:cocain-bpg-mf}
	The goal of this section is to introduce an inertial variant of BPG-MF, called CoCaIn BPG-MF. The effective step-size choice for BPG-MF can be restrictive due to  large constant like $\norm{{\bf A}}_F$ (see \eqref{eq:main-lsmad-h}), for which we present a practical example in the numerical experiments. In order to allow for larger step-sizes, one needs to adapt it locally, which is often done via a backtracking procedure. CoCaIn BPG-MF combines inertial steps with a novel backtracking procedure proposed in \cite{MOPS2019}.\ifpaper\else\medskip\fi

	Inertial algorithms often lead to better convergence \cite{OCBP2014,PS2016,MOPS2019}. The classical Nesterov Accelerated Gradient (NAG) method \cite{N1983} and the popular Fast Iterative Shrinkage-Thresholding Algorithm  (FISTA) \cite{BT2009} employ an extrapolation based inertial strategy. However, the extrapolation is governed by a parameter which is typically scheduled to follow certain iteration-dependent scheme \cite{N1983,HRX2018}and is restricted to the convex setting. Recently with Convex--Concave Inertial Bregman Proximal Gradient (CoCaIn BPG) \cite{MOPS2019}, it was shown that one could leverage the upper bound (convexity of $Lh -g$) and lower bound (convexity of $Lh +g$) to incorporate inertia in an adaptive manner. \ifpaper\else\medskip\fi 


	We recall now the update steps of CoCaIn BPG \cite{MOPS2019} to solve \eqref{eq:main-problem}.  Let $h\in {\mathcal G}(C)$, $\lambda>0$, and $x^0=x^1\in\R^d$ be an initalization, then in each iteration the extrapolated point $y^k = x^k + \gamma_k (x^k - x^{k-1})$ is computed followed by a BPG like update (at $y^k$) given by $x^{k+1} \in T_{\tau_k}(y^k)$, where $\gamma_k$ is the inertial parameter and $\tau_k$ is the step-size parameter. Similar conditions to BPG are required for the convergence to a stationary point. We use CoCaIn BPG for Matrix Factorization (CoCaIn BPG-MF) and our proposed novel kernel generating distance $h$ from \eqref{eq:main-lsmad-h} makes the convergence results of  \cite{MOPS2019} applicable. Along with Assumption~\ref{A:AssumptionB}, we require the following assumption.
	\begin{assumption} \label{A:Assumption0}
	\ifpaper$\rm{(i)}$\else\begin{itemize}\item[$\rm{(i)}$]\fi\,   There exists $\alpha \in \R$ such that $f({\bf U}, {\bf Z}) - \frac{\alpha}{2}\left(\norm{{\bf U}}_F^2 + \norm{{\bf Z}}_F^2 \right)$ is convex.\ifpaper\\\else\fi
	 \ifpaper$\rm{(ii)}$\else\item[$\rm{(ii)}$]\fi \,  The kernel generating distance $h$ is $\sigma$-strongly convex on $\R^{M \times K}\times \R^{K \times N}$.  
	 \ifpaper\else\end{itemize}\fi
	\end{assumption}

	The Assumption~\ref{A:Assumption0}$\rm{(i)}$ refers to notion of semi-convexity of the function $f$, (see \cite{Ochs18,MOPS2019}) and seems to be closely connected to the inertial feature of an algorithm. For notational brevity, we use $D_{g}\left(x , y\right) := g\left(x\right) - \left[g\left(y\right) + \act{\nabla g\left(y\right) , x - y}\right]$ which may also be negative if g is not a kernel generating distance. Moreover, we use $D_h(({\bf X}_1,{\bf Y}_1),({\bf X}_2,{\bf Y}_2))$ as $D_h({\bf X}_1,{\bf Y}_1,{\bf X}_2,{\bf Y}_2)$. We provide CoCaIn BPG-MF below.\ifpaper\else\\\fi

	{\centering
		\fcolorbox{black}{Orange!10}{\parbox{0.98\textwidth}{\textbf{CoCaIn BPG-MF: Convex--Concave Inertial BPG for Matrix Factorization.} \\
				{\textbf{Input.}} Choose $\delta,\varepsilon>0$ with $1>\delta>\epsilon$, $h \in \mathcal{G}(C)$ with $C \equiv \idom h$, g is $L$-smad on $C$ w.r.t $h$.\\
				{\textbf{Initialization.}} $ ({\bf U}^{{\bf 1}}, {\bf Z}^{{\bf 1}})=({\bf U}^{{\bf 0}}, {\bf Z}^{{\bf 0}}) \in  \idom h \cap \dom f$, ${\bar L}_0 > \frac{-\alpha}{(1-\delta)\sigma}$ and $\tau_{0} \leq {\bar L}_{0}^{-1}$. \\
				{\textbf{General Step.}} For $k = 1 , 2 , \ldots$, compute extrapolated points
				\begin{align}\label{eq:cocain-0}
				Y_{{\bf U}}^{\bf k} &={\bf U}^{k} + \gamma_k \left({\bf U^{k}} - {\bf U^{k-1}}\right)\quad\text{and}\quad Y_{{\bf Z}}^{\bf k} ={\bf Z}^{k} + \gamma_k \left({\bf Z^{k}} - {\bf Z^{k-1}}\right)\,,
				\end{align}
				where $\gamma_k \geq 0$ such that
				\begin{align}\label{eq:cocain-1}
				(\delta -\varepsilon)D_{h}\left({\bf U}^{{\bf k-1}},{\bf Z}^{{\bf k-1}},{\bf U}^{{\bf k}},{\bf Z}^{{\bf k}}\right) \geq (1 + {\underline L}_{k}\tau_{k-1})D_{h}\left({\bf U}^{{\bf k}},{\bf Z}^{{\bf k}},Y_{\bf U}^{{\bf k}},Y_{\bf Z}^{{\bf k}}\right)\,,
				\end{align}
				where ${\underline L}_k$ satisfies 
				\begin{equation}\label{eq:cocain-2}
					D_{g}\left({\bf U}^{{\bf k}},{\bf Z}^{{\bf k}},Y_{\bf U}^{{\bf k}},Y_{\bf Z}^{{\bf k}}\right) \geq -{\underline L}_k D_h\left({\bf U}^{{\bf k}},{\bf Z}^{{\bf k}},Y_{\bf U}^{{\bf k}},Y_{\bf Z}^{{\bf k}}\right)\,.
				\end{equation}
				Choose ${\bar L}_k \geq {\bar L}_{k-1}$, and set $\tau_k \leq \min\{ \tau_{k-1}, {\bar L}_{k}^{-1} \}$. Now, compute 
				\[
				{\bf P^k} = \tau_k\nabla_{\bf U} g\left(Y_{\bf U}^{{\bf k}},Y_{\bf Z}^{{\bf k}} \right) - \nabla_{\bf U} h(Y_{\bf U}^{{\bf k}},Y_{\bf Z}^{{\bf k}})\,, \quad {\bf Q^k} =  \tau_k\nabla_{\bf Z} g\left(Y_{\bf U}^{{\bf k}},Y_{\bf Z}^{{\bf k}} \right) - \nabla_{\bf Z} h(Y_{\bf U}^{{\bf k}},Y_{\bf Z}^{{\bf k}})\,,\vspace{-1em}
				\]
				\begin{equation}
				({\bf U}^{{\bf k+1}}, {\bf Z}^{{\bf k+1}}) \in \underset{({\bf U},{\bf Z}) \in \barc}{\argmin} \left\{  \tau_k f({\bf U},{\bf Z})+ \act{ {\bf P^k}, {\bf U}}  + \act{ {\bf Q^k}, {\bf Z}} + h({\bf U},{\bf Z})  \right\}\,,\label{eq:cocain-3}
				\end{equation}
				such that ${\bar L}_k$ satisfies
				\begin{equation}\label{eq:cocain-4}
					D_{g}\left({\bf U}^{{\bf k+1}},{\bf Z}^{{\bf k+1}},Y_{\bf U}^{{\bf k}},Y_{\bf Z}^{{\bf k}}\right) \leq {\bar L}_k D_{h}\left({\bf U}^{{\bf k+1}},{\bf Z}^{{\bf k+1}},Y_{\bf U}^{{\bf k}},Y_{\bf Z}^{{\bf k}}\right)\,.
				\end{equation}\vspace{-1.5em}
		}}
	}\ifpaper\else\\\fi

	The extrapolation step is performed in \eqref{eq:cocain-0}, which is similar to NAG/FISTA. However, the inertia cannot be arbitrary and the analysis from \cite{MOPS2019} requires step \eqref{eq:cocain-1} which is governed by the convexity of lower bound, ${\underline L}_kh+g$, however only locally as in \eqref{eq:cocain-2}.  The update step \eqref{eq:cocain-3} is similar to BPG-MF, however the step-size is controlled via the convexity of upper bound ${\bar L}_kh-g$, but only locally as in \eqref{eq:cocain-4}. The local adaptation of the steps  \eqref{eq:cocain-2} and \eqref{eq:cocain-4} is performed via backtracking.  Since, ${\bar L}_k$ can be potentially very small compared to $L$, hence potentially large steps can be taken. There is no restriction on ${\underline L}_k$ in each iteration, and smaller  ${\underline L}_k$ can result in high value for the inertial parameter $\gamma_k$. Thus the algorithm in essence aims to detect "local convexity" of the objective.  The update steps of CoCaIn BPG-MF can be executed sequentially without any nested loops for the backtracking. One can always find the inertial parameter $\gamma_k$ in \eqref{eq:cocain-1} due to \cite[Lemma 4.1]{MOPS2019}. For certain cases, \eqref{eq:cocain-1} yields an explicit condition on $\gamma_k$. For example, for  $h({\bf U},{\bf Z}) = \frac{1}{2}(\norm{\bf U}_F^2 + \norm{\bf Z}_F^2)$, we have $0\leq\gamma_k \leq \sqrt{\frac{\delta-\varepsilon}{1+\tau_{k-1}{\underline L}_k}}$. We now provide below the convergence result from \cite[Theorem 5.2]{MOPS2019} adapted to CoCaIn BPG-MF.\ifpaper\else\\\fi
	\begin{theorem}[Global Convergence of CoCaIn BPG-MF]\label{thm:thm-main}  Let  Assumptions~\ref{A:AssumptionA}, \ref{A:AssumptionB} and \ref{A:Assumption0} hold, let $g$ be $L$-smad with respect to $h$ with full domain. Assume $\nabla g, \nabla h$ to be Lipschitz continuous on any bounded subset. Let  $\left\{({\bf U^{k+1}}, {\bf Z}^{{\bf k+1}})\right\}_{k \in \nn}$ be a bounded sequence generated by  CoCaIn BPG-MF, and suppose $f,g$ satisfy the KL property, then, such a sequence has finite length, and converges to a critical point.
	\end{theorem}

\subsection{Closed Form Solutions for Update Steps of BPG-MF and CoCaIn BPG-MF}\label{ssec:closed-form}
	Our second significant contribution is to make BPG-MF and CoCaIn BPG-MF an efficient choice for solving Matrix Factorization, namely closed form expressions for the main update steps \eqref{alg:bpg-mf-2}, \eqref{eq:cocain-3}. For the derivation, we refer to the \ifpaper supplementary material\else appendix\fi, here we just state our results.\ifpaper\else\\\fi

	 For the L2-regularized problem $$ g({\bf U},{\bf Z}) = \frac 12\norm{{\bf A} - {\bf U}{\bf Z}}^2_{F},\quad f({\bf U},{\bf Z})  =    \frac{\lambda_0}{2}\left(\norm{{\bf U} }_F^2 + \norm{{\bf Z} }_F^2 \right),\quad h = h_a$$  with $c_1 = 3, c_2=\norm{{\bf A}}_F$ and $0<\lambda <1$ the BPG-MF  updates are:\ifpaper\else\\\fi

	{\centering
		\fcolorbox{black}{Orange!10}{\parbox{0.98\textwidth}{ 
		${\bf U}^{{\bf k+1}} = -r {\bf P^k}$\,, ${\bf Z}^{{\bf k+1}} = -r{\bf Q^k}$ with $r\geq 0$\,, $c_1\big(\norm{-{\bf P^k}}_F^2 +  \norm{-{\bf Q^k}}_F^2  \big)r^3 + (c_2+\lambda_0)r-1 = 0\,.$
		}}
	}\ifpaper\else\\\fi

	For NMF with additional non-negativity constraints, we replace $-{\bf P^k}$ and $-{\bf Q^k}$ by $\Pi_{+}(-{\bf P^k})$ and $\Pi_{+}(-{\bf Q^k})$ respectively where $\Pi_{+}(.) = \max\{0,.\}$ and $\max$ is applied element wise. 

	 Now consider the following L1-Regularized problem
	\begin{equation}
	g({\bf U},{\bf Z}) = \frac 12\norm{{\bf A} - {\bf U}{\bf Z}}^2_{F},\quad f({\bf U},{\bf Z})=   \lambda_1\left(\norm{{\bf U} }_1 + \norm{{\bf Z} }_1 \right),\quad h = h_a\,.
	\end{equation}
	The {soft-thresholding operator} is defined for any $y \in \R^d$ by $\SSS_{\theta}\left(y\right) = \max\left\{ \left|y\right| - \theta , 0 \right\}\sgn\left(y\right)$ where $\theta>0$. Set $c_1 = 3, c_2=\norm{{\bf A}}_F$ and $0<\lambda <1$ the BPG-MF updates with the above given $g,f, h$ are:\ifpaper\else\\\fi

	{\centering
		\fcolorbox{black}{Orange!10}{\parbox{0.98\textwidth}{ 
		${\bf U}^{{\bf k+1}} = r \SSS_{\lambda_1\lambda}(- {\bf P^k})$, ${\bf Z}^{{\bf k+1}} = r\SSS_{\lambda_1\lambda}(-{\bf Q^k})$ with $r\geq 0$ and  $$c_1\left(\norm{\SSS_{\lambda_1\lambda}(-{\bf P^k})}_F^2 +  \norm{\SSS_{\lambda_1\lambda}(-{\bf Q^k})}_F^2  \right)r^3 + c_2 r-1 = 0\,.$$\vspace{-1.4em}
		}}
	}\ifpaper\else\\\fi

	We denote a vector of ones as ${\bf e}_D \in \R^D$.
	For additional non-negativity constraints  we need to replace  $\SSS_{\lambda_1\lambda}(-{\bf P^k})$ with $\Pi_{+}(-\left({\bf P^k}+\lambda_1\lambda{\bf e}_{M}{\bf e}_{K}^T\right))$ and $\SSS_{\lambda_1\lambda}\left(-{\bf Q^k}\right)$ to $\Pi_{+}(-\left({{\bf Q^k} +\lambda_1\lambda{\bf e}_{K}{\bf e}_{N}^T}\right))$. Excluding the gradient computation, the computational complexity of our updates is $O(MK + NK)$ only, thanks to linear operations. PALM and iPALM additionally involve calculating Lipschitz constants with at most $O(K^2\max\{M,N\}^2)$ computations. Examples like Graph Regularized NMF (GNMF) \cite{CHHH2011}, Sparse NMF \cite{BST2014}, Matrix Completion \cite{KRV2009}, Nuclear Norm Regularization \cite{CCS2010,HO2014}, Symmetric NMF \cite{ZLLL2018} and proofs are given in the \ifpaper supplementary material\else appendix\fi. 
	
\section{Experiments}\label{sec:exps}
	In this section, we show experiments for \eqref{eq:mat-fac}. Denote the regularization settings, \textbf{R1:} with  ${\mathcal R}_1\equiv {\mathcal R}_2\equiv 0$, \textbf{R2:}  with L2 regularization ${\mathcal R}_1({\bf U}) = \frac{\lambda_0}{2}\norm{{\bf U}}_F^2$ and  ${\mathcal R}_2({\bf Z})=\frac{\lambda_0}{2}\norm{{\bf Z}}_F^2$ for some $\lambda_0>0$, \textbf{R3:} with L1 Regularization ${\mathcal R}_1({\bf U}) = \lambda_0\norm{{\bf U}}_1$ and  ${\mathcal R}_2({\bf Z})=\lambda_0\norm{{\bf Z}}_1$ for some $\lambda_0>0$. \ifpaper\else\\\fi

	\textbf{Algorithms.} We compare our first order optimization algorithms, BPG-MF and CoCaIn BPG-MF,  and recent state-of-the-art optimization methods iPALM \cite{PS2016} and PALM \cite{BST2014}. We focus on  algorithms that guarantee convergence to a stationary point. We also use BPG-MF-WB, where WB stands for "with backtracking", which is equivalent to CoCaIn BPG-MF with $\gamma_k \equiv 0$.  We use two settings for iPALM, where  all the extrapolation parameters  are set to a single value $\beta$ set to $0.2$ and $0.4$. PALM is equivalent to iPALM if $\beta=0$. We use the same initialization for all methods.\ifpaper\else\\\fi

	\textbf{Simple Matrix Factorization.} We set $\mathcal U = \R^{M \times K} \text{ and } \mathcal Z = \R^{K \times N}$. We use a randomly generated synthetic data matrix with $A \in \R^{200 \times 200}$ and report performance in terms of function value for three regularization settings, \textbf{R1}, \textbf{R2} and \textbf{R3} with $K=5$. Note that this enforces a factorization into at most rank 5 matrizes ${\bf U}$ and ${\bf Z}$, which yields an additional implicit regularization. For \textbf{R2} and \textbf{R3} we use $\lambda_0 = 0.1$. CoCaIn BPG-MF is superior\footnote{Note that in the $y$-axis label $v(\PPP)$ is the least objective value attained by any of the methods.} as shown in  Figure~\ref{fig:synthetic} .\ifpaper\else\\\fi

	\textbf{Statistical Evaluation.} We also provide the statistical evaluation of all the algorithms in  Figure~\ref{fig:statistical-evaluation}, for the above problem.  The optimization variables  are sampled from [0,0.1] and 50 random seeds are considered. CoCaIn BPG  outperforms other methods, however PALM methods are also very competitive. In L1 regularization setting, the performance of CoCaIn BPG is the best. In all settings, BPG-MF performance is worst due to a constant step size, which might change in settings where local adapation with backtracking line search is computationally not feasible. \ifpaper\else\\\fi

	\textbf{Matrix Completion.} In recommender systems \cite{KRV2009} given a matrix $A$ with entries at few index pairs in set $\Omega$, the goal is to obtain factors ${\bf U}$ and ${\bf Z}$ that generalize via following optimization problem
	\begin{equation}\label{eq:mat-com-exp-2}
	\min_{{\bf U}\in \R^{M \times K}, {\bf Z}\in \R^{K \times N}}\, \left\{ \Psi({\bf U},{\bf Z}) :=  \frac 12\norm{P_{\Omega}\left({\bf A} - {\bf U}{\bf Z}\right)}^2_{F} + \frac{\lambda_0}{2}\left(\norm{{\bf U} }_F^2 + \norm{{\bf Z} }_F^2 \right)  \right\}\,,
	\end{equation}
	where $P_{\Omega}$ preserves the given matrix entries and sets others to zero. We use 80\% data of MovieLens-100K, MovieLens-1M and MovieLens-10M \cite{HK2016} datasets and use other 20\% to test (details in the \ifpaper supplementary material\else appendix\fi). CoCaIn BPG-MF is faster than all methods as given in Figure~\ref{fig:movielens}.\ifpaper\else\\\fi
	\begin{figure}[hbt!]
		\centering
		\begin{subfigure}{0.32\textwidth}
			\centering
			\includegraphics[width=1\textwidth]{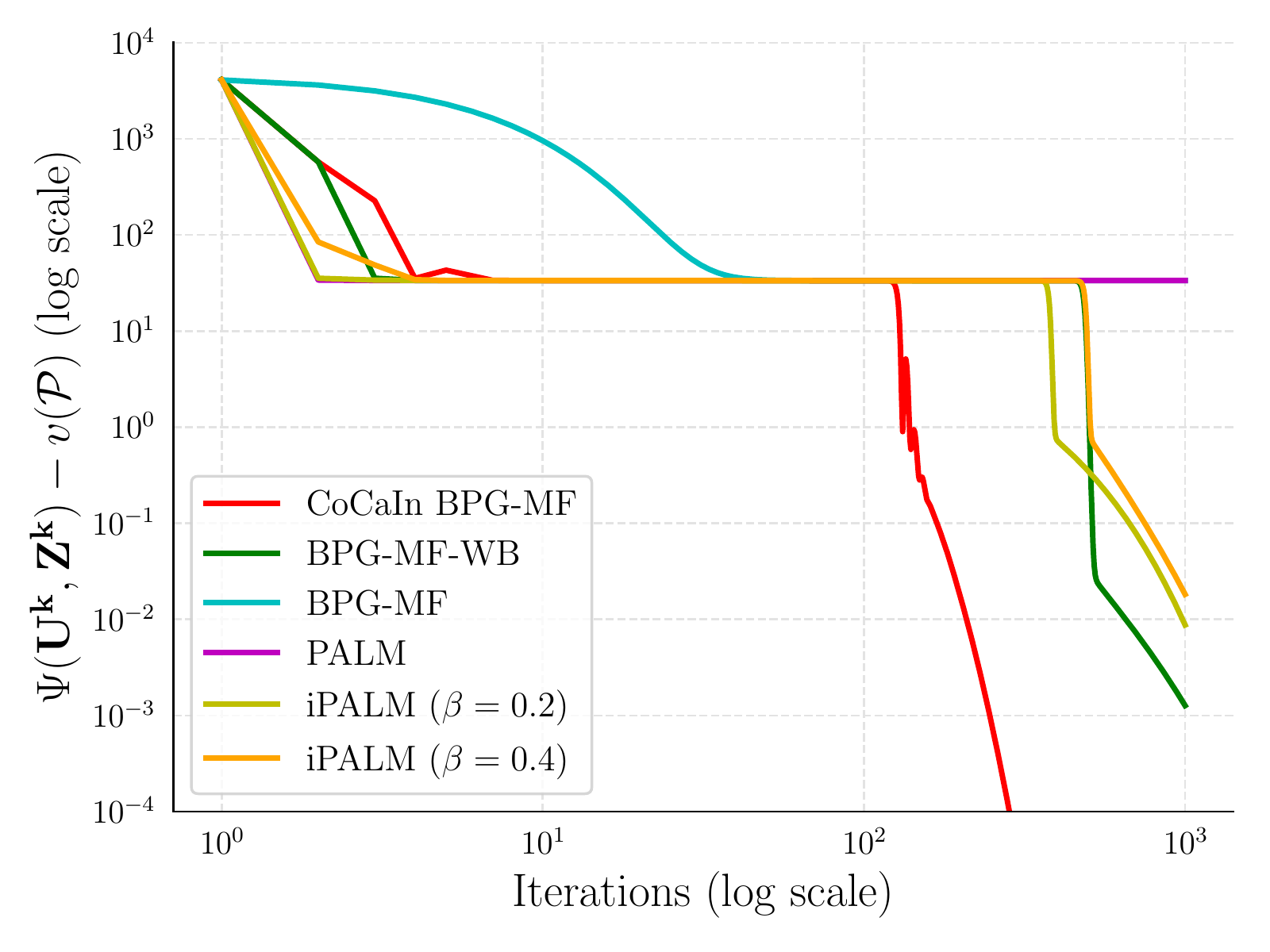}
			\caption{No Regularization}
		\end{subfigure}
		\begin{subfigure}{0.32\textwidth}
			\centering
			\includegraphics[width=1\textwidth]{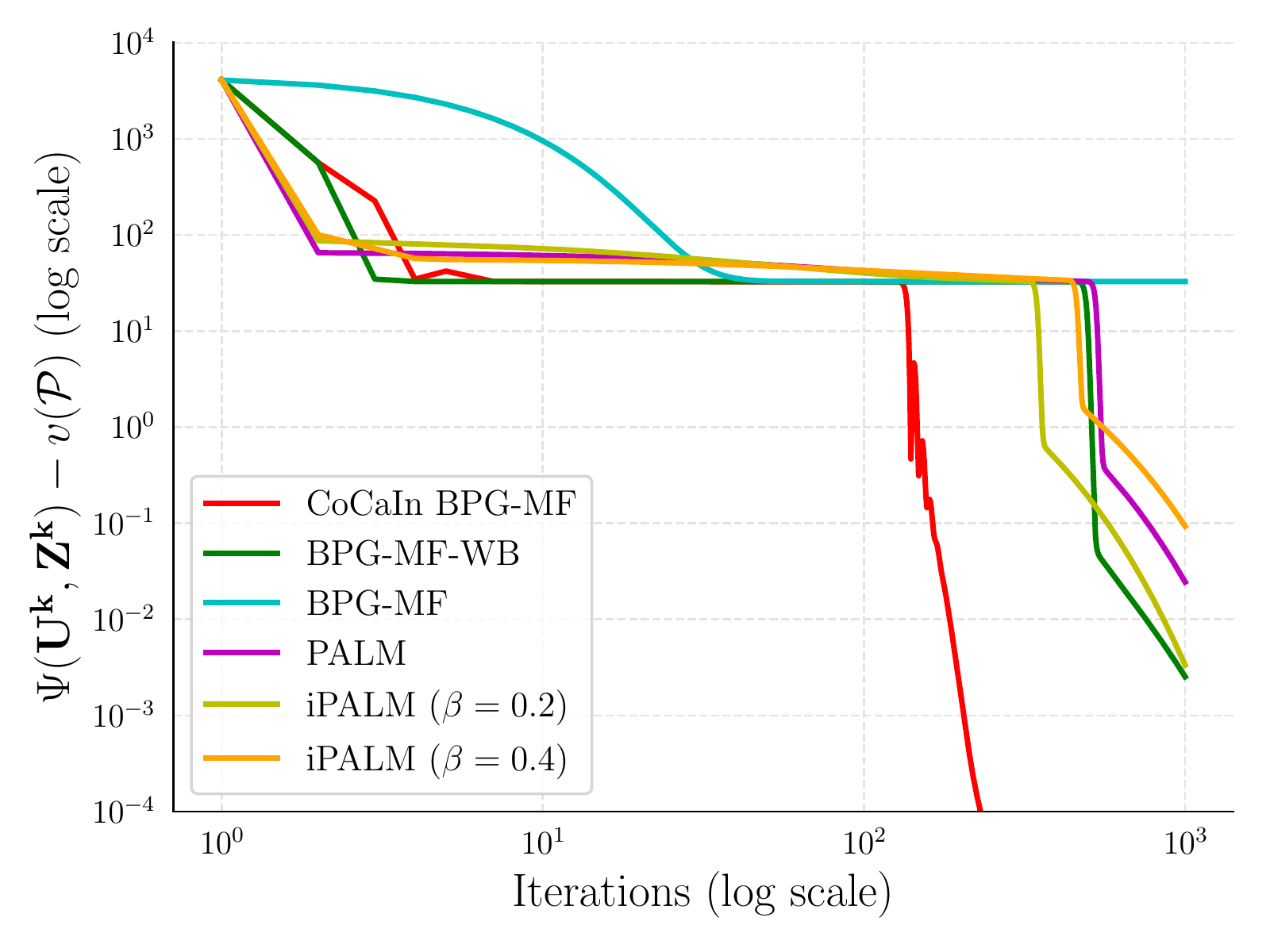}
			\caption{L2-Regularization}
		\end{subfigure}
		\begin{subfigure}{0.32\textwidth}
			\centering
			\includegraphics[width=1\textwidth]{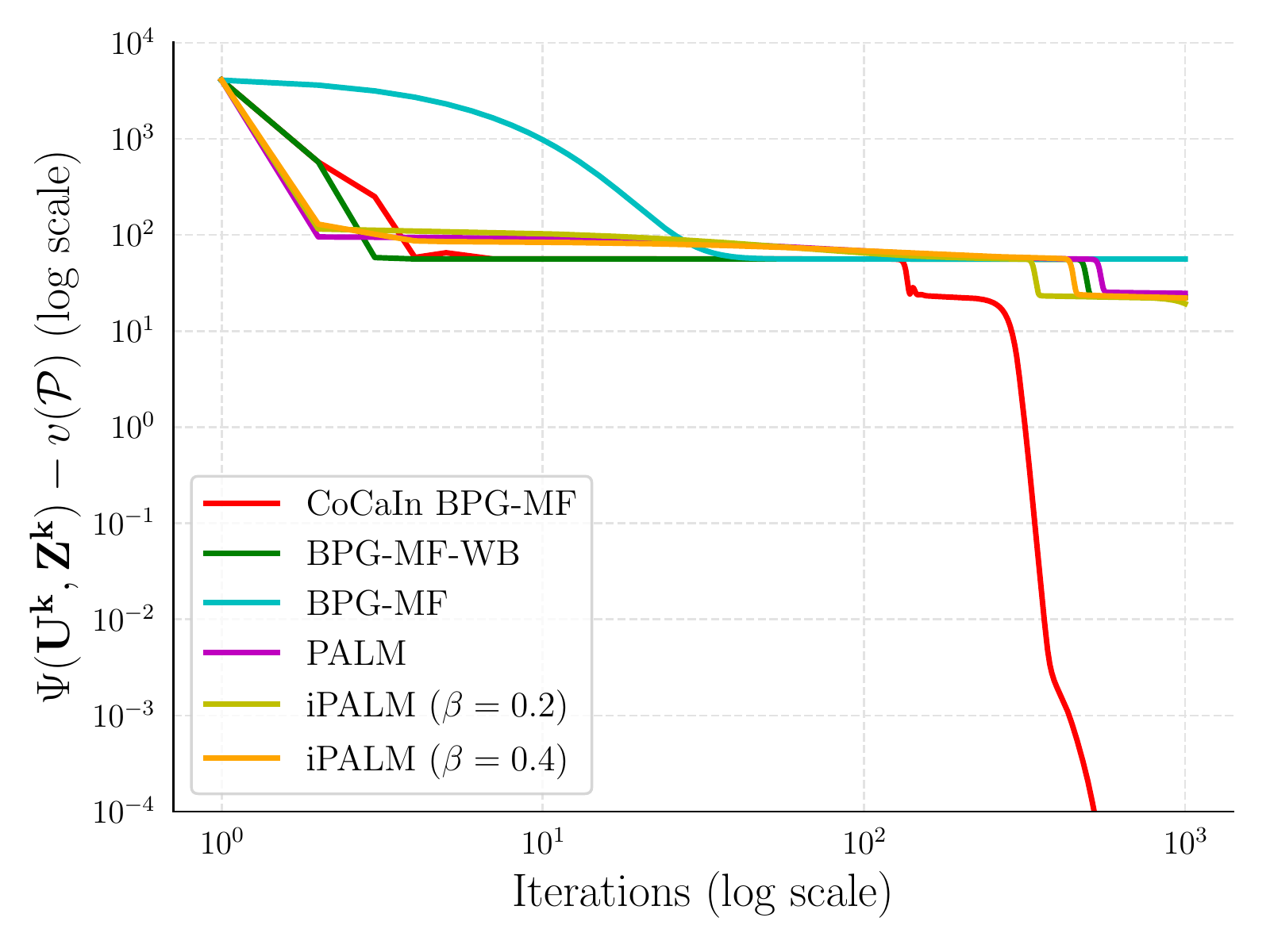}
			\caption{L1-Regularization}
		\end{subfigure}
		\caption{\textbf{Simple Matrix Factorization on Synthetic Dataset.}}
		\label{fig:synthetic}
	\end{figure}

	\begin{figure}[hbt!]
		\centering
		\begin{subfigure}{0.32\textwidth}
			\centering
			\includegraphics[width=1\textwidth]{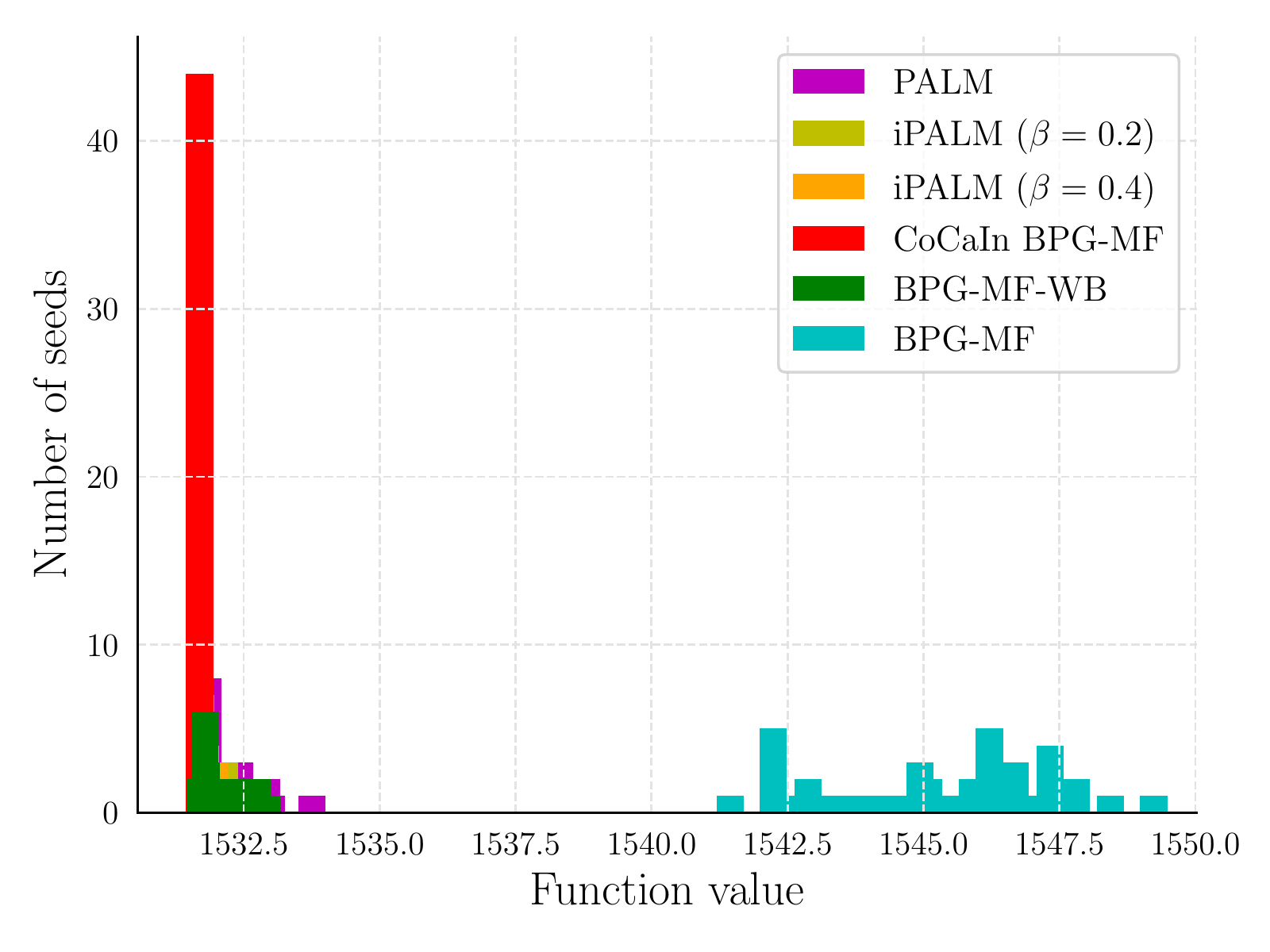}
			\caption{No Regularization}
		\end{subfigure}
		\begin{subfigure}{0.32\textwidth}
			\centering
			\includegraphics[width=1\textwidth]{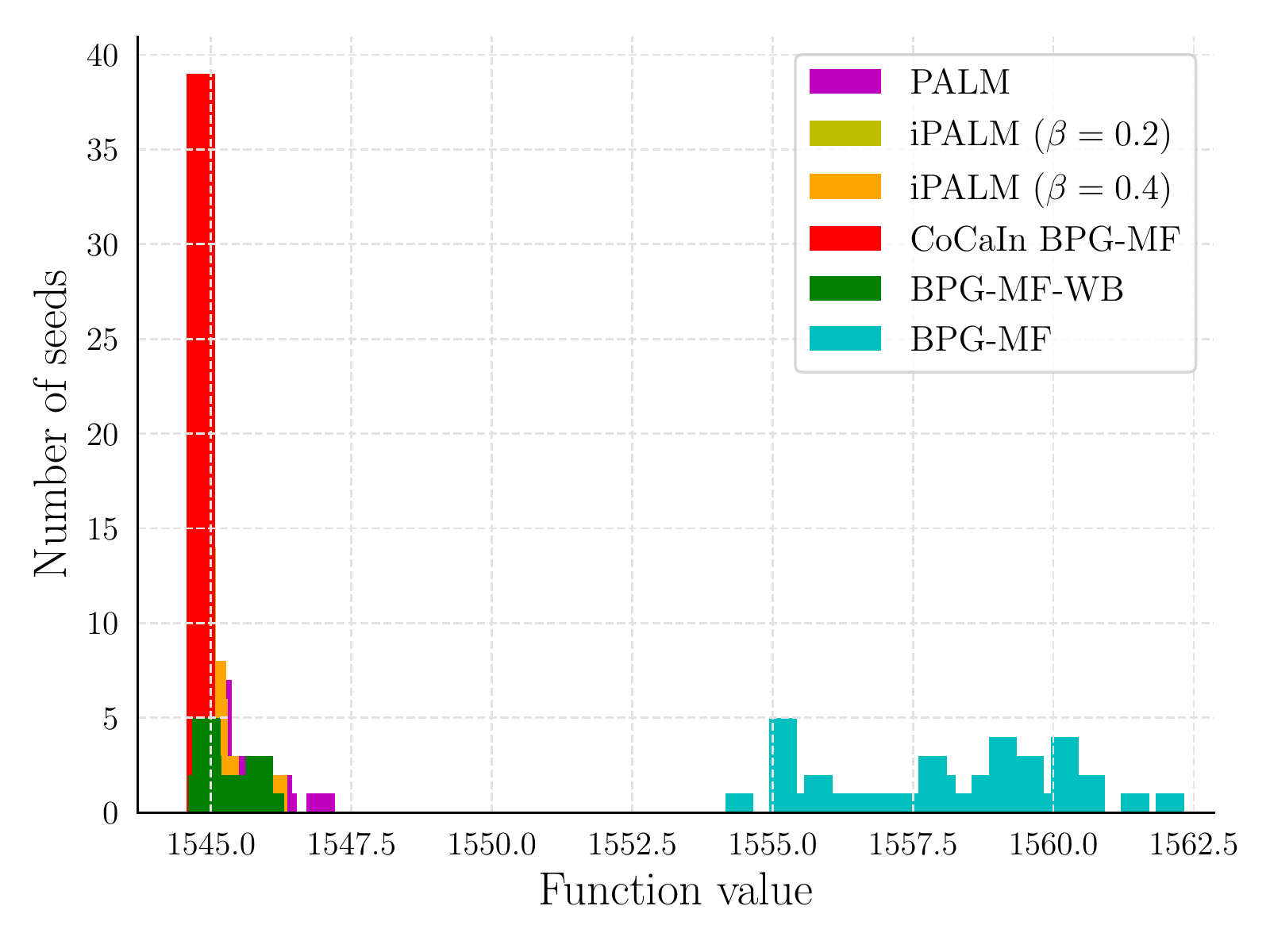}
			\caption{L2-Regularization}
		\end{subfigure}
		\begin{subfigure}{0.32\textwidth}
			\centering
			\includegraphics[width=1\textwidth]{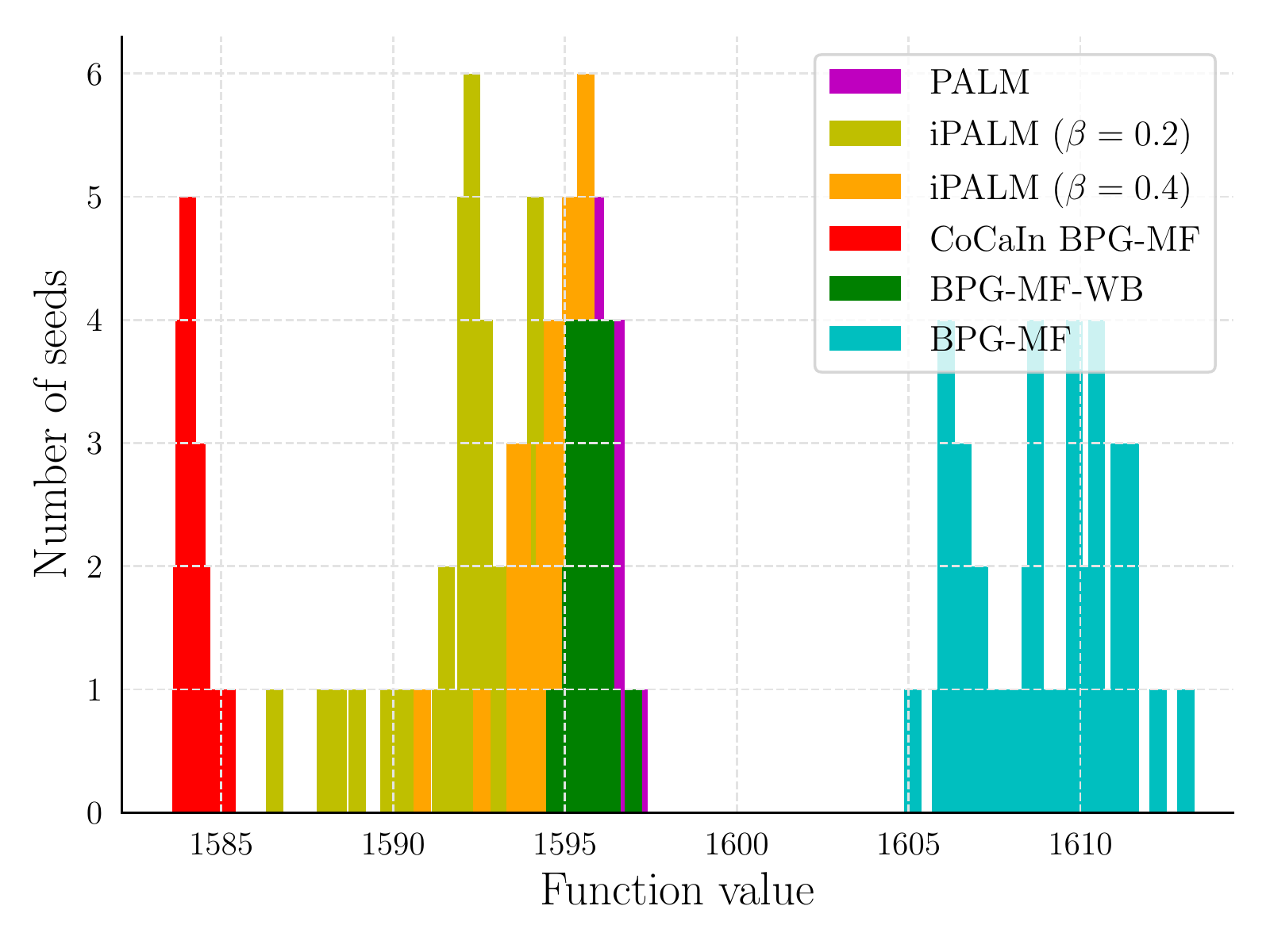}
			\caption{L1-Regularization}
		\end{subfigure}
		\caption{\textbf{Statistical Evaluation on Simple Matrix Factorization.}}
		\label{fig:statistical-evaluation}
	\end{figure}

	\begin{figure}[hbt!]
		\begin{subfigure}{0.325\textwidth}
			\centering
			\includegraphics[width=0.9\textwidth]{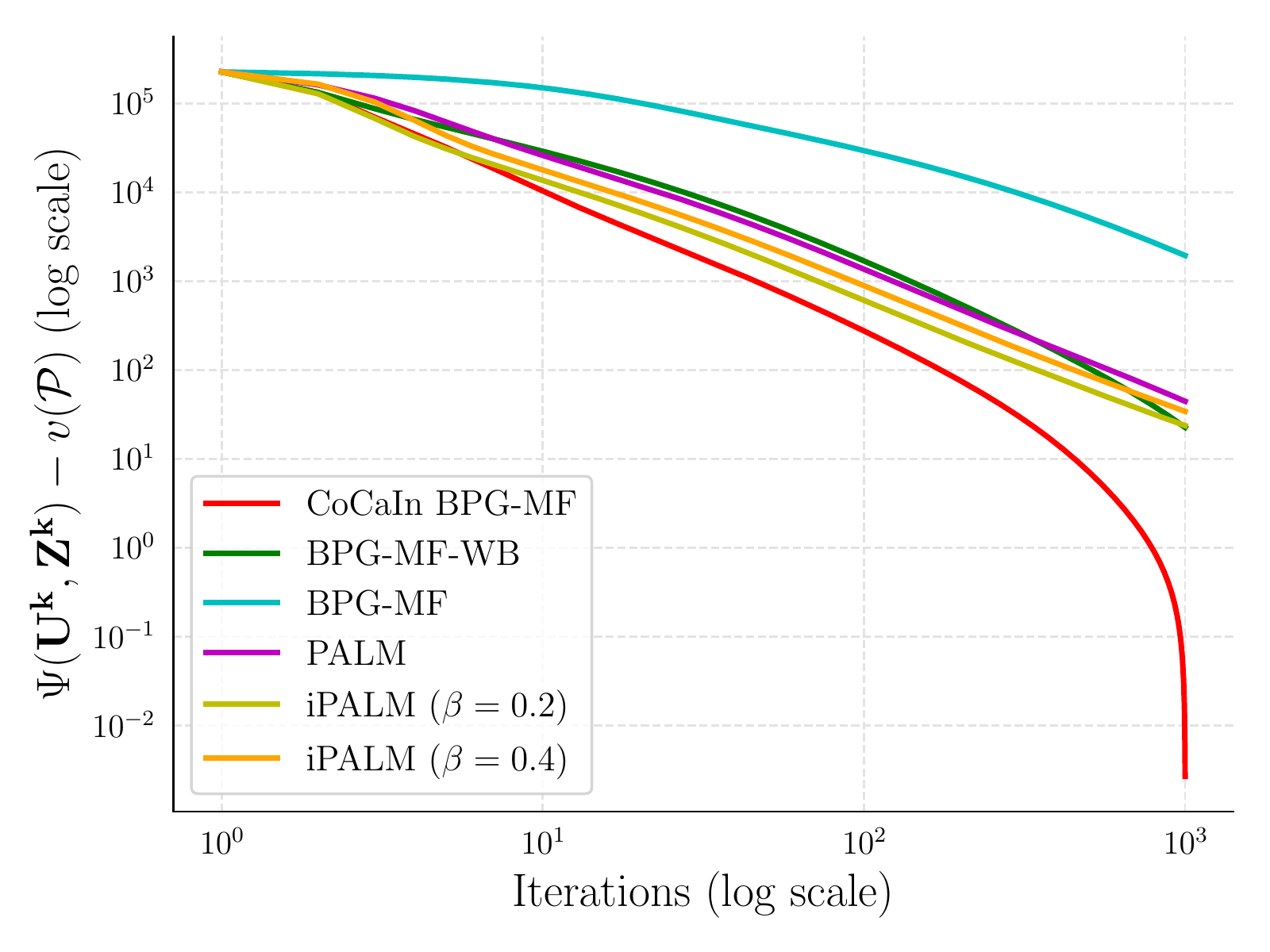}
			\caption{MovieLens-100K}
		\end{subfigure}
		\begin{subfigure}{0.325\textwidth}
			\centering
			\includegraphics[width=0.9\textwidth]{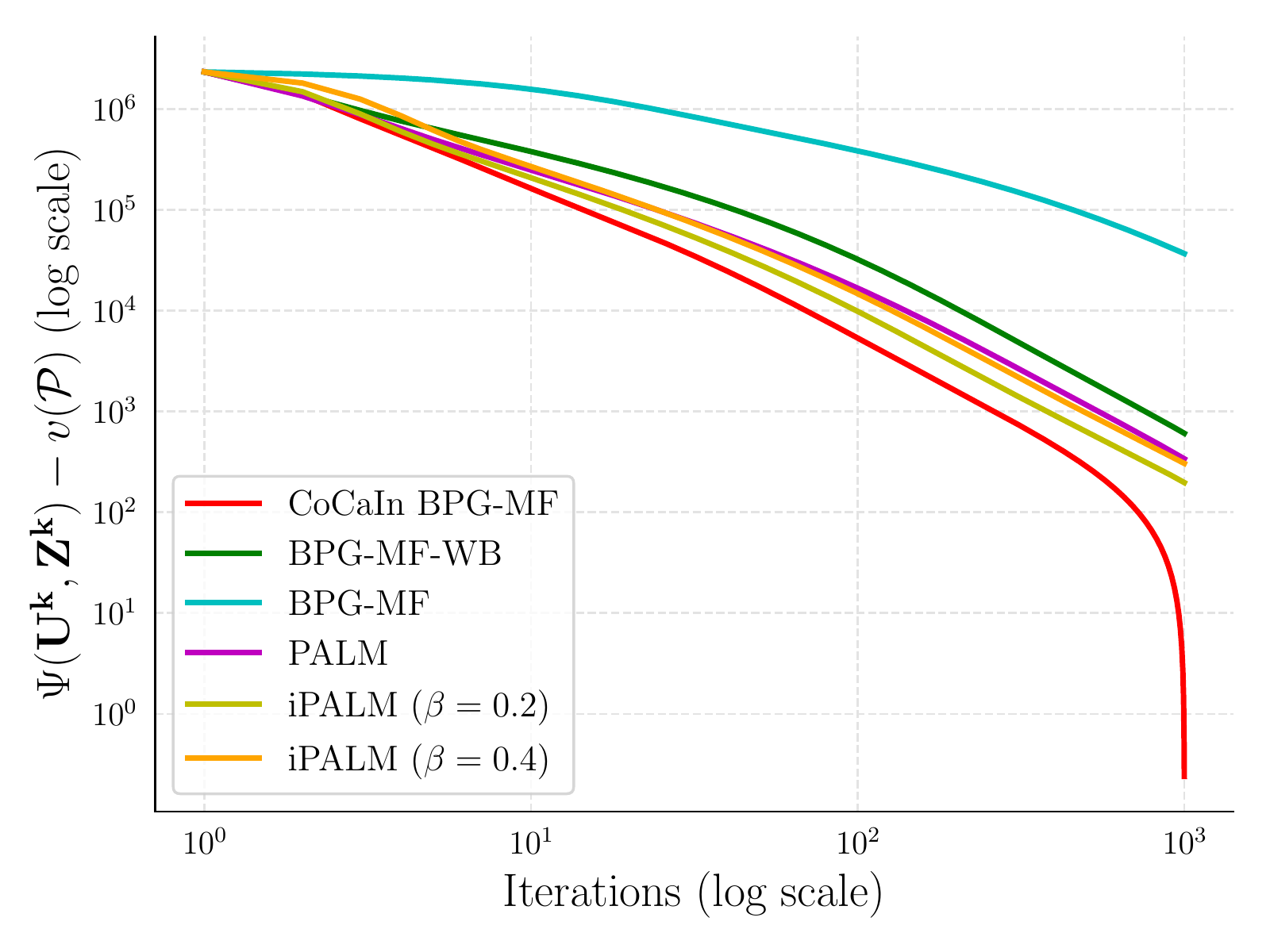}
			\caption{MovieLens-1M}
		\end{subfigure}
		\begin{subfigure}{0.325\textwidth}
			\centering
			\includegraphics[width=0.9\textwidth]{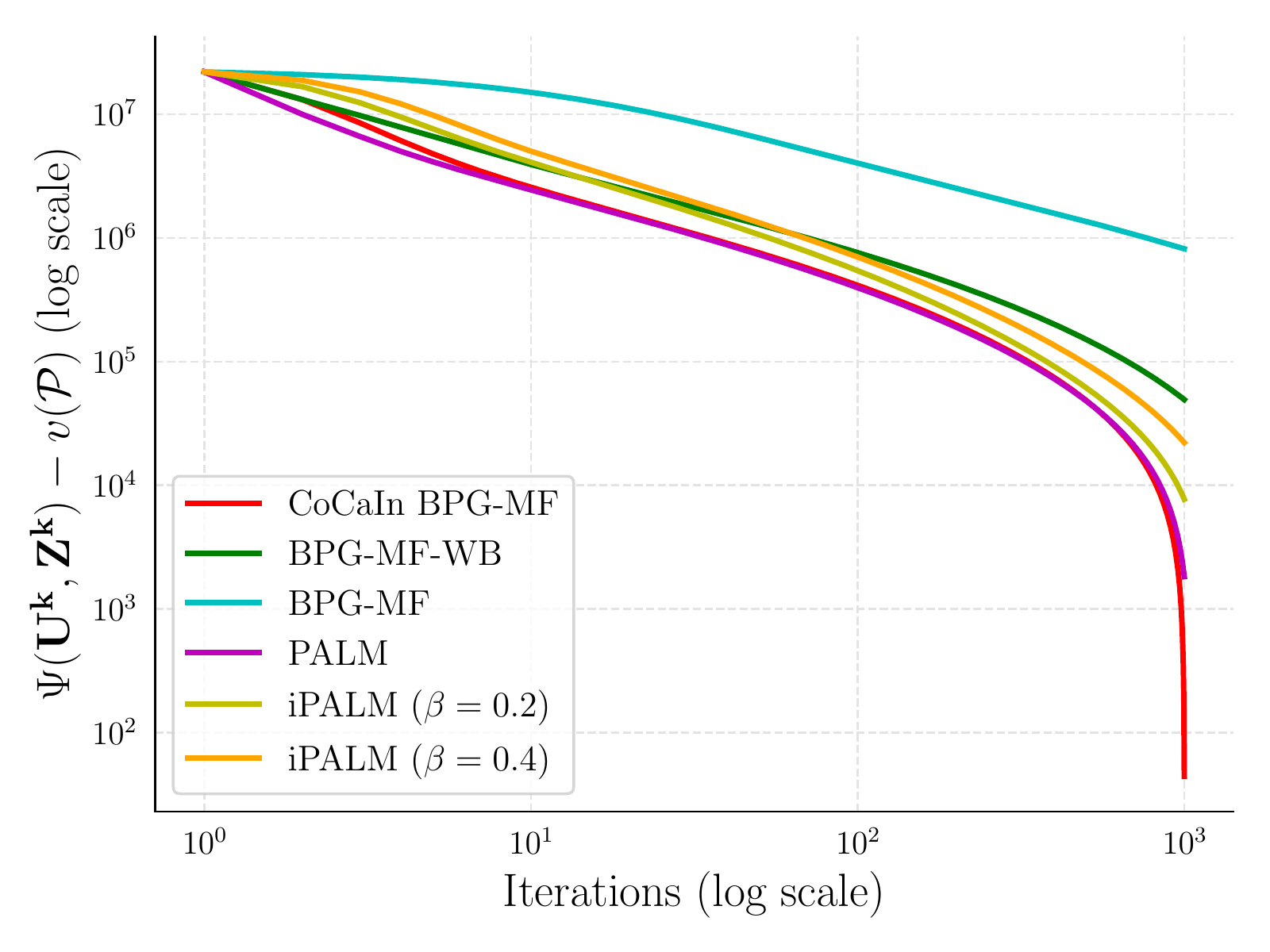}
			\caption{MovieLens-10M}
		\end{subfigure}
		\caption{\textbf{Matrix Completion on MovieLens Datasets \cite{HK2016}.} }
		\label{fig:movielens}
	\end{figure}
	
	As evident from  Figures~\ref{fig:synthetic}, \ref{fig:nmf-data}, \ref{fig:movielens}, CoCaIn BPG-MF, BPG-MF-WB can result in better performance than well known alternating methods. BPG-MF  is not better than PALM and iPALM because of prohibitively small step-sizes (due to $\norm{{\bf A}}_F$ in \eqref{eq:main-lsmad-h}), which is resolved by CoCaIn BPG-MF and BPG-MF-WB using backtracking. Time comparisons are provided in the \ifpaper supplementary material\else appendix\fi, where we show that our methods are competitive.
	\section*{Conclusion and Extensions} 
	We proposed non-alternating algorithms to solve matrix factorization problems, contrary to the typical alternating strategies. We use the Bregman proximal algorithms, BPG \cite{BSTV2018} and an inertial variant CoCaIn BPG \cite{MOPS2019} for matrix factorization problems. We developed a novel Bregman distance, crucial for proving convergence to a stationary point. Moreover, we also provide non-trivial efficient closed form update steps for many matrix factorization problems. This line of thinking raises new open questions, such as extensions to Tensor Factorization \cite{KB2009}, to Robust Matrix Factorization \cite{YK2018}, stochastic variants \cite{DDM2018,GLQSSR2019,MH2017,nguyen2018sgd} and state-of-the-art matrix factorization model  \cite{JM2018}.
\ifpaper
\else
	\section*{Acknowledgments}
	Mahesh Chandra Mukkamala and Peter Ochs were supported by the German Research Foundation (DFG Grant OC 150/1-1). We thank all the reviewers for providing their valuable comments. Mahesh Chandra Mukkamala thanks Antoine Gautier for his insightful comments.
\fi
\ifpaper
{\small 
\bibliographystyle{plain}
\bibliography{notes}
}
\newpage
\else
\fi
\appendix
\section{Discussion}
We briefly remark some properties of the update steps of BPG-methods. Note that the updates are independent for ${\bf U}$ and ${\bf Z}$ in \eqref{eq:mat-fac-ex-1}, where updates can be done in parallel blockwise (communication is only required to solve the 1D cubic equation). This can be potentially used to increase the speedup in practice, in particular for large matrices. Some terms in gradients overlap, so using temporary variables in implementation can possibly increase the speedup. These speedups are not restricted to  \eqref{eq:mat-fac-ex-1}, however to all the update steps we mentioned in this paper. \ifpaper\else\\\fi

We now provide insights on why BPG-methods are a better choice over other methods, with focus on alternating methods.
\begin{itemize}
\item PALM-methods estimate a Lipschitz constant with respect to a block of coordinates in each iteration, which is expensive for large block matrices. BPG-methods use a global L-smad constant, which is computed only once. 
\item PALM-methods cannot be parallelized block wise, for example, in the two block case, the computation of the Lipschitz constant of the second block must wait for the first block to be updated, hence it is inherently serial.
\item Alternating minimization methods do not converge for  non-smooth regularization terms and can be inefficient (for, e.g., ALS) for some matrix factorization problems (see, for example, \cite{KB2009,P1973}). BPG-methods and PALM-methods converge (due to linearization).
\item PALM is not applicable to the 2D function $g(x,y) = (x^3 + y^3)^2$, because the block-wise Lipschitz continuity of the gradients fails to hold even after fixing one variable. BPG-methods are applicable here.
\item PALM is not applicable to, for example, symmetric matrix Factorization as also pointed in \cite{DBA2019} or the following penalty method based (relaxed) orthogonal NMF problem (see \eqref{eq:prob-2})
\begin{equation*}
	\min_{{\bf U}\in \mathcal U, {\bf Z}\in \mathcal Z}\, \left\{ \Psi \equiv  \frac{1}{2}\norm{{\bf A} - {\bf U}{\bf Z}}^2_{F} + \frac{\rho}{2}\norm{{\bf U}^T{\bf U}-{\bf I}}_F^2 + {\bf I}_{{\bf U} \geq {\bf 0}} + {\bf I}_{{\bf Z} \geq {\bf 0}}+ {\mathcal R}_1({\bf U}) +  {\mathcal R}_2({\bf Z}) \right\}\,,
	\end{equation*}
where second term does not have a block-wise Lipschitz continuous gradient for any $\rho>0$. Here BPG-methods are applicable (similarly also for Projective NMF) with minor changes to the Bregman distance. For symmetric matrix factorization, we recover the kernel generating distances proposed in \cite{DBA2019}.
\item BPG-methods are very general so the choice of applications will increase substantially and  this will potentially open doors to design new losses and regularizers, without restricting to Lipschitz continuous gradients.
\end{itemize}

\textbf{State of the art models.}  The state-of-the-art matrix factorization models in  \cite{JM2018} go beyond two factors and new factorization models are introduced. BPG algorithms are not valid in their setting, and requires potentially developing new Bregman distances. Also, BPG based methods are not applicable for big data setting, where stochasticity plays a major role. The stochastic version of BPG was recently proposed in \cite{DDM2018}. The empirical comparisons to \cite{JM2018} is still open. Moreover, designing the appropriate kernels in the context of new factorization models can possibly require substantially technical proofs. 
\medskip


\textbf{Extensions.} Our algorithms can potentially extended to several applications, for example, multi-task learning, general matrix sensing, weighted PCA with various applications including cluster analysis, phase retrieval, power system state estimation. Even though CoCaIn BPG-MF appears to perform best, the performance of BPG-MF which forms the basis for CoCaIn BPG-MF, is worst as illustrated in \ref{sec:exps}. This possibly implies that the kernel choice or the coefficients involved in the kernels are not optimal. Such optimal choice of kernel generating distances were partially explored in the context of symmetric matrix factorization setting in \cite{DBA2019}, where new Bregman distances based on Gram kernels were introduced with state of the art performance in applicable settings.

\section{Overview of the Results}
Below, we provide a table with the problem or content description  and corresponding section where the results are presented.
\begin{center}
 \begin{tabular}{c c c } 
 \hline
 \textbf{Matrix Factorization problem} & \textbf{Section} \\ [0.5ex] 
 \hline
 Standard Matrix Factorization  &  Section~\ref{sec:closed-forms} \\ 
 \hline
 L2-Regularized Matrix Factorization &  Section~\ref{ssec:l2-reg-mf}  \\
 \hline
 Graph Regularized Matrix Factorization & Section~\ref{ssec:graph-reg-mf}  \\
 \hline
 L1-Regularized Matrix Factorization & Section~\ref{ssec:l1-reg-mf}  \\
 \hline
 Nuclear Norm Regularized Matrix Factorization  & Section~\ref{ssec:nuclear-reg-mf}  \\ 
  \hline
 Non-negative Matrix Factorization (NMF)  & Section~\ref{sec:nmf-closed-form} \\
 \hline
L2-regularized NMF  & Section~\ref{ssec:l2-reg-nmf}\\
 \hline
L1-regularized NMF  & Section~\ref{ssec:l1-reg-nmf} \\
 \hline
Graph Regularized NMF  & Section~\ref{ssec:exten-gnmf} \\
 \hline
Symmetric NMF via Non-Symmetric Relaxation  & Section~\ref{ssec:symnf-nmf} \\
 \hline
Sparse NMF & Section~\ref{ssec:sparse-nmf} \\
 \hline
Matrix Completion  & Section~\ref{sec:matrix-completion} \\
\hline
Closed Form Solution with 5th-order Polynomials  & Section~\ref{sec:closed-5-order} \\
\hline
Conversion to Cubic Equation & Section~\ref{ssec:conv-cubic}\\
\hline
Extensions to Mixed Regularization Terms  & Section~\ref{ssec:mixed-reg}\\
\hline
Technical Proofs  & Section~\ref{sec:technical-proofs}\\
 \hline
 Additional Experiments  & Section~\ref{sec:additional-experiments}\\
 \hline
\end{tabular}
\end{center}

\section{Closed Form Solutions Part I for Matrix Factorization }\label{sec:closed-forms}

	Since, the update steps of BPG-MF and CoCaIn BPG-MF have same structure, we provide the closed form expressions to just BPG-MF. We start with the following technical lemma.
	\begin{lemma}\label{lem:helper-1} Let ${\bf Q} \in \R^{A \times B}$ for some positive integers $A$ and $B$. Let $t \geq 0$ and $\norm{{{\bf Q}}}_F \neq 0$ then 
	\[
	\min_{{\bf X} \in \R^{A \times B}}\left\{ \act{{{\bf Q}}, {\bf X}} : \norm{\bf X}_F^2 = t^2 \right\} \equiv \min_{{\bf X} \in \R^{A \times B}}\left\{ \act{{{\bf Q}}, {\bf X}} : \norm{\bf X}_F^2 \leq t^2 \right\} = -{t}\norm{{{\bf Q}}}_F\,,
	 \]
	  with the minimizer at  ${\bf X}^* = -t{{\bf Q}}/\norm{{{\bf Q}}}_F$\,.
	\end{lemma}
	\begin{proof} The proof is inspired from \cite[Lemma 9]{LT2013}. On rewriting we have the following equivalence
	\[
	\min_{{\bf X} \in \R^{A \times B}}\left\{ \act{{{\bf Q}}, {\bf X}} : \norm{\bf X}_F^2 \leq  t^2 \right\} \equiv -\max_{{\bf X} \in \R^{A \times B}}\left\{ \act{-{\bf Q}, {\bf X}} : \norm{\bf X}_F^2 \leq  t^2 \right\}\,.
	\]
	The expression $\act{-{\bf Q}, {\bf X}}$ is maximized at ${\bf X^*} = c (-{\bf Q}) $ for certain constant $c$. On substituting we have $$\act{-{\bf Q}, {\bf X^*}} = c\norm{{\bf Q}}_F^2\,.$$
	Since, the dependence on $c$ is linear and we additionally require $\norm{\bf X}_F^2 \leq  t^2 $, we can set $c=\frac{t}{\norm{{\bf Q}}_F}$ if $\norm{{\bf Q}}_F \neq 0$ else $c=0$. Hence, the minimizer to
	\[ 
	\min_{{\bf X} \in \R^{A \times B}}\left\{ \act{{{\bf Q}}, {\bf X}} : \norm{\bf X}_F^2 \leq  t^2 \right\}
	\]
	is attained at ${\bf X^*} = -t\frac{{\bf Q}}{\norm{{\bf Q}}_F}$ for $\norm{{\bf Q}}_F \neq 0$ else ${\bf X^*} = 0$. The equivalence in the statement follows as $\norm{{\bf X^*}}_F^2 = t^2$.
	\end{proof}
	
	Consider the following non-convex matrix factorization problem
	\begin{equation}\label{eq:mat-fac-ex-1a}
	\min_{{\bf U}\in \R^{M \times K}, {\bf Z}\in \R^{K \times N}}\, \left\{ \Psi({\bf U},{\bf Z}) :=  \frac 12\norm{{\bf A} - {\bf U}{\bf Z}}^2_{F}  \right\}\,.
	\end{equation}
	Denote $g = \Psi$, $f :=0$, $h = h_a$.
	\begin{proposition}\label{prop:closed-form-1}
	In BPG-MF, with above defined $g,f,h$ the update steps in each iteration are given by ${\bf U}^{{\bf k+1}} = -r\,{\bf P^k}$, ${\bf Z}^{{\bf k+1}} = -r\,{\bf Q^k}$ where $r$ is the non-negative real root of
	\begin{align}
		&c_1\left(\norm{{\bf Q^k}}_F^2 +  \norm{{\bf P^k}}_F^2  \right)r^3 + c_2r-1 = 0\label{alg:simple-0}\,,
	\end{align}
	with $c_1 = 3$ and $c_2 = \norm{{\bf A}}_F$.
	\end{proposition}
	\begin{proof}
	Consider the following subproblem
	\begin{align*}
	  ({\bf U}^{{\bf k+1}}, {\bf Z}^{{\bf k+1}}) \in \underset{({\bf U},{\bf Z}) \in \R^{M \times K}\times \R^{K\times N}}{\argmin} &\left\{\act{ {\bf P^k}, {\bf U}} + \act{ {\bf Q^k} , {\bf Z}} \right.\\
	& \left. + c_1 \left(\frac{\norm{\bf U}_F^2 +\norm{\bf Z}_F^2 }{2} \right)^2 + c_2\left( \frac{\norm{\bf U}_F^2 +\norm{\bf Z}_F^2 }{2}\right)  \right\}\,.
	\end{align*}
	Denote the objective in the above minimization problem as $\mathcal{O({\bf U}^{\bf k}, {\bf Z}^{\bf k})}$. Now, the following holds
	\begin{align}
		&\min_{({\bf U},{\bf Z}) \in \R^{M \times K}\times \R^{K\times N}} \left(\mathcal{O({\bf U}^{\bf k}, {\bf Z}^{\bf k})}\right)\nonumber \\
		&\equiv \min_{t_1 \geq 0, t_2 \geq 0} \left\{\min_{({\bf U},{\bf Z}) \in \R^{M \times K}\times \R^{K\times N},\norm{\bf U}_F =t_1,\norm{\bf Z}_F =t_2} \left(\mathcal{O({\bf U}^{\bf k}, {\bf Z}^{\bf k})} \right)\right\}\,,\label{eq:subprob-equiv-1}\\
		&\equiv \min_{t_1 \geq 0, t_2 \geq 0} \left\{\min_{({\bf U},{\bf Z}) \in \R^{M \times K}\times \R^{K\times N},\norm{\bf U}_F \leq t_1,\norm{\bf Z}_F \leq t_2} \left(\mathcal{O({\bf U}^{\bf k}, {\bf Z}^{\bf k})} \right)\right\}\label{eq:subprob-equiv-2}\,,
	\end{align}
	where the first step is a simple rewriting of the objective.  The second step is non-trivial. In order to prove \eqref{eq:subprob-equiv-2} we rewrite \eqref{eq:subprob-equiv-1} as 
	\begin{align*}
		\min_{t_1 \geq 0, t_2 \geq 0} &\left\{  \min_{{\bf U_1} \in \R^{M \times K}} \left\{ \act{ {\bf P^k}, {\bf U_1}  } : \norm{{\bf U_1}}_F^2 = t_1\right\} \right. \\
	& \left. + \min_{{\bf Z_1} \in \R^{K\times N} } \left\{ \act{ {\bf Q^k} , {\bf Z_1}}: \norm{{\bf Z_1}}_F^2 = t_2\right\}\right.\\
	&\left. + c_1\left(\frac{t_1^2 + t_2^2}{2}\right)^2 + c_2\left(\frac{t_1^2 + t_2^2}{2}\right) \right\}\,.
	\end{align*}
	Now, note the following equivalence due to Lemma~\ref{lem:helper-1}  
	\begin{align*}
		\min_{{\bf U_1} \in \R^{M \times K}} \left\{ \act{ {\bf P^k}, {\bf U_1}  } : \norm{{\bf U_1}}_F^2 = t_1\right\} &\equiv \min_{{\bf U_1} \in \R^{M \times K}} \left\{ \act{ {\bf P^k}, {\bf U_1}  } : \norm{{\bf U_1}}_F^2 \leq t_1\right\}\,,\\
		\min_{{\bf Z_1} \in \R^{K\times N} } \left\{ \act{ {\bf Q^k} , {\bf Z_1}}: \norm{{\bf Z_1}}_F^2 = t_2\right\} &\equiv \min_{{\bf Z_1} \in \R^{K\times N} } \left\{ \act{ {\bf Q^k} , {\bf Z_1}}: \norm{{\bf Z_1}}_F^2 \leq t_2\right\}\,.
	\end{align*}
	This proves \eqref{eq:subprob-equiv-2}. Now, we solve for  $({\bf U}^{{\bf k+1}}, {\bf Z}^{{\bf k+1}})$ via the following strategy. Denote
	\[
	{{\bf U}^{{\bf *}}_1}(t_1) \in  \argmin \left\{ \act{ {\bf P^k}, {\bf U_1}  } : {\bf U_1} \in \R^{M \times K}\,, \norm{{\bf U_1}}_F^2 \leq t_1\right\}\,,
	\]
	\[
		{{\bf Z}^{*}_1 }(t_2) \in \argmin \left\{ \act{ {\bf Q^k} , {\bf Z_1}}: {\bf Z_1} \in \R^{K\times N}\,, \norm{{\bf Z_1}}_F^2 \leq t_2 \right\}\,.
	\]
	Then we obtain  $({\bf U}^{{\bf k+1}}, {\bf Z}^{{\bf k+1}}) =  ({{\bf U}^{{\bf *}}_1}(t_1^*),{{\bf Z}^{*}_1 }(t_2^*))$, where $t_1^*$ and $t_2^*$ are obtained by solving the following two dimensional subproblem
	\begin{align*}
	  (t_1^*,t_2^*) \in  \underset{t_1\geq 0, t_2\geq 0}{\mathrm{argmin}}   &\left\{  \min_{{\bf U_1} \in \R^{M \times K}} \left\{ \act{ {\bf P^k}, {\bf U_1}  } : \norm{{\bf U_1}}_F^2 \leq t_1\right\} \right. \\
	& \left. + \min_{{\bf Z_1} \in \R^{K\times N} } \left\{ \act{ {\bf Q^k} , {\bf Z_1}}: \norm{{\bf Z_1}}_F^2 \leq t_2\right\}\right.\\
	&\left. + c_1\left(\frac{t_1^2 + t_2^2}{2}\right)^2 + c_2\left(\frac{t_1^2 + t_2^2}{2}\right) \right\}\,.
	\end{align*}

	Note that inner minimization subproblems can be trivially solved once we obtain ${{\bf U}^{{\bf *}}_1}(t_1)$ and ${{\bf Z}^{*}_1 }(t_2)$ via Lemma~\ref{lem:helper-1}. Then the solution to the subproblem in each iteration is as follows:
	\begin{align*}
	{\bf U}^{{\bf k+1}} 
	&= \left.
	\begin{cases}
	    t_1^*\frac{-{\bf P^k}}{\norm{{\bf P^k}}_F}, & \text{for } \norm{{\bf P^k}}_F \neq 0\,, \\
	    {\bf 0} &  otherwise\,.
	\end{cases}\right.\\
	{\bf Z}^{{\bf k+1}} 
	&= \left.
	\begin{cases}
	    t_2^*\frac{-{\bf Q^k}}{\norm{{\bf Q^k}}_F}, & \text{for } \norm{{\bf Q^k}}_F \neq 0\,, \\
	    {\bf 0} &  otherwise\,.
	\end{cases}
	 \right.
	\end{align*}
	We solve for $t_1^*$ and $t_2^*$ with the following two dimensional minimization problem
	\begin{align*}
	\underset{t_1\geq 0, t_2\geq 0}{\mathrm{argmin}} & \left\{  -t_1\norm{{\bf P^k}}_F  - t_2{\norm{{\bf Q^k}}_F}  + c_1\left(\frac{t_1^2 + t_2^2}{2}\right)^2 + c_2\left(\frac{t_1^2 + t_2^2}{2}\right)\right\}\,.
	\end{align*}
	Thus, the solutions $t_1^*$ and $t_2^*$ are the non-negative real roots of the following equations
	\begin{align*}
	&-\norm{{\bf P^k}}_F  + c_1(t_1^2 + t_2^2)t_1 + c_2t_1 = 0\\
	&-\norm{{\bf Q^k}}_F + c_1(t_1^2 + t_2^2)t_2 + c_2t_2 = 0
	\end{align*}
	Further simplifications lead to $t_1 = r \norm{{\bf P^k}}_F$ and $t_2 = r \norm{{\bf Q^k}}_F$ for some $r \geq 0$ such that $r$ satisfies the following cubic equation
	\[
		c_1\left(\norm{{\bf Q^k}}_F^2 +  \norm{{\bf P^k}}_F^2  \right)r^3 + c_2r-1 = 0\,.
	\]
	\end{proof}

	\subsection{Extensions to L2-Regularized Matrix Factorization}\label{ssec:l2-reg-mf} We consider the following L2-Regularized Matrix Factorization problem \cite{LY2009}.
	\begin{equation}\label{eq:mat-fac-ex-2}
	\min_{{\bf U}\in \R^{M \times K}, {\bf Z}\in \R^{K \times N}}\, \left\{ \Psi({\bf U},{\bf Z}) :=  \frac 12\norm{{\bf A} - {\bf U}{\bf Z}}^2_{F} + \frac{\lambda_0}{2}\left(\norm{{\bf U} }_F^2 + \norm{{\bf Z} }_F^2 \right)  \right\}\,.
	\end{equation}
	Denote $g := \frac 12\norm{{\bf A} - {\bf U}{\bf Z}}^2_{F}  $, $f:=    \frac{\lambda_0}{2}\left(\norm{{\bf U} }_F^2 + \norm{{\bf Z} }_F^2 \right)$ and $h = h_a$.
	\begin{proposition}\label{prop:closed-form-l2a}
	In BPG-MF, with the above defined $g,f,h$ the update steps in each iteration are given by ${\bf U}^{{\bf k+1}} = -r\,{\bf P^k}$, ${\bf Z}^{{\bf k+1}} = -r\,{\bf Q^k}$ where $r$ is the non-negative real root of
	\begin{align}
		&c_1\left(\norm{{\bf Q^k}}_F^2 +  \norm{{\bf P^k}}_F^2  \right)r^3 + (c_2+\lambda_0)r-1 = 0\label{alg:simple-0}\,,
	\end{align}
	with $c_1 = 3$ and $c_2 = \norm{{\bf A}}_F$.
	\end{proposition}
	We skip the proof as it is very similar to Proposition~\ref{prop:closed-form-1} and only change is in $c_2$.

	\subsection{Extensions to Graph Regularized Matrix Factorization}\label{ssec:graph-reg-mf} Graph Regularized Matrix Factorization was proposed in \cite{CHHH2011}. However, they used non-negativity constraints. We simplify the problem here by not considering the non-negativity constraints. We later show in Section~\ref{ssec:exten-gnmf}, how the  non-negativity constraints are handled. Here, given ${\bf \mathcal L} \in \R^{M \times M}$ we are interested to solve
	\begin{equation*}
	\min_{{\bf U}\in \R^{M \times K}, {\bf Z}\in \R^{K \times N}}\, \left\{ \Psi({\bf U},{\bf Z}) :=  \frac 12\norm{{\bf A} - {\bf U}{\bf Z}}^2_{F} +\frac{\mu_0}{2}{\bf tr}({\bf U}^T{\bf \mathcal L}{\bf U}) + \frac{\lambda_0}{2}\left(\norm{{\bf U} }_F^2 + \norm{{\bf Z} }_F^2 \right)  \right\}\,.
	\end{equation*}
	In such a case, it is easy to extend the following ideas to Graph Regularized Non-negative Matrix Factorization.	We show here $L$-smad property. We first need the following technical lemma.
	\begin{lemma}\label{lem:hessain-helper-1}
	Let $g_1({\bf U})  = {\bf tr}({\bf U}^T{\bf \mathcal L}{\bf U})$, then for any ${\bf H} \in \R^{M \times K}$ we have $\nabla g_1({\bf U}) = {\bf \mathcal L}{\bf U} + {\bf \mathcal L}^T{\bf U}$, 
	\begin{align*}
	\act{{\bf H}, \nabla^2 g_1({\bf U}){\bf H}} &= 2\act{{\bf \mathcal L}{\bf H}, {\bf H}}\,.
	\end{align*}
	\end{lemma}
	\begin{proof}
	Note that ${\bf tr}({\bf U}^T{\bf \mathcal L}{\bf U}) = \act{{\bf \mathcal L}{\bf U}, {\bf U}}$, now we obtain for ${\bf H} \in \R^{M \times K}$ the following
	\begin{align*}
	\act{{\bf \mathcal L}({\bf U + H}), {\bf U + H}} &= \act{{\bf \mathcal L}({\bf U + H}), {\bf U + H}}\\
	&= \act{{\bf \mathcal L}{\bf U}, {\bf U}} + \act{{\bf \mathcal L}{\bf U}, {\bf H}} + \act{{\bf \mathcal L}{\bf H}, {\bf U}} + \act{{\bf \mathcal L}{\bf H}, {\bf H}}\,,\\
	&= \act{{\bf \mathcal L}{\bf U}, {\bf U}} + \act{{\bf \mathcal L}{\bf U}, {\bf H}} + \act{{\bf \mathcal L}^T{\bf U}, {\bf H}} + \act{{\bf \mathcal L}{\bf H}, {\bf H}}\,.
	\end{align*}
	Thus the statement holds, by collecting the first and second order terms.
	\end{proof}
	Now, we prove the $L$-smad property.
	\begin{proposition}\label{prop:l-smad-2}
	Let $g({\bf U}, {\bf Z}) = \frac 12\norm{{\bf A} - {\bf U}{\bf Z}}^2_{F} +\frac{\mu_0}{2}{\bf tr}({\bf U}^T{\bf \mathcal L}{\bf U})$. Then, for a certain constant $L\geq 1$, the function $g$ satisfies $L$-smad property with respect to the following kernel generating distance,
	 \[
	 h_c({\bf U},{\bf Z}) =  3h_1({\bf U},{\bf Z}) + \left(\norm{{\bf A}}_F + \mu_0\norm{\bf \mathcal L}_F\right) h_2({\bf U},{\bf Z}) \,.
	 \]
	\end{proposition}
	\begin{proof}
		The proof is similar to Proposition~\ref{prop:l-smad-1} and Lemma~\ref{lem:hessain-helper-1} must be applied for the result.
	\end{proof}
	Denote $g := \frac 12\norm{{\bf A} - {\bf U}{\bf Z}}^2_{F} +\frac{\mu_0}{2}{\bf tr}({\bf U}^T{\bf \mathcal L}{\bf U})$, $f:=  \frac{\lambda_0}{2}\left(\norm{{\bf U} }_F^2 + \norm{{\bf Z} }_F^2 \right) $ and $h = h_c$.
	\begin{proposition}\label{prop:closed-form-2a}
	In BPG-MF, with the above defined $f,g,h$ the update steps in each iteration are given by ${\bf U}^{{\bf k+1}} = -r\,{\bf P^k}$, ${\bf Z}^{{\bf k+1}} = -r\,{\bf Q^k}$ where $r\geq 0$ and satisfies
	\begin{align}
		&c_1\left(\norm{{\bf Q^k}}_F^2 +  \norm{{\bf P^k}}_F^2  \right)r^3 + (c_2+\mu_0\norm{\bf \mathcal L}_F + \lambda_0)r-1 = 0\,,\label{alg:simple-0}
	\end{align}
	with $c_1 = 3$ and $c_2 = \norm{{\bf A}}_F$.
	\end{proposition}
	The proof is similar to Proposition~\ref{prop:closed-form-1} and only $c_2$ changes.
	\subsection{Extensions to L1-Regularized Matrix Factorization}\label{ssec:l1-reg-mf}
	Now consider the following matrix factorization problem with L1-Regularization
	\begin{equation}\label{eq:mat-fac-ex-3}
	\min_{{\bf U}\in \R^{M \times K}, {\bf Z}\in \R^{K \times N}}\, \left\{ \Psi({\bf U},{\bf Z}) :=  \frac 12\norm{{\bf A} - {\bf U}{\bf Z}}^2_{F} + \lambda_1\left(\norm{{\bf U} }_1 + \norm{{\bf Z} }_1 \right)  \right\}\,.
	\end{equation}
	Recall that {soft-thresholding operator} is defined for any $y \in \R^d$ by
	\begin{equation} \label{eq:soft}
		\SSS_{\theta}\left(y\right) = \argmin_{x \in \real^{d}} \left\{ \theta\norm{x}_{1} + \frac{1}{2}\norm{x - y}^{2} \right\} = \max\left\{ \left|y\right| - \theta , 0 \right\}\sgn\left(y\right)\,,
	\end{equation}
	where $\theta>0$ and the operations are applied element-wise. We require the following technical result.
	\begin{lemma}\label{lem:helper-2} Let ${\bf Q} \in \R^{A \times B}$ for some positive integers $A$ and $B$. Let $t_0 >0$ and let $t \geq 0$ then 
	\[
	 \min_{{\bf X} \in \R^{A \times B}}\left\{ \act{{{\bf Q}}, {\bf X}} + t_0\norm{{\bf X}}_1 : \norm{{\bf X}}_F^2 \leq t^2 \right\} = -t\norm{S_{t_0}(-{\bf Q})}_F\,.
	 \]
	 with the minimizer at ${\bf X^*} = t\frac{S_{t_0}(-{\bf Q})}{\norm{S_{t_0}(-{\bf Q})}_F}$ for $\norm{S_{t_0}(-{\bf Q})}_F \neq 0$ and otherwise all ${\bf X}$ such that $\norm{{\bf X}}_F^2 \leq t^2$ are minimizers. Moreover  we have the following equivalence,
	 \begin{equation}\label{eq:equivalence}
	 \min_{{\bf X} \in \R^{A \times B}}\left\{ \act{{{\bf Q}}, {\bf X}} + t_0\norm{{\bf X}}_1 : \norm{{\bf X}}_F^2 \leq t^2 \right\} \equiv \min_{{\bf X} \in \R^{A \times B}}\left\{ \act{{{\bf Q}}, {\bf X}} + t_0\norm{{\bf X}}_1 : \norm{{\bf X}}_F^2 = t^2 \right\} \,.
	 \end{equation}
	\end{lemma}
	\begin{proof}
	We have the following equivalence
	\begin{align*}
		\min_{{\bf X} \in \R^{A \times B}}\left\{ \act{{{\bf Q}}, {\bf X}} + t_0\norm{{\bf X}}_1 : \norm{{\bf X}}_F^2 \leq t^2 \right\} \equiv -\max_{{\bf X} \in \R^{A \times B}}\left\{ \act{{-{\bf Q}}, {\bf X}} - t_0\norm{{\bf X}}_1 : \norm{{\bf X}}_F^2 \leq t^2 \right\}\,.
	\end{align*}
	Then the result follows due to \cite[Proposition 14]{LT2013} with the minimizer at ${\bf X^*} = t\frac{S_{t_0}(-{\bf Q})}{\norm{S_{t_0}(-{\bf Q})}_F}$ for $\norm{S_{t_0}(-{\bf Q})}_F \neq 0$ and ${\bf 0}$ otherwise. The equivalence statement in \eqref{eq:equivalence} follows as $\norm{{\bf X^*}}_F^2 = t^2$ for $\norm{S_{t_0}(-{\bf Q})}_F \neq 0$ and otherwise all the points satisfying $\norm{{\bf X}}_F^2 = t^2$ are minimizers.
	\end{proof}
	Denote $g := \frac 12\norm{{\bf A} - {\bf U}{\bf Z}}^2_{F}$, $f:=   \lambda_1\left(\norm{{\bf U} }_1 + \norm{{\bf Z} }_1 \right)$ and $h = h_a$. 
	\begin{proposition}\label{prop:closed-form-3}
		In BPG-MF, with the above defined $g,f,h$ the update steps in each iteration are given by ${\bf U}^{{\bf k+1}} = r\SSS_{\lambda_1\lambda}(-{\bf P^k})$, ${\bf Z}^{{\bf k+1}} = r\SSS_{\lambda_1\lambda}(-{\bf Q^k})$ where $r\geq 0$ and satisfies
	\begin{align}
		&c_1\left(\norm{\SSS_{\lambda_1\lambda}\left(-{\bf Q^k}\right)}_F^2 +  \norm{\SSS_{\lambda_1\lambda}\left(-{\bf P^k}\right)}_F^2  \right)r^3 + c_2r-1 = 0\label{alg:simple-1}\,,
	\end{align}
	with $c_1 = 3$ and $c_2 = \norm{{\bf A}}_F$.
	\end{proposition}
	\begin{proof}The proof is similar to that of Proposition~\ref{prop:closed-form-1}, but with certain changes due to the  L1 norm in the objective. Consider the following subproblem
	\begin{align*}
	  ({\bf U}^{{\bf k+1}}, {\bf Z}^{{\bf k+1}}) \in \underset{({\bf U},{\bf Z}) \in \R^{M \times K}\times \R^{K\times N}}{\argmin} &\left\{\lambda\lambda_1\left(\norm{{\bf U} }_1 + \norm{{\bf Z} }_1 \right)   + \act{ {\bf P^k}, {\bf U}} + \act{ {\bf Q^k} , {\bf Z}} \right.\\
	  &\left.+ c_1 \left(\frac{\norm{\bf U}_F^2 +\norm{\bf Z}_F^2 }{2} \right)^2 + c_2\left( \frac{\norm{\bf U}_F^2 +\norm{\bf Z}_F^2 }{2}\right) \right\}\,,
	\end{align*}
	Denote the objective in the above minimization problem as $\mathcal{O({\bf U}^{\bf k}, {\bf Z}^{\bf k})}$. Now, we show that the following holds
	\begin{align}
		&\min_{({\bf U},{\bf Z}) \in \R^{M \times K}\times \R^{K\times N}} \left(\mathcal{O({\bf U}^{\bf k}, {\bf Z}^{\bf k})}\right) \nonumber\\
		&\equiv \min_{t_1 \geq 0, t_2 \geq 0} \left\{\min_{({\bf U},{\bf Z}) \in \R^{M \times K}\times \R^{K\times N},\norm{\bf U}_F =t_1,\norm{\bf Z}_F =t_2} \left(\mathcal{O({\bf U}^{\bf k}, {\bf Z}^{\bf k})} \right)\right\}\,,\label{eq:subprob-equiv-1a}\\
		&\equiv \min_{t_1 \geq 0, t_2 \geq 0} \left\{\min_{({\bf U},{\bf Z}) \in \R^{M \times K}\times \R^{K\times N},\norm{\bf U}_F \leq t_1,\norm{\bf Z}_F \leq t_2} \left(\mathcal{O({\bf U}^{\bf k}, {\bf Z}^{\bf k})} \right)\right\}\label{eq:subprob-equiv-2a}\,.
	\end{align}
	where the first step is a simple rewriting of the objective.  The second step is non-trivial. In order to prove \eqref{eq:subprob-equiv-2a} we rewrite \eqref{eq:subprob-equiv-1a} as 
	\begin{align*}
		\min_{t_1 \geq 0, t_2 \geq 0} &\left\{  \min_{{\bf U_1} \in \R^{M \times K}} \left\{ \act{ {\bf P^k}, {\bf U_1}  } + \lambda\lambda_1\norm{{\bf U} }_1 : \norm{{\bf U_1}}_F^2 = t_1\right\} \right. \\
	& \left. + \min_{{\bf Z_1} \in \R^{K\times N} } \left\{ \act{ {\bf Q^k} , {\bf Z_1}} + \lambda\lambda_1\norm{{\bf Z} }_1: \norm{{\bf Z_1}}_F^2 = t_2\right\}\right.\\
	&\left. + c_1\left(\frac{t_1^2 + t_2^2}{2}\right)^2 + c_2\left(\frac{t_1^2 + t_2^2}{2}\right) \right\}\,.
	\end{align*}

	where the second step \eqref{eq:subprob-equiv-2a} uses Lemma~\ref{lem:helper-2} and strong convexity of $h$. Now,  note the following equivalence due to Lemma~\ref{lem:helper-2} 
	\begin{align}
		&\min_{{\bf U_1} \in \R^{M \times K}} \left\{ \act{ {\bf P^k}, {\bf U_1}  } + \lambda\lambda_1\norm{{\bf U} }_1 : \norm{{\bf U_1}}_F^2 = t_1\right\} \nonumber \\
		&\equiv \min_{{\bf U_1} \in \R^{M \times K}} \left\{ \act{ {\bf P^k}, {\bf U_1}  }  + \lambda\lambda_1\norm{{\bf U} }_1: \norm{{\bf U_1}}_F^2 \leq t_1\right\} \label{eq:rewrite-1}\,,
	\end{align} 
	and
	\begin{align}
		&\min_{{\bf Z_1} \in \R^{K\times N} } \left\{ \act{ {\bf Q^k} , {\bf Z_1}} + \lambda\lambda_1\norm{{\bf Z} }_1: \norm{{\bf Z_1}}_F^2 = t_2\right\}\nonumber\\
		 &\equiv \min_{{\bf Z_1} \in \R^{K\times N} } \left\{ \act{ {\bf Q^k} , {\bf Z_1}} + \lambda\lambda_1\norm{{\bf Z} }_1: \norm{{\bf Z_1}}_F^2 \leq t_2\right\}\label{eq:rewrite-2}\,.
	\end{align}
	We solve the subproblems via the following strategy. Denote
	\[
	{{\bf U}^{{\bf *}}_1}(t_1) \in  \argmin \left\{ \act{ {\bf P^k}, {\bf U_1}  } + \lambda\lambda_1\norm{{\bf U} }_1  : {\bf U_1} \in \R^{M \times K}\,, \norm{{\bf U_1}}_F^2 \leq t_1\right\}
	\]
	\[
		{{\bf Z}^{*}_1 }(t_2)\in \argmin \left\{ \act{ {\bf Q^k} , {\bf Z_1} } + \lambda\lambda_1\norm{{\bf Z} }_1: {\bf Z_1} \in \R^{K\times N}\,, \norm{{\bf Z_1}}_F^2 \leq t_2 \right\}
	\]
	Then we obtain  $({\bf U}^{{\bf k+1}}, {\bf Z}^{{\bf k+1}}) =  ({{\bf U}^{{\bf *}}_1}(t_1^*),{{\bf Z}^{*}_1 }(t_2^*))$, where $t_1^*$ and $t_2^*$ are obtained by solving the following two dimensional subproblem
	\begin{align*}
	  (t_1^*,t_2^*) \in  \underset{t_1\geq 0, t_2\geq 0}{\mathrm{argmin}}   &\left\{  \min_{{\bf U_1} \in \R^{M \times K}} \left\{ \act{ {\bf P^k}, {\bf U_1}  } + \lambda\lambda_1\norm{{\bf U} }_1 : \norm{{\bf U_1}}_F^2 \leq t_1\right\} \right. \\
	& \left. + \min_{{\bf Z_1} \in \R^{K\times N} } \left\{ \act{ {\bf Q^k} , {\bf Z_1}} + \lambda\lambda_1\norm{{\bf Z} }_1: \norm{{\bf Z_1}}_F^2 \leq t_2\right\} \right. \\
	&\left. + c_1\left(\frac{t_1 + t_2}{2}\right)^2 + c_2\left(\frac{t_1 + t_2}{2}\right) \right\}\,.
	\end{align*}

	Note that inner minimization subproblems can be trivially solved once we obtain ${{\bf U}^{{\bf *}}_1}(t_1)$ and ${{\bf Z}^{*}_1 }(t_2)$. Due to Lemma~\ref{lem:helper-2} we obtain the solution to the subproblem in each iteration as follows
	\begin{align*}
	{\bf U}^{{\bf k+1}} 
	&= \left.
	\begin{cases}
	    t_1^*\frac{S_{\lambda\lambda_1}(-{\bf P^k})}{\norm{S_{\lambda\lambda_1}(-{\bf P^k})}_F}, & \text{for } \norm{S_{\lambda\lambda_1}(-{\bf P^k})}_F \neq 0\,, \\
	    {\bf 0} &  otherwise\,.
	\end{cases}\right.\\
	{\bf Z}^{{\bf k+1}} 
	&= \left.
	\begin{cases}
	    t_2^*\frac{S_{\lambda\lambda_1}(-{\bf Q^k})}{\norm{S_{\lambda\lambda_1}(-{\bf Q^k})}_F}, & \text{for } \norm{S_{\lambda\lambda_1}(-{\bf Q^k})}_F \neq 0\,, \\
	    {\bf 0} &  otherwise\,.
	\end{cases}
	 \right.
	\end{align*}
	We solve for $t_1^*$ and $t_2^*$ with the following two dimensional minimization problem
	\begin{align*}
	\underset{t_1\geq 0, t_2\geq 0}{\mathrm{argmin}} & \left\{  -t_1\norm{S_{\lambda\lambda_1}(-{\bf P^k})}_F  - t_2{\norm{S_{\lambda\lambda_1}(-{\bf Q^k})}_F}   + c_1\left(\frac{t_1^2 + t_2^2}{2}\right)^2 + c_2\left(\frac{t_1^2 + t_2^2}{2}\right)\right\}\,.
	\end{align*}
	Thus, the solutions $t_1^*$ and $t_2^*$ are the non-negative real roots of the following equations
	\begin{align*}
	&-\norm{S_{\lambda\lambda_1}(-{\bf P^k})}_F  + c_1(t_1^2 + t_2^2)t_1 + c_2t_1 = 0\,\\
	&-\norm{S_{\lambda\lambda_1}(-{\bf Q^k})}_F + c_1(t_1^2 + t_2^2)t_2 + c_2t_2 = 0\,.
	\end{align*}
	Set $t_1 = r \norm{S_{\lambda\lambda_1}(-{\bf P^k})}_F$ and $t_2 = r \norm{S_{\lambda\lambda_1}(-{\bf Q^k})}_F$ for some $r \geq 0$. This results in the following cubic equation,
	\[
		c_1\left(\norm{S_{\lambda\lambda_1}(-{\bf Q^k})}_F^2 +  \norm{S_{\lambda\lambda_1}(-{\bf P^k})}_F^2  \right)r^3 + c_2r-1 = 0\,,
	\]
	where the solution is the non-negative real root. 
	\end{proof}
	\subsection{Extensions with Nuclear Norm Regularization}\label{ssec:nuclear-reg-mf}
	We start with the notion of Singular Value Shrinkage Operator \cite{CCS2010}, where given a matrix ${\bf Q} \in \R^{A \times B}$ of rank $K$ with Singular Value Decomposition given by ${\bf U\Sigma V}^T$ with ${\bf U} \in \R^{A \times K}$, ${\bf \Sigma} \in \R^{K \times K}$ and ${\bf V} \in \R^{K \times N}$ for $t\geq 0$ the output is
	\begin{equation}
	{\mathcal D}_{t}({\bf Q}) = {\bf U} \SSS_{t}({\bf \Sigma}){\bf V}^T\,,
	\end{equation} 
	where the soft-thresholding operator is applied only to the singular values. Before we proceed, we require the following technical lemma.
	\begin{lemma}\label{lem:helper-nuclear}
	Let ${\bf Q} \in \R^{A \times B}$ of rank $K$ with Singular Value Decomposition given by ${\bf U\Sigma V}^T$ with ${\bf U} \in \R^{A \times K}$, ${\bf \Sigma} \in \R^{K \times K}$ and ${\bf Z} \in \R^{K \times N}$. Let $t \geq 0$ and $\norm{{{\bf Q}}}_F \neq 0$ then  
	\[
	 \min_{{\bf X} \in \R^{A \times B}}\left\{ \act{{{\bf Q}}, {\bf X}}  + t_0 \norm{{\bf X}}_{\ast}: \norm{{\bf X}}_F^2 \leq t^2 \right\} = -t\norm{\SSS_{t_0}(-{\bf \Sigma})}\,.
	 \]
	 with ${\bf X^*} = t\frac{{\mathcal D}_{t_0}(-{\bf Q})}{\norm{{\mathcal D}_{t}(-{\bf Q})}_F}$ if $\norm{{\mathcal D}_{t_0}(-{\bf Q})} \neq 0$ else any ${\bf X}$ such that $\norm{{\bf X}}_F^2\leq t^2$ is a minimizer. Moreover we have the following equivalence
	 \begin{equation}\label{eq:nuc-eq}
	 \min_{{\bf X} \in \R^{A \times B}}\left\{ \act{{{\bf Q}}, {\bf X}} + t_0 \norm{{\bf X}}_{\ast}: \norm{{\bf X}}_F^2 \leq t^2 \right\} = \min_{{\bf X} \in \R^{A \times B}}\left\{ \act{{{\bf Q}}, {\bf X}} + t_0 \norm{{\bf X}}_{\ast}: \norm{{\bf X}}_F^2 = t^2 \right\}\,.
	 \end{equation}
	\end{lemma}
	\begin{proof}
		The sub-differential  of the nuclear norm \cite{CCS2010} is given by 
		\begin{equation}\label{eq:nuc-subdiff}
		\partial \norm{{\bf X}}_{\ast} = \left\{{\bf UV^{T} + W}: {\bf W} \in \R^{A\times B }, {\bf U^TW} = 0,{\bf WV} = 0, \norm{{\bf W}}_2 \leq 1\right\}\,.
		\end{equation}
		The normal cone for the set $C_1 = \left\{{\bf X}: \norm{{\bf X}}_F^2 \leq t^2 \right\}$ is given by
		\begin{align*}
		{\mathcal N}_{C_1}({\bf {\bar X}}) &= \left\{{\bf V} \in \R^{A \times B} : \act{{\bf V}, {\bf X - {\bar X} }}\leq 0 \text{ for all }  {\bf X} \in C_1 \right\} \equiv \left\{\theta {\bf {\bar X}} : \theta\geq 0\right\}\,.
		\end{align*}
		We consider the following problem
		\[
	 	\min_{{\bf X} \in \R^{A \times B}}\left\{ \act{{{\bf Q}}, {\bf X}}  + t_0 \norm{{\bf X}}_{\ast}: \norm{{\bf X}}_F^2 \leq t^2 \right\}\,.
	 	\]
		and the optimality condition \cite[Theorem 10.1, p. 422]{RW1998-B} results in
		\[
		{\bf 0} \in {\bf Q} + t_0\partial \norm{{\bf X}}_{\ast} + {\mathcal N}_{C_1}({\bf X})\,.
		\]
		We follow the strategy from \cite[Theorem 2.1]{CCS2010}. One can decompose $-{\bf Q}$ as 
		\[
		-{\bf Q} = {\bf U_0 \Sigma_0 V_0^T} + {\bf U_1 \Sigma_1 V_1^T}\,.
		\]
		where ${\bf U_0},{\bf V_0}$ contain the singular vectors for singular values greater than $t_0$ and ${\bf U_1},{\bf V_1}$ for less than equal to $t_0$.	 Then with ${\bf X} = \bf U_0 \Sigma V_0^T$,  the optimality condition becomes
  		\begin{align}
		{\bf 0} = {\bf Q} + t_0 ({\bf U_0V^{T}_0 + W}) + \theta{\bf U_0 \Sigma V^{T}_0}\,,
		\end{align}
		and thus we obtain
		\[
		{\bf U_0 \Sigma_0 V_0^T}  +{\bf U_1 \Sigma_1 V_1^T} =  t_0 \left( {\bf U_0 V_0^T + W}\right) + \theta{\bf U_0 \Sigma V_0^T}\,.
		\]
		With ${\bf W} = t_0^{-1}{\bf U_1 \Sigma_1 V_1^T}$ all the conditions in \eqref{eq:nuc-subdiff} are satisfied. For some unknown $\theta \geq 0$ we have
		\[
		\theta {\bf \Sigma} = {\bf \Sigma_0} - t_0 {\bf I}\,.
		\]
		  The objective  $\act{{{\bf Q}}, {\bf X}}  + t_0 \norm{{\bf X}}_{\ast}$ is now monotonically decreasing with $\theta$ after substituting. Thus, we obtain the solution ${\bf X} = \frac{t}{\norm{{\bf \Sigma_0} - t_0 {\bf I}}} {\bf U_0}\left({\bf \Sigma_0} - t_0 {\bf I}\right){\bf V}_0^T $ for ${\norm{{\bf \Sigma_0} - t_0 {\bf I}}} \neq 0$ else the solution is ${\bf 0}$. The equivalence statement in \eqref{eq:nuc-eq} follows trivially because if ${\norm{{\bf \Sigma_0} - t_0 {\bf I}}} \neq 0$ we have $\norm{{\bf X}}_F^2 = t^2$ otherwise all the points satisfying $\norm{{\bf X}}_F^2 \leq t^2$ are minimizers.
	\end{proof}
	
	Here, we want to solve matrix factorization problem with nuclear norm regularization, where for certain constant $\lambda_2>0$ we want to solve
	\begin{equation}\label{eq:mat-fac-ex-20}
	\min_{{\bf U}\in \R^{M \times K}, {\bf Z}\in \R^{K \times N}}\, \left\{ \Psi({\bf U},{\bf Z}) :=  \frac 12\norm{{\bf A} - {\bf U}{\bf Z}}^2_{F} + \lambda_2\left(\norm{{\bf U}}_{\ast} + \norm{{\bf Z}}_{\ast} \right)   \right\}\,.
	\end{equation}

	Denote $g := \frac 12\norm{{\bf A} - {\bf U}{\bf Z}}^2_{F}$, $f:=   \lambda_2\left(\norm{{\bf U} }_{\ast} + \norm{{\bf Z} }_{\ast} \right)$ and $h = h_a$. 
	\begin{proposition}\label{prop:closed-form-3a}
		In BPG-MF, with the above defined $g,f,h$ the update steps in each iteration are given by ${\bf U}^{{\bf k+1}} = r{\mathcal D}_{\lambda_1\lambda}(-{\bf P^k})$, ${\bf Z}^{{\bf k+1}} = r{\mathcal D}_{\lambda_1\lambda}(-{\bf Q^k})$ where $r\geq 0$ and satisfies
	\begin{align}
		&c_1\left(\norm{{\mathcal D}_{\lambda_1\lambda}\left(-{\bf Q^k}\right)}_F^2 +  \norm{{\mathcal D}_{\lambda_1\lambda}\left(-{\bf P^k}\right)}_F^2  \right)r^3 + c_2r-1 = 0\label{alg:simple-1}\,,
	\end{align}
	with $c_1 = 3$ and $c_2 = \norm{{\bf A}}_F$.
	\end{proposition}
	The proof is similar to Proposition~\ref{prop:closed-form-3} but Lemma~\ref{lem:helper-nuclear} must be used instead of Lemma~\ref{lem:helper-2}.
	\subsection{Extensions with Non-Convex Sparsity Constraints}  We want to solve the matrix factorization problem with non-convex sparsity constraints \cite{BST2014}
	\begin{equation}\label{eq:mat-fac-ex-20}
	\min_{{\bf U}\in \R^{M \times K}, {\bf Z}\in \R^{K \times N}}\, \left\{ \Psi({\bf U},{\bf Z}) :=  \frac 12\norm{{\bf A} - {\bf U}{\bf Z}}^2_{F} : \norm{\bf U}_0 \leq s_1, \norm{\bf Z}_0 \leq s_2,   \right\}\,.
	\end{equation}
	The problem with additional non-negativity constraints, the so called Sparse NMF  is considered in Section~\ref{ssec:sparse-nmf}. Now, denote $g := \frac 12\norm{{\bf A} - {\bf U}{\bf Z}}^2_{F}$, $f:=    {\bf I}_{\norm{\bf U}_0 \leq s_1} +  {\bf I}_{\norm{\bf Z}_0 \leq s_2} $ and $h = h_a$. Note that  the Assumption \ref{A:Assumption0} is not valid here, hence CoCaIn BPG-MF theory does not hold and hints at possible extensions of CoCaIn BPG-MF, which is an interesting open question. Before, we proceed, we require the following concept. Let $y \in \R^d$ and without loss of generality we can assume that $|y_1|\geq |y_2|\geq \ldots\geq |y_d|$, then the {hard-thresholding operator} \cite{LT2013} is given   by
	\begin{equation} \label{eq:soft}
		\HHH_{s}\left(y\right) = \argmin_{x \in \real^{d}} \left\{ \norm{x - y}^{2} :  \norm{x}_0 \leq s \right\} = \begin{cases}
		y_i,  &i\leq s,\\
		0, &\text{otherwise},
		\end{cases}
	\end{equation}
	where $s>0$ and the operations are applied element-wise. We require the following technical lemma.
	\begin{lemma}\label{lem:non-convex-1} Let ${\bf Q} \in \R^{A \times B}$ for some positive integers $A$ and $B$. Let $t \geq 0$ and $\norm{{{\bf Q}}}_F \neq 0$ then  
	\[
	 \min_{{\bf X} \in \R^{A \times B}}\left\{ \act{{{\bf Q}}, {\bf X}} : \norm{{\bf X}}_F^2 \leq t^2,\norm{\bf X}_0 \leq s \right\} = -t\norm{\HHH_{s}(-{\bf Q})}\,.
	 \]
	 with the minimizer ${\bf X^*} = \frac{t\HHH_{s}(-{\bf Q})}{\norm{\HHH_{s}(-{\bf Q})}}$ if $\norm{\HHH_{s}(-{\bf Q})} \neq 0$ else  ${\bf X^*} = {\bf 0}$\,. Moreover  we have the following equivalence
	 \[
	 \min_{{\bf X} \in \R^{A \times B}}\left\{ \act{{{\bf Q}}, {\bf X}} : \norm{{\bf X}}_F^2 \leq t^2,\norm{\bf X}_0 \leq s \right\} = \min_{{\bf X} \in \R^{A \times B}}\left\{ \act{{{\bf Q}}, {\bf X}} : \norm{{\bf X}}_F^2 = t^2,\norm{\bf X}_0 \leq s \right\}\,.
	 \]
	\end{lemma}
	\begin{proof} The proof is similar to \cite[Proposition 11]{LT2013}. We have
	\begin{align*}
	\min_{{\bf X} \in \R^{A \times B}}\left\{ \act{{{\bf Q}}, {\bf X}} : \norm{{\bf X}}_F^2 \leq t^2,\norm{\bf X}_0 \leq s  \right\}	&= -\max_{{\bf X} \in \R^{A \times B}}\left\{ \act{-{{\bf Q}}, {\bf X}} : \norm{{\bf X}}_F^2 \leq t^2,\norm{\bf X}_0 \leq s \right\}\,,\\
	&= -\max_{{\bf X} \in \R^{A \times B}}\left\{ \act{{\mathcal H}_s(-{{\bf Q}}), {\bf X}} : \norm{{\bf X}}_F^2 \leq t^2\right\}\,.
	\end{align*}
	The first equality is a simple rewriting of the objective. Then, the corresponding objective $\act{-{\bf Q},{\bf X}}$ can be maximized with $\sum_{i=1}^A\sum_{j=1}^B {\bf I}_{(i,j) \in \Omega_0} (-{\bf Q}_{ij}{\bf X}_{ij})$ where $\Omega_0$ is set of index pairs and ${\bf I}_{(i,j) \in \Omega_0}$ is $1$ if the index pair if $(i,j)\in \Omega_0$ and zero otherwise.  Note that the objective $\act{-{\bf Q},{\bf X}}$ is maximized if $\Omega_0$ contains all the index pairs corresponding to the elements of $-{\bf Q}$ with highest absolute value which is captured by Hard-thresholding operator. Thus, the second equality follows and the solution follows due to Lemma~\ref{lem:helper-1}. The equivalence statement follows as $\norm{{\bf X}^*}_F^2 = t^2$ for  $\norm{\HHH_{s}(-{\bf Q})} \neq 0$ else the function value is zero and is attained by all the points in the set $\left\{ {\bf X} : \norm{{\bf X}}_F^2 \leq t^2 \right\}$ are minimizers, hence the equivalence.
	\end{proof}
	\begin{proposition}
		In BPG-MF, with the above defined $g,f,h$ the update steps in each iteration are given by ${\bf U}^{{\bf k+1}} = r\HHH_{s_1}(-{\bf P^k})$, ${\bf Z}^{{\bf k+1}} = r\HHH_{s_2}(-{\bf Q^k})$ where $r\geq0$ and satisfies
	\begin{align}
		c_1\left(\norm{\HHH_{s_1}\left(-{\bf Q^k}\right)}_F^2 +  \norm{\HHH_{s_2}\left(-{\bf P^k}\right)}_F^2  \right)r^3 + c_2r-1 = 0\label{alg:simple-1}\,,
	\end{align}
	with $c_1 = 3$ and $c_2 = \norm{{\bf A}}_F$.
	\end{proposition}
	The proof is similar to Proposition~\ref{prop:closed-form-3} but Lemma~\ref{lem:non-convex-1} must be used instead of Lemma~\ref{lem:helper-2}.
\section{Closed Form Solutions Part II for NMF variants}\label{sec:nmf-closed-form}
	For simplicity we consider the following problem \cite{LS1999,LS2001}
	\begin{equation}\label{eq:mat-fac-ex-1a}
	\min_{{\bf U}\in \R^{M \times K}, {\bf Z}\in \R^{K \times N}}\, \left\{ \Psi({\bf U},{\bf Z}) :=  \frac 12\norm{{\bf A} - {\bf U}{\bf Z}}^2_{F}  + {\bf I}_{{\bf U} \geq {\bf 0}} + {\bf I}_{{\bf Z} \geq {\bf 0}}\right\}\,.
	\end{equation}
	We set ${\mathcal R}_1({\bf U}) = 0$, ${\mathcal R}_2({\bf Z}) = 0$, $g = \Psi$ and $f={\bf I}_{{\bf U} \geq {\bf 0}} + {\bf I}_{{\bf Z} \geq {\bf 0}}$ where ${\bf I}$ is the indicator operator. 
	We start with the following technical lemma.
	\begin{lemma}\label{lem:helper-3}Let ${\bf Q} \in \R^{A \times B}$ for some positive integers $A$ and $B$. Let $t \geq 0$ and $\norm{{{\bf Q}}}_F \neq 0$ then
	\[
	 \min_{{\bf X} \in \R^{A \times B}}\left\{ \act{{{\bf Q}}, {\bf X}} : \norm{{\bf X}}_F^2 \leq t^2,{\bf X} \geq {\bf 0} \right\} = -t\norm{\Pi_{+}(-{{\bf Q}})}_F \,,
	 \]
	 with the minimizer ${\bf X}^* = t\frac{ \Pi_{+}(-{{\bf Q}})}{\norm{\Pi_{+}(-{{\bf Q}})}_F}$ if $\norm{\Pi_{+}(-{{\bf Q}})}_F \neq 0$ else ${\bf X}^* = {\bf 0}$. For $\norm{\Pi_{+}(-{{\bf Q}})}_F \neq 0$, we have the following equivalence
	 \begin{equation}\label{eq:equivalence-1}
	 \min_{{\bf X} \in \R^{A \times B}}\left\{ \act{{{\bf Q}}, {\bf X}} : \norm{{\bf X}}_F^2 \leq t^2,{\bf X} \geq {\bf 0} \right\} \equiv \min_{{\bf X} \in \R^{A \times B}}\left\{ \act{{{\bf Q}}, {\bf X}} : \norm{{\bf X}}_F^2 = t^2,{\bf X} \geq {\bf 0} \right\}\,.
	 \end{equation}
	\end{lemma}
	\begin{proof} On rewriting we have the following equivalence
	\[
	\min_{{\bf X} \in \R^{A \times B}}\left\{ \act{{{\bf Q}}, {\bf X}} : \norm{\bf X}_F^2 \leq  t^2,{\bf X} \geq {\bf 0}  \right\} \equiv -\max_{{\bf X} \in \R^{A \times B}}\left\{ \act{-{\bf Q}, {\bf X}} : \norm{\bf X}_F^2 \leq  t^2,{\bf X} \geq {\bf 0}  \right\}\,.
	\]
	The expression $\act{-{\bf Q}, {\bf X}}$ is maximized at ${\bf X^*} = c \Pi_{+}(-{\bf Q}) $ for certain constant $c$. On substituting we have $$\act{-{\bf Q}, {\bf X^*}} = c\norm{ \Pi_{+}(-{\bf Q})}_F^2\,.$$
	Since, the dependence on $c$ is linear and we additionally require $\norm{\bf X}_F^2 \leq  t^2 $, we can set $c=\frac{t}{\norm{\Pi_{+}(-{\bf Q}) }_F}$ if $\norm{\Pi_{+}(-{\bf Q}) }_F \neq 0$ else $c=0$. Hence, the minimizer to
	\[ 
	\min_{{\bf X} \in \R^{A \times B}}\left\{ \act{{{\bf Q}}, {\bf X}} : \norm{\bf X}_F^2 \leq  t^2 \right\}
	\]
	is attained at ${\bf X^*} = -t\frac{\Pi_{+}(-{\bf Q}) }{\norm{\Pi_{+}(-{\bf Q}) }_F}$ for $\norm{\Pi_{+}(-{\bf Q}) }_F \neq 0$ else ${\bf X^*} = 0$. The equivalence in the statement follows as $\norm{{\bf X^*}}_F^2 = t^2$.
	\end{proof}
	
	Denote $g = \Psi$, $f = {\bf I}_{{\bf U} \geq {\bf 0}} + {\bf I}_{{\bf Z} \geq {\bf 0}}$ and $h = h_a$.
	\begin{proposition}\label{prop:main-prop-nmf}
		In BPG-MF, when $g = \Psi$ in \eqref{eq:mat-fac-ex-1a} the update step in each iteration are given by ${\bf U}^{{\bf k+1}} = \Pi_{+}(-{\bf P^k})$, ${\bf Z}^{{\bf k+1}} = \Pi_{+}(-{\bf Q^k})$ where $r\geq 0$ and satisfies
	\begin{align}
		&c_1\left(\norm{\Pi_{+}(-{\bf Q^k})}_F^2 +  \norm{\Pi_{+}(-{\bf P^k})}_F^2  \right)r^3 + c_2r-1 = 0\,.\label{alg:simple-3}\,,
	\end{align}
	with $c_1 = 3$ and $c_2 = \norm{{\bf A}}_F$.
	\end{proposition}
	\begin{proof} The proof is similar to that of Proposition~\ref{prop:closed-form-1}, but with certain changes due to the involved non-negativity constraints for the objective.
	Consider the following subproblem
	\begin{align*}
	  ({\bf U}^{{\bf k+1}}, {\bf Z}^{{\bf k+1}}) \in \underset{({\bf U},{\bf Z}) \in \R^{M \times K}_{+}\times \R^{K\times N}_{+}}{\argmin} &\left\{\act{ {\bf P^k}, {\bf U}} + \act{ {\bf Q^k} , {\bf Z}} \right.\\
	& \left. + c_1 \left(\frac{\norm{\bf U}_F^2 +\norm{\bf Z}_F^2 }{2} \right)^2 + c_2\left( \frac{\norm{\bf U}_F^2 +\norm{\bf Z}_F^2 }{2}\right) \right\}\,.
	\end{align*}
	Denote the objective in the above minimization problem as $\mathcal{O({\bf U}^{\bf k}, {\bf Z}^{\bf k})}$. Now, we show that the following holds
	\begin{align}
		&\min_{({\bf U},{\bf Z}) \in \R^{M \times K}\times \R^{K\times N}} \left(\mathcal{O({\bf U}^{\bf k}, {\bf Z}^{\bf k})}\right) \nonumber\\
		&\equiv \min_{t_1 \geq 0, t_2 \geq 0} \left\{\min_{({\bf U},{\bf Z}) \in \R^{M \times K}\times \R^{K\times N},\norm{\bf U}_F =t_1,\norm{\bf Z}_F =t_2} \left(\mathcal{O({\bf U}^{\bf k}, {\bf Z}^{\bf k})} \right)\right\}\,,\label{eq:equiv-nmf-0}\\
		&\equiv \min_{t_1 \geq 0, t_2 \geq 0} \left\{\min_{({\bf U},{\bf Z}) \in \R^{M \times K}\times \R^{K\times N},\norm{\bf U}_F \leq t_1,\norm{\bf Z}_F \leq t_2} \left(\mathcal{O({\bf U}^{\bf k}, {\bf Z}^{\bf k})} \right)\right\}\label{eq:equiv-nmf-1}\,,
	\end{align}
	where the first step is a simple rewriting of the objective and involved variables and the second equivalence proof is similar to that equivalence of  \eqref{eq:subprob-equiv-2a} and \eqref{eq:subprob-equiv-1a} in Proposition~\ref{prop:closed-form-3}, which we describe now. The second step is non-trivial. In order to prove \eqref{eq:equiv-nmf-1} we rewrite \eqref{eq:equiv-nmf-0} as 
	\begin{align*}
		\min_{t_1 \geq 0, t_2 \geq 0} &\left\{  \min_{{\bf U_1} \in \R^{M \times K}} \left\{ \act{ {\bf P^k}, {\bf U_1}  } : \norm{{\bf U_1}}_F^2 = t_1, {\bf U_1} \geq 0\right\} \right. \\
	& \left. + \min_{{\bf Z_1} \in \R^{K\times N} } \left\{ \act{ {\bf Q^k} , {\bf Z_1}} : \norm{{\bf Z_1}}_F^2 = t_2, {\bf Z_1} \geq 0\right\}\right.\\
	&\left. + c_1\left(\frac{t_1^2 + t_2^2}{2}\right)^2 + c_2\left(\frac{t_1^2 + t_2^2}{2}\right) \right\}\,.
	\end{align*}

	where the second step uses Lemma~\ref{lem:helper-3} and strong convexity of $h$. Now, due to Lemma~\ref{lem:helper-2}, if $\norm{\Pi_{+}(-{{\bf P^k}})}_F \neq 0$ we have
	\begin{align}
		&\min_{{\bf U_1} \in \R^{M \times K}} \left\{ \act{ {\bf P^k}, {\bf U_1}  }  : \norm{{\bf U_1}}_F^2 = t_1\,, {\bf U_1} \geq 0\right\} \equiv \min_{{\bf U_1} \in \R^{M \times K}} \left\{ \act{ {\bf P^k}, {\bf U_1}  } : \norm{{\bf U_1}}_F^2 \leq t_1\,, {\bf U_1} \geq 0\right\} \label{eq:rewrite-1}\,,
	\end{align} 
	and similarly if $\norm{\Pi_{+}(-{{\bf Q^k}})}_F \neq 0$ we have
	\begin{align}
		&\min_{{\bf Z_1} \in \R^{K\times N} } \left\{ \act{ {\bf Q^k} , {\bf Z_1}} : \norm{{\bf Z_1}}_F^2 = t_2\,, {\bf Z_1} \geq 0\right\}\equiv \min_{{\bf Z_1} \in \R^{K\times N} } \left\{ \act{ {\bf Q^k} , {\bf Z_1}} : \norm{{\bf Z_1}}_F^2 \leq t_2\,, {\bf Z_1} \geq 0\right\}\label{eq:rewrite-2}\,.
	\end{align}
	Note that if $\norm{\Pi_{+}(-{{\bf P^k}})}_F = 0$ and $\norm{{\bf P^k}}_F \neq 0$ then the objective 
	\[
	\min_{{\bf U_1} \in \R^{M \times K}} \left\{ \act{ {\bf P^k}, {\bf U_1}  }  : \norm{{\bf U_1}}_F^2 = t_1\,, {\bf U_1} \geq 0\right\}
	\]
	 with minimum function value of a positive value $t_1 \underset{{i \in [M],\, j \in [K]}}{\min} \{({\bf P^k})_{i,j}\}$ where we have $[A] = \{1,2,\ldots,A\}$ for a positive integer $A$. Similarly if $\norm{\Pi_{+}(-{{\bf Q^k}})}_F = 0$ and $\norm{{\bf Q^k}}_F \neq 0$ the minimum function value  for 
	 \[
	 \min_{{\bf Z_1} \in \R^{K\times N} } \left\{ \act{ {\bf Q^k} , {\bf Z_1}} : \norm{{\bf Z_1}}_F^2 = t_2\,, {\bf Z_1} \geq 0\right\}
	 \]
	 is a positive value $t_2\underset{{i \in [K],\, j \in [N]}}{\min}\{ ({\bf Q^k})_{i,j}\}$. Thus for  $\norm{{\bf P^k}}_F \neq 0$ with $\norm{\Pi_{+}(-{{\bf P^k}})}_F = 0$ (or $\norm{{\bf Q^k}}_F \neq 0$ with $\norm{\Pi_{+}(-{{\bf Q^k}})}_F = 0$) the final objective \eqref{eq:equiv-nmf-0} is monotonically increasing in $t_1$ (or $t_2$)  which will drive $t_1$ (or $t_2$) to $0$ due to the constraint $t_1\geq 0$ (or $t_2\geq 0$). So, without loss of generality we can consider $\norm{\Pi_{+}(-{{\bf Q^k}})}_F \neq  0$ and $\norm{\Pi_{+}(-{{\bf Q^k}})}_F = 0$. Now, we obtain the solutions via the following strategy. Denote
	\[
	{{\bf U}^{{\bf *}}_1}(t_1) \in  \argmin \left\{ \act{ {\bf P^k}, {\bf U_1}  } : {\bf U_1} \in \R^{M \times K}_{+}\,, \norm{{\bf U_1}}_F^2 \leq t_1\right\}\,,
	\]
	\[
		{{\bf Z}^{*}_1 }(t_2) \in \argmin \left\{ \act{ {\bf Q^k} , {\bf Z_1}}: {\bf Z_1} \in \R^{K\times N}_{+}\,, \norm{{\bf Z_1}}_F^2 \leq t_2 \right\}\,.
	\]
	Then we obtain  $({\bf U}^{{\bf k+1}}, {\bf Z}^{{\bf k+1}}) =  ({{\bf U}^{{\bf *}}_1}(t_1^*),{{\bf Z}^{*}_1 }(t_2^*))$, where $t_1^*$ and $t_2^*$ are obtained by solving the following two dimensional subproblem
	\begin{align*}
	  (t_1^*,t_2^*) \in  \underset{t_1\geq 0, t_2\geq 0}{\mathrm{argmin}}   &\left\{  \min_{{\bf U_1} \in \R^{M \times K}_{+}} \left\{ \act{ {\bf P^k}, {\bf U_1}  } : \norm{{\bf U_1}}_F^2 \leq t_1\right\} \right. \\
	& \left. + \min_{{\bf Z_1} \in \R^{K\times N}_{+} } \left\{ \act{ {\bf Q^k} , {\bf Z_1}}: \norm{{\bf Z_1}}_F^2 \leq t_2\right\} \right.\\
	&\left. + c_1\left(\frac{t_1 + t_2}{2}\right)^2 + c_2\left(\frac{t_1 + t_2}{2}\right) \right\}\,.
	\end{align*}
	Note that inner minimization subproblems can be trivially solved once we obtain ${{\bf U}^{{\bf *}}_1}(t_1)$ and ${{\bf Z}^{*}_1 }(t_2)$. Due to Lemma~\ref{lem:helper-3} we obtain the solution to the subproblem in each iteration as follows
	\begin{align*}
	{\bf U}^{{\bf k+1}} 
	&= \left.
	\begin{cases}
	    t_1^*\frac{\Pi_{+}(-{\bf P^k})}{\norm{\Pi_{+}(-{\bf P^k})}_F}, & \text{for } \norm{\Pi_{+}(-{\bf P^k})}_F \neq 0\,, \\
	    {\bf 0}, &  otherwise\,.
	\end{cases}\right.\\
	{\bf Z}^{{\bf k+1}} 
	&= \left.
	\begin{cases}
	    t_2^*\frac{\Pi_{+}(-{\bf Q^k})}{\norm{\Pi_{+}(-{\bf Q^k})}_F}, & \text{for } \norm{\Pi_{+}(-{\bf Q^k})}_F \neq 0\,, \\
	    {\bf 0}, &  otherwise\,.
	\end{cases}
	 \right.
	\end{align*}
	We solve for $t_1^*$ and $t_2^*$ with the following two dimensional minimization problem
	\begin{align*}
	\underset{t_1\geq 0, t_2\geq 0}{\mathrm{argmin}} & \left\{  -t_1\norm{\Pi_{+}(-{\bf P^k})}_F  - t_2{\norm{\Pi_{+}(-{\bf Q^k})}_F}   + c_1\left(\frac{t_1^2 + t_2^2}{2}\right)^2 + c_2\left(\frac{t_1^2 + t_2^2}{2}\right)\right\}\,.
	\end{align*}
	Thus, the solutions $t_1^*$ and $t_2^*$ are the non-negative real roots of the following equations
	\begin{align*}
	&-\norm{\Pi_{+}(-{\bf P^k})}_F  + c_1(t_1^2 + t_2^2)t_1 + c_2t_1 = 0\,,\\
	&-\norm{\Pi_{+}(-{\bf Q^k})}_F + c_1(t_1^2 + t_2^2)t_2 + c_2t_2 = 0\,.
	\end{align*}
	Further simplifications lead to $t_1 = r \norm{\Pi_{+}(-{\bf P^k})}_F$ and $t_2 = r \norm{\Pi_{+}(-{\bf Q^k})}_F$ for some $r \geq 0$. This results in the following cubic equation,
	\[
		c_1\left(\norm{\Pi_{+}(-{\bf Q^k})}_F^2 +  \norm{\Pi_{+}(-{\bf P^k})}_F^2  \right)r^3 + c_2r-1 = 0\,,
	\]
	where the solution is the non-negative real root. 
	\end{proof}
	\subsection{Extensions to L2-regularized NMF}\label{ssec:l2-reg-nmf} Here, the goal is solve the following minimization problem
	\begin{equation*}
	\min_{{\bf U}\in \R^{M \times K}, {\bf Z}\in \R^{K \times N}}\, \left\{ \Psi({\bf U},{\bf Z}) :=  \frac 12\norm{{\bf A} - {\bf U}{\bf Z}}^2_{F} + \frac{\lambda_0}{2}\left(\norm{{\bf U} }_F^2 + \norm{{\bf Z} }_F^2 \right) + {\bf I}_{{\bf U} \geq {\bf 0}} + {\bf I}_{{\bf Z} \geq {\bf 0}} \right\}\,.
	\end{equation*}
	Denote $g := \frac 12\norm{{\bf A} - {\bf U}{\bf Z}}^2_{F} + \frac{\lambda_0}{2}\left(\norm{{\bf U} }_F^2 + \norm{{\bf Z} }_F^2 \right)$, $f:=  {\bf I}_{{\bf U} \geq {\bf 0}} + {\bf I}_{{\bf Z} \geq {\bf 0}}$ and $h = h_b$.
	\begin{proposition}
		In BPG-MF, with above defined $g,f,h$ the update step in each iteration are given by ${\bf U}^{{\bf k+1}} = \Pi_{+}(-{\bf P^k})$, ${\bf Z}^{{\bf k+1}} = \Pi_{+}(-{\bf Q^k})$ where $r\geq 0$ and satisfies
	\begin{align*}
		&c_1\left(\norm{\Pi_{+}(-{\bf Q^k})}_F^2 +  \norm{\Pi_{+}(-{\bf P^k})}_F^2  \right)r^3 + (c_2 + \lambda_0)r-1 = 0\,,\label{alg:simple-3}
	\end{align*}
	with $c_1 = 3$ and $c_2 = \norm{{\bf A}}_F$.
	\end{proposition}
	The proof is similar to Proposition~\ref{prop:main-prop-nmf} with only change in $c_2$.
	\subsection{Extensions to L1-regularized NMF}\label{ssec:l1-reg-nmf} Here, the goal is solve the following minimization problem
	\begin{equation*}
		\min_{{\bf U}\in \R^{M \times K}, {\bf Z}\in \R^{K \times N}}\, \left\{ \Psi({\bf U},{\bf Z}) :=  \frac 12\norm{{\bf A} - {\bf U}{\bf Z}}^2_{F} + \lambda_1\left(\norm{{\bf U} }_1 + \norm{{\bf Z} }_1  \right)   + {\bf I}_{{\bf U} \geq {\bf 0}} + {\bf I}_{{\bf Z} \geq {\bf 0}} \right\}\,.
	\end{equation*}
	We denote ${\bf e_D}$ to be a vector of dimension ${\bf D}$ with all its elements set to 1. 
	\begin{lemma}\label{lem:helper-4a}Let ${\bf Q} \in \R^{A \times B}$ for some positive integers $A$ and $B$. Let $t \geq 0$ and $\norm{{{\bf Q}}}_F \neq 0$ then
	\[
	 \min_{{\bf X} \in \R^{A \times B}}\left\{ \act{{{\bf Q}}, {\bf X}} + t_0\norm{{\bf X}}_1 : \norm{{\bf X}}_F^2 \leq t^2,{\bf X} \geq 0 \right\} = -t\norm{\Pi_{+}(-\left({{\bf Q }+t_0{\bf e_{A}e_{B}}^T}\right))}_F 
	 \]
	with  the minimizer ${\bf X}^* = t\frac{\Pi_{+}(-\left({{\bf Q }+t_0{\bf e_{A}e_{B}}^T}\right))}{\norm{\Pi_{+}(-\left({{\bf Q }+t_0{\bf e_{A}e_{B}}^T}\right))}_F}$ if the condition $\norm{\Pi_{+}(-\left({{\bf Q }+t_0{\bf e_{A}e_{B}}^T}\right))}_F \neq 0$ holds\,.
	\end{lemma}
	\begin{proof}
	By using ${\bf X }\geq 0$ and the basic trace properties we have the following equivalence
	\[
		\norm{{\bf X}}_1 = \sum_{i,j}{\bf X}_{ij} =  {\bf e_{A}}^T{\bf X e_{B}} = {\bf tr}\left({\bf e_{A}}^T{\bf X e_{B}}\right) = {\bf tr}\left({\bf e_{B}e_{A}}^T{\bf X }\right) = \act{{\bf e_{A}e_{B}}^T,{\bf X}}\,,
	\]
	hence we have the following equivalence
	\begin{align*}
	&\min_{{\bf X} \in \R^{A \times B}}\left\{ \act{{{\bf Q}}, {\bf X}} + t_0\norm{{\bf X}}_1 : \norm{{\bf X}}_F^2\leq t^2,{\bf X} \geq 0 \right\}\\
	& \equiv 	\min_{{\bf X} \in \R^{A \times B}}\left\{ \act{{{\bf Q }+t_0{\bf e_{A}e_{B}}^T}, {\bf X}}  : \norm{{\bf X}}_F^2 \leq t^2,{\bf X} \geq 0 \right\}\,
	\end{align*}
	Now, the solution follows due to Lemma~\ref{lem:helper-3}.
	\end{proof}
	Denote $g := \frac 12\norm{{\bf A} - {\bf U}{\bf Z}}^2_{F}$, $f:= \lambda_1\left(\norm{{\bf U} }_1 + \norm{{\bf Z} }_1  \right)  + {\bf I}_{{\bf U} \geq {\bf 0}} + {\bf I}_{{\bf Z} \geq {\bf 0}}$ and $h=h_a$.
	\begin{proposition}
		In BPG-MF, with the above defined $g,f,h$ the update steps in each iteration are given by ${\bf U}^{{\bf k+1}} = r\Pi_{+}(-\left({{\bf P^k}+t_0{\bf e}_{M}{\bf e}_{K}^T}\right))$, 
		${\bf Z}^{{\bf k+1}} = r\Pi_{+}(-\left({{\bf Q^k}+t_0{\bf e}_{K}{\bf e}_{N}^T}\right))$ where $r\geq0$ and satisfies
	\begin{align*}
		\begin{split}
		&c_1\left(\norm{\Pi_{+}(-\left({{\bf P^k}+t_0{\bf e}_{M}{\bf e}_{K}^T}\right))}_F^2  +  \norm{\Pi_{+}(-\left({{\bf Q^k}+t_0{\bf e}_{K}{\bf e}_{N}^T}\right))}_F^2  \right)r^3+ c_2r-1 = 0\,,
		\end{split}
	\end{align*}
	with $c_1 = 3$, $c_2 = \norm{{\bf A}}_F$ and $t_0 = \lambda\lambda_1$.
	\end{proposition}
	We skip the proof as it is similar to Proposition~\ref{prop:main-prop-nmf}.
	
	\subsection{Extensions to Graph Regularized Non-negative Matrix Factorization}\label{ssec:exten-gnmf}
	Graph Regularized Non-negative Matrix Factorization was proposed in \cite{CHHH2011}. 
	Here, given  ${\bf \mathcal L} \in \R^{M \times M}$ we are interested to solve
	\begin{align*}
	\min_{{\bf U}\in \R^{M \times K}, {\bf Z}\in \R^{K \times N}}\, \left\{ \Psi({\bf U},{\bf Z}) \right. &=  \frac 12\norm{{\bf A} - {\bf U}{\bf Z}}^2_{F} +\frac{\mu_0}{2}{\bf tr}({\bf U}^T{\bf \mathcal L}{\bf U})  \\
	&\left. + \frac{\lambda_0}{2}\left(\norm{{\bf U} }_F^2 + \norm{{\bf Z} }_F^2 \right) + {\bf I}_{{\bf U} \geq {\bf 0}} + {\bf I}_{{\bf Z} \geq {\bf 0}} \right\}.
	\end{align*}
	Recall that
	\[
	h_c({\bf U},{\bf Z}) =  3h_1({\bf U},{\bf Z}) + \left(\norm{{\bf A}}_F + \mu_0\norm{\bf \mathcal L}_F\right) h_2({\bf U},{\bf Z}) \,.
	\]
	Denote $g := \frac 12\norm{{\bf A} - {\bf U}{\bf Z}}^2_{F} +\frac{\mu_0}{2}{\bf tr}({\bf U}^T{\bf \mathcal L}{\bf U})$, $f:=  \frac{\lambda_0}{2}\left(\norm{{\bf U} }_F^2 + \norm{{\bf Z} }_F^2 \right) + {\bf I}_{{\bf U} \geq {\bf 0}} + {\bf I}_{{\bf Z} \geq {\bf 0}}$ and $h = h_c$.
	\begin{proposition}\label{prop:closed-form-2a}
	In BPG-MF, with the above defined $f,g,h$ the update steps in each iteration are given by ${\bf U}^{{\bf k+1}} = r \Pi_{+}(-\,{\bf P^k})$, ${\bf Z}^{{\bf k+1}} = r\Pi_{+}(-\,{\bf Q^k})$ where $r\geq 0$ and satisfies
	\begin{align}
		&c_1\left(\norm{\Pi_{+}(-{\bf Q^k})}_F^2 +  \norm{\Pi_{+}(-{\bf P^k})}_F^2  \right)r^3 + (c_2+\mu_0\norm{\bf \mathcal L}_F + \lambda_0)r-1 = 0\,,\label{alg:simple-0}
	\end{align}
	with $c_1 = 3$ and $c_2 = \norm{{\bf A}}_F$.
	\end{proposition}
	The proof is similar to Proposition~\ref{prop:main-prop-nmf} and only $c_2$ changes.

	\subsection{Extensions to Symmetric NMF via Non-Symmetric Relaxation.}\label{ssec:symnf-nmf} In \cite{ZLLL2018}, the following optimization problem was proposed in the context of Symmetric NMF where the factors ${\bf U}$ and ${\bf Z}^T$ are equal. The symmetricity of the factors was lifted via a quadratic penalty terms resulting in the following problem
	\begin{equation*}
	\min_{{\bf U}\in \R^{M \times K}, {\bf Z}\in \R^{K \times N}}\, \left\{ \Psi({\bf U},{\bf Z}) :=  \frac 12\norm{{\bf A} - {\bf U}{\bf Z}}^2_{F} + \frac{\lambda_0}{2}\norm{{\bf U- Z}^T }_F^2 + {\bf I}_{{\bf U} \geq {\bf 0}} + {\bf I}_{{\bf Z} \geq {\bf 0}} \right\}\,.
	\end{equation*}

	Now, we prove the $L$-smad property. We need the following technical lemma.
	\begin{lemma}\label{lem:technical-10}
	Let $g({\bf U}, {\bf Z}) = \frac 12\norm{{\bf A} - {\bf U}{\bf Z}}^2_{F} +\frac{\lambda_0}{2}\norm{{\bf U-Z}^T }_F^2 $ be as defined above, we have the following
	\begin{align*}
	\nabla_{\bf U} g({\bf A},{\bf U}{\bf Z}) &= \lambda_0\left({\bf U-Z}^T\right)-({\bf A} - {\bf U}{\bf Z}){\bf Z}^T\\
	\nabla_{\bf Z} g({\bf A},{\bf U}{\bf Z}) &= \lambda_0\left({\bf U-Z}^T\right)+ {\bf U}^T({\bf A} - {\bf U}{\bf Z})
	\end{align*}
	and 
	\begin{align*}
	&\act{({\bf H}_{{\bf 1}},{\bf H}_{{\bf 2}}), \nabla^2 g({\bf A},{\bf U}{\bf Z}) ({\bf H}_{{\bf 1}}, {\bf H}_{{\bf 2}})} \\
	&= -2\act{{\bf A} - {\bf U}{\bf Z}, {\bf H}_{{\bf 1}}{\bf H}_{{\bf 2}}} + \norm{{\bf U}{\bf H}_{{\bf 2}} +  {\bf H}_{{\bf 1}}{\bf Z}}_F^2  + \lambda_0 \norm{{\bf H_1}-{\bf H_2}^T}_F^2 \,.
	\end{align*}
	\end{lemma}
	\begin{proof}
	The first part of proof for function $\frac 12\norm{{\bf A} - {\bf U}{\bf Z}}^2_{F}$ follows from Proposition~\ref{prop:l-smad-1}. For the other term, with the Forbenius dot product, we obtain
	\begin{align*}
	&\frac{\lambda_0}{2}\norm{{\bf U+ H_1- Z}^T-{\bf H_2}^T }_F^2\\
	&= \frac{\lambda_0}{2}  \left(\norm{{\bf U-Z}^T}_F^2 + 2\act{{\bf U-Z}^T,{\bf H_1 - H_2}^T} + \norm{{\bf H_1}-{\bf H_2}^T}_F^2 \right)\,.
	\end{align*}
	Combining with  Lemma~\ref{lem:technical-1}, the statement follows from the collecting the first order and second order terms. 
	\end{proof}
	
	\begin{proposition}\label{prop:l-smad-2a0}
	Let $g({\bf U}, {\bf Z}) = \frac 12\norm{{\bf A} - {\bf U}{\bf Z}}^2_{F} +\frac{\lambda_0}{2}\norm{{\bf U-\bf Z} }_F^2 $. Then, for a certain constant $L\geq 1$, the function $g$ satisfies $L$-smad property with respect to the following kernel generating distance,
	 \[
	 h_d({\bf U},{\bf Z}) =  3h_1({\bf U},{\bf Z}) + \left(\norm{{\bf A}}_F + 2\lambda_0\right) h_2({\bf U},{\bf Z}) \,.
	 \]
	\end{proposition}
	\begin{proof}
		The proof is similar to Proposition~\ref{prop:l-smad-1} and Lemma~\ref{lem:technical-10} must be applied for the result.
	\end{proof}
	Denote $g :=\frac 12\norm{{\bf A} - {\bf U}{\bf Z}}^2_{F} +\frac{\lambda_0}{2}\norm{{\bf U-\bf Z} }_F^2 $, $f:= {\bf I}_{{\bf U} \geq {\bf 0}} + {\bf I}_{{\bf Z} \geq {\bf 0}} $ and $h = h_d$.
	\begin{proposition}\label{prop:closed-form-2a0}
	In BPG-MF, with the above defined update steps in each iteration are given by ${\bf U}^{{\bf k+1}} = r \Pi_{+}\left(-\,{\bf P^k}\right)$, ${\bf Z}^{{\bf k+1}} = r\Pi_{+}\left(-\,{\bf Q^k}\right)$ where $r\geq 0$ and satisfies
	\begin{align}
		&c_1\left(\norm{\Pi_{+}\left(-\,{\bf P^k}\right)}_F^2 +  \norm{\Pi_{+}\left(-\,{\bf Q^k}\right)}_F^2  \right)r^3 + (c_2+2\lambda_0)r-1 = 0\label{alg:simple-0}\,,
	\end{align}
	with $c_1 = 3$ and $c_2 = \norm{{\bf A}}_F$.
	\end{proposition}
	The proof is similar to Proposition~\ref{prop:main-prop-nmf} and only $c_2$ changes.
	\subsection{Extensions to NMF with Non-Convex Sparsity Constraints (Sparse NMF)}\label{ssec:sparse-nmf}
	 Consider the following problem from \cite{BST2014}
	\begin{equation*}
	\min_{{\bf U}\in \R^{M \times K}, {\bf Z}\in \R^{K \times N}}\, \left\{ \Psi({\bf U},{\bf Z}) :=  \frac 12\norm{{\bf A} - {\bf U}{\bf Z}}^2_{F} : {\bf U}\geq {\bf 0}, \norm{\bf U}_0 \leq s_1, {\bf Z}\geq {\bf 0}, \norm{\bf Z}_0 \leq s_2,   \right\}\,,
	\end{equation*}
	where $s_1$ and $s_2$ are two known positive integers. Denote $g := \frac 12\norm{{\bf A} - {\bf U}{\bf Z}}^2_{F}$, $f:=  {\bf I}_{{\bf U} \geq {\bf 0}} +  {\bf I}_{\norm{\bf U}_0 \leq s_1} + {\bf I}_{{\bf Z} \geq {\bf 0}} +  {\bf I}_{\norm{\bf Z}_0 \leq s_2} $  and $h = h_a$.  Note that  the Assumption \ref{A:Assumption0} is not valid here, hence CoCaIn BPG-MF theory does not hold and hints at possible extensions of CoCaIn BPG-MF, which is an interesting open question. We start with the following technical lemma.
	\begin{proposition} Let ${\bf Q} \in \R^{A \times B}$ for some positive integers $A$ and $B$. Let $t \geq 0$ and $\norm{{{\bf Q}}}_F \neq 0$ then  
	\[
	 \min_{{\bf X} \in \R^{A \times B}}\left\{ \act{{{\bf Q}}, {\bf X}} : \norm{{\bf X}}_F^2 \leq t^2,\norm{\bf X}_0 \leq s ,{\bf X} \geq {\bf 0} \right\} = -t \norm{{\mathcal H}_s(\Pi_{+}(-{{\bf Q}}))}_F\,.
	 \]
	 with the minimizer ${\bf X^*} = t\frac{{\mathcal H}_s(\Pi_{+}(-{{\bf Q}}))}{\norm{{\mathcal H}_s(\Pi_{+}(-{{\bf Q}}))}_F}$ if $\norm{{\mathcal H}_s(\Pi_{+}(-{{\bf Q}}))}_F \neq 0$ else ${\bf X^*} ={\bf 0}$. If $\norm{{\mathcal H}_s(\Pi_{+}(-{{\bf Q}}))}_F \neq 0$ we have the following equivalence
	 \begin{align}
	  &\min_{{\bf X} \in \R^{A \times B}}\left\{ \act{{{\bf Q}}, {\bf X}} : \norm{{\bf X}}_F^2 \leq t^2,\norm{\bf X}_0 \leq s ,{\bf X} \geq {\bf 0} \right\}\\
	  &\equiv  \min_{{\bf X} \in \R^{A \times B}}\left\{ \act{{{\bf Q}}, {\bf X}} : \norm{{\bf X}}_F^2 = t^2,\norm{\bf X}_0 \leq s ,{\bf X} \geq {\bf 0} \right\}
	 \end{align}
	\end{proposition}
	\begin{proof}
	We have 
	\begin{align*}
	&\min_{{\bf X}}\left\{ \act{{{\bf Q}}, {\bf X}} : \norm{{\bf X}}_F^2 \leq t^2,\norm{\bf X}_0 \leq s ,{\bf X} \geq {\bf 0} \right\} \\
	&= -\max_{{\bf X}}\left\{ \act{-{{\bf Q}}, {\bf X}} : \norm{{\bf X}}_F^2 \leq t^2,\norm{\bf X}_0 \leq s ,{\bf X} \geq {\bf 0} \right\}\,,\\
	&= -\max_{{\bf X}}\left\{ \act{\Pi_{+}(-{{\bf Q}}), {\bf X}} : \norm{{\bf X}}_F^2 \leq t^2, \norm{\bf X}_0 \leq s  \right\}\,,\\
	&= -\max_{{\bf X}}\left\{\act{{\mathcal H}_s (\Pi_{+}(-{{\bf Q}})), {\bf X}} : \norm{{\bf X}}_F^2 \leq t^2 \right\}\,.
	\end{align*}
	The first equality is a simple rewriting of the objective. Then, the corresponding objective $\act{-{\bf Q},{\bf X}}$ can be maximized with $\sum_{i=1}^A\sum_{j=1}^B {\bf I}_{(i,j) \in \Omega_0} (-{\bf Q}_{ij}{\bf X}_{ij})$ where $\Omega_0$ is set of index pairs and ${\bf I}_{(i,j) \in \Omega_0}$ is $1$ if the index pair if $(i,j)\in \Omega_0$ and zero otherwise.  It is easy to see that the objective $\act{-{\bf Q},{\bf X}}$ is maximized if $\Omega_0$ contains all the index pairs corresponding to the elements of $-{\bf Q}$ with highest absolute value which is captured by Hard-thresholding operator. However due to the non-negativity constraint if there is any $-{\bf Q}_{ij}$ such that it is negative, then since ${\bf X}_{ij}$ will be driven to zero. So, before we use the Hard-thresholding operator, we need to use $\Pi_{+}(.) = \max\{0,.\}$  in second equality.   The third equality follows as a consequence of hard sparsity constraint similar to Lemma~\ref{lem:non-convex-1} and the solution follows due to Lemma~\ref{lem:helper-1}. The equivalence statement follows as $\norm{{\bf X}^*}_F^2 = t^2$.
	\end{proof}
	\begin{proposition}
		In BPG-MF, with the above defined $g,f,h$ the update steps in each iteration are ${\bf U}^{{\bf k+1}} = r\HHH_{s_1}(\Pi_{+}(-{\bf P^k}))$, ${\bf Z}^{{\bf k+1}} = r\HHH_{s_2}(\Pi_{+}(-{\bf Q^k}))$ where $r\geq0$ and satisfies
	\begin{align*}
		c_1\left(\norm{\HHH_{s_1}\left(\Pi_{+}(-{\bf Q^k})\right)}_F^2 +  \norm{\HHH_{s_2}\left(\Pi_{+}(-{\bf P^k})\right)}_F^2  \right)r^3 + c_2r-1 = 0\,,
	\end{align*}
	with $c_1 = 3$ and $c_2 = \norm{{\bf A}}_F$.
	\end{proposition}
	The proof is similar to Proposition~\ref{prop:main-prop-nmf}.

\section{Matrix Completion Problem}\label{sec:matrix-completion}
	Matrix Completion is an important non-convex optimization problem, which arises in practical real world applications, such as recommender systems \cite{KRV2009,CR2009,FZSH2017}. Give a matrix ${\bf A}$ where only the values at the index set given by $\Omega$ are given. The goal is obtain the rest of the values. One of the popular strategy is to obtain the factors ${\bf U} \in \R^{M \times K}$ and ${\bf Z} \in \R^{K \times N}$ for a small positive integer $K$. This is cast into the following problem,
	\begin{equation}\label{eq:mat-com-ex-2}
	\min_{{\bf U}\in \R^{M \times K}, {\bf Z}\in \R^{K \times N}}\, \left\{ \Psi({\bf U},{\bf Z}) :=  \frac 12\norm{P_{\Omega}\left({\bf A} - {\bf U}{\bf Z}\right)}^2_{F} + \frac{\lambda_0}{2}\left(\norm{{\bf U} }_F^2 + \norm{{\bf Z} }_F^2 \right)  \right\}\,,
	\end{equation}
	where $P_{\Omega}$ is an masking operator over index set $\Omega$ which preserves the given matrix entries and sets others to zero.. We require the following technical lemma.
	\begin{lemma}\label{lem:technical-1}
	Let $g := \frac 12\norm{P_{\Omega}\left({\bf A} - {\bf U}{\bf Z}\right)}^2_{F}$ be as defined above, we have the following
	\[
		\nabla_{\bf U} g({\bf A},{\bf U}{\bf Z}) = -P_{\Omega}({\bf A} - {\bf U}{\bf Z}){\bf Z}^T, \quad \nabla_{\bf Z} g({\bf A},{\bf U}{\bf Z}) = -{\bf U}^TP_{\Omega}({\bf A} - {\bf U}{\bf Z})
	\]
	\[
	\act{({\bf H}_{{\bf 1}},{\bf H}_{{\bf 2}}), \nabla^2 g({\bf A},{\bf U}{\bf Z}) ({\bf H}_{{\bf 1}}, {\bf H}_{{\bf 2}})} = \norm{P_{\Omega}({\bf U}{\bf H}_{{\bf 2}} +{\bf H}_{{\bf 1}}{\bf Z})}_F^2- 2\act{P_{\Omega}({\bf A} - {\bf U}{\bf Z}), {\bf H}_{{\bf 1}}{\bf H}_{{\bf 2}}}\,.
	\]
	\end{lemma}
	\begin{proof}
	With the Forbenius dot product, we have
	\begin{align*}
	\norm{P_{\Omega}({\bf A} - {\bf U}{\bf Z})}^2_{F} &= \act{P_{\Omega}({\bf A} - {\bf U}{\bf Z}), P_{\Omega}({\bf A} - {\bf U}{\bf Z})}\,.
	\end{align*}
	In the above expression by substituting ${\bf U}$ with ${\bf U}+{\bf H}_{{\bf 1}}$ and ${\bf Z}$ with ${\bf Z}+{\bf H}_{{\bf 2}}$, we obtain 
	\begin{align*}
	 &\act{P_{\Omega}({\bf A} - ({\bf U}+{\bf H}_{{\bf 1}})({\bf Z} + {\bf H}_{{\bf 2}})), P_{\Omega}({\bf A} -  ({\bf U}+{\bf H}_{{\bf 1}})({\bf Z} + {\bf H}_{{\bf 2}}))}\,, \\
	 & = \norm{P_{\Omega}({\bf A} - {\bf U}{\bf Z})}_F^2 + \norm{P_{\Omega}({\bf U}{\bf H}_{{\bf 2}} +{\bf H}_{{\bf 1}}{\bf Z})}_F^2 \\
	 & - 2 \act{P_{\Omega}({\bf A} - {\bf U}{\bf Z}), P_{\Omega}({\bf U}{\bf H}_{{\bf 2}} +{\bf H}_{{\bf 1}}{\bf Z})}  - 2\act{P_{\Omega}({\bf A} - {\bf U}{\bf Z}), P_{\Omega}({\bf H}_{{\bf 1}}{\bf H}_{{\bf 2}})}
	\end{align*}
	where in the last term we ignored the terms higher than second order.
	Collecting all the first order terms we have
	\begin{align*}
		&- 2 \act{P_{\Omega}({\bf A} - {\bf U}{\bf Z}), P_{\Omega}({\bf U}{\bf H}_{{\bf 2}} +{\bf H}_{{\bf 1}}{\bf Z})} \\
		&= - 2 \act{P_{\Omega}({\bf A} - {\bf U}{\bf Z}), {\bf U}{\bf H}_{{\bf 2}} +{\bf H}_{{\bf 1}}{\bf Z}}\\
		&= -2\act{P_{\Omega}({\bf A} - {\bf U}{\bf Z}){\bf Z}^T, {\bf H_1}} -2\act{{\bf U}^TP_{\Omega}({\bf A} - {\bf U}{\bf Z}),{\bf H_2}}
	\end{align*}
	and similarly collecting all the second order terms we have
	\begin{align*}
	 &\norm{P_{\Omega}({\bf U}{\bf H}_{{\bf 2}} +{\bf H}_{{\bf 1}}{\bf Z})}_F^2- 2\act{P_{\Omega}({\bf A} - {\bf U}{\bf Z}), P_{\Omega}({\bf H}_{{\bf 1}}{\bf H}_{{\bf 2}})}\\
	 &=\norm{P_{\Omega}({\bf U}{\bf H}_{{\bf 2}} +{\bf H}_{{\bf 1}}{\bf Z})}_F^2- 2\act{P_{\Omega}({\bf A} - {\bf U}{\bf Z}), {\bf H}_{{\bf 1}}{\bf H}_{{\bf 2}}}
	\end{align*}
	Thus the statement follows using the second order Taylor expansion.
	\end{proof}
	\begin{proposition}
	Let $g:= \frac 12\norm{P_{\Omega}\left({\bf A} - {\bf U}{\bf Z}\right)}^2_{F}$ and $h_1,h_2$ be as defined as in \eqref{eq:breg-funcs}. Then, for a certain constant $L\geq 1$, the function $g$ satisfies $L$-smad property with respect to the following kernel generating distance,
	 \[
	 h_a({\bf U},{\bf Z}) =  3h_1({\bf U},{\bf Z}) + \norm{P_{\Omega}({\bf A})}_F h_2({\bf U},{\bf Z}) \,.
	 \]
	\end{proposition}
	\begin{proof}
	With  Lemma~\ref{lem:technical-1} we obtain  
	 \begin{align*}
	& \act{({\bf H}_1,{\bf H}_2), \nabla^2 g(A,{\bf U}{\bf Z}) ({\bf H}_1,{\bf H}_2)} \\
	&= \norm{P_{\Omega}({\bf U}{\bf H}_{{\bf 2}} +{\bf H}_{{\bf 1}}{\bf Z})}_F^2- 2\act{P_{\Omega}({\bf A} - {\bf U}{\bf Z}), {\bf H}_{{\bf 1}}{\bf H}_{{\bf 2}}}\\
	&\leq  \norm{{\bf H}_{{\bf 1}}{\bf Z} + {\bf U}{\bf H}_{{\bf 2}}}_F^2- 2\act{P_{\Omega}({\bf A} - {\bf U}{\bf Z}), {\bf H}_{{\bf 1}}{\bf H}_{{\bf 2}}}\\
	& \leq 2 \norm{{\bf H}_1{\bf Z}}_F^2 + 2\norm{{\bf U}{\bf H}_2}_F^2  + 2 \norm{P_{\Omega}({\bf A})}_F\norm{{\bf H}_1{\bf H}_2}_F + 2\norm{P_{\Omega}({\bf U}{\bf Z})}_F\norm{{\bf H}_1{\bf H}_2}_F\,,\\
	& \leq 2 \norm{{\bf H}_1{\bf Z}}_F^2 + 2\norm{{\bf U}{\bf H}_2}_F^2  + 2 \norm{P_{\Omega}({\bf A})}_F\norm{{\bf H}_1{\bf H}_2}_F + 2\norm{{\bf U}{\bf Z}}_F\norm{{\bf H}_1{\bf H}_2}_F\,.
	 \end{align*}
	 The rest of the proof is similar to Proposition~\ref{prop:l-smad-1}.
	\end{proof}
	\begin{proposition}
	Let $g:= \frac 12\norm{P_{\Omega}\left({\bf A} - {\bf U}{\bf Z}\right)}^2_{F} + \frac{\lambda_0}{2}\left(\norm{{\bf U} }_F^2 + \norm{{\bf Z} }_F^2 \right)$ and $h_1,h_2$ be as defined as in \eqref{eq:breg-funcs}. Then, for a certain constant $L\geq 1$, the function $g$ satisfies $L$-smad property with respect to the following kernel generating distance,
	 \[
	 h_a({\bf U},{\bf Z}) =  3h_1({\bf U},{\bf Z}) + \left( \norm{P_{\Omega}({\bf A})}_F + \lambda_0\right) h_2({\bf U},{\bf Z}) \,.
	 \]
	\end{proposition}
	The update steps are very similar as what we described earlier in Section~\ref{sec:closed-forms} and \ref{sec:nmf-closed-form}.

\section{Closed Form Solution with 5th-order Polynomial}\label{sec:closed-5-order}
	The goal of this section is to show a case, where while obtaining the update step of BPG-MF we obtain a 5th order polynomial equation, for which Newton based method solvers can be used. We later show that we can obtain a cubic equation by slightly modifying the kernel generating distance. Let $\lambda_0 >0$ and we consider the following problem
	\begin{equation}\label{eq:mat-fac-ex-1c}
	\min_{{\bf U}\in \R^{M \times K}, {\bf Z}\in \R^{K \times N}}\, \left\{ \Psi({\bf U},{\bf Z}) :=  \frac 12\norm{{\bf A} - {\bf U}{\bf Z}}^2_{F}  + \frac{\lambda_0}{2} \norm{{\bf U}}_F^2  \right\}\,.
	\end{equation}
	We set ${\mathcal R}_1({\bf U}) = \frac{\lambda_0}{2} \norm{{\bf U}}_F^2$, ${\mathcal R}_2({\bf Z}) = 0$, $g = \frac 12\norm{{\bf A} - {\bf U}{\bf Z}}^2_{F}$, $f({\bf U},{\bf Z})= \frac{\lambda_0}{2} \norm{{\bf U}}_F^2$ and $h = h_a$.
	\begin{proposition}\label{prop:closed-form-1c}
	In BPG-MF, with above defined $g,f,h$ the update steps in each iteration are given by ${\bf U}^{{\bf k+1}} = -\frac{{\bf P^k}}{r_1+\lambda_0}$, ${\bf Z}^{{\bf k+1}} = -\frac{{\bf Q^k}}{r_1}$ where $r_1\geq 0$ and satisfies
	\begin{align}
		&c_1\left(\norm{{\bf Q^k}}_F^2(r_1+\lambda_0)^2 +  \norm{{\bf P^k}}_F^2r_1^2  \right) + c_2r_1^2(r_1+\lambda_0)^2-r_1^3(r_1+\lambda_0)^2 = 0\,,
	\end{align}
	with $c_1 = 3$ and $c_2 = \norm{{\bf A}}_F$.
	\end{proposition}
	\begin{proof} The proof is similar to that of Proposition~\ref{prop:closed-form-1}.
	Consider the following subproblem
	\begin{align*}
	  ({\bf U}^{{\bf k+1}}, {\bf Z}^{{\bf k+1}}) \in \underset{({\bf U},{\bf Z}) \in \R^{M \times K}\times \R^{K\times N}}{\argmin} &\left\{\frac{\lambda_0}{2} \norm{{\bf U}}_F^2  + \act{ {\bf P^k}, {\bf U}} + \act{ {\bf Q^k} , {\bf Z}} \right.\\
	& \left. + c_1 \left(\frac{\norm{\bf U}_F^2 +\norm{\bf Z}_F^2 }{2} \right)^2 + c_2\left( \frac{\norm{\bf U}_F^2 +\norm{\bf Z}_F^2 }{2}\right) \right\}\,,
	\end{align*}
	Denote the objective in the above minimization problem as $\mathcal{O({\bf U}^{\bf k}, {\bf Z}^{\bf k})}$. Now, we show that the following holds
	\begin{align*}
		&\min_{({\bf U},{\bf Z}) \in \R^{M \times K}\times \R^{K\times N}} \left(\mathcal{O({\bf U}^{\bf k}, {\bf Z}^{\bf k})}\right) \\
		&\equiv \min_{t_1 \geq 0, t_2 \geq 0} \left\{\min_{({\bf U},{\bf Z}) \in \R^{M \times K}\times \R^{K\times N},\norm{\bf U}_F =t_1,\norm{\bf Z}_F =t_2} \left(\mathcal{O({\bf U}^{\bf k}, {\bf Z}^{\bf k})} \right)\right\}\,,\\
		&\equiv \min_{t_1 \geq 0, t_2 \geq 0} \left\{\min_{({\bf U},{\bf Z}) \in \R^{M \times K}\times \R^{K\times N},\norm{\bf U}_F \leq t_1,\norm{\bf Z}_F \leq t_2} \left(\mathcal{O({\bf U}^{\bf k}, {\bf Z}^{\bf k})} \right)\right\}\,.
	\end{align*}
	where the first step is a simple rewriting of the objective and the second step follows as there is no change in the constraint set and due to Lemma~\ref{lem:helper-1}, which is given precisely in Proposition~\ref{prop:closed-form-1} where the equivalence argument used for \eqref{eq:subprob-equiv-2} and \eqref{eq:subprob-equiv-1} holds here. Note that in the first step, we used $\norm{\bf U}_F =t_1$ this results in deviation of value of $c_2$ to $c_2 +\lambda_0$, corresponding to ${\bf U}$ (see below). We solve for  $({\bf U}^{{\bf k+1}}, {\bf Z}^{{\bf k+1}})$ via the following strategy. Denote
	\[
	{{\bf U}^{{\bf *}}_1}(t_1) \in  \argmin \left\{ \act{ {\bf P^k}, {\bf U_1}  } : {\bf U_1} \in \R^{M \times K}\,, \norm{{\bf U_1}}_F^2 \leq t_1\right\}\,,
	\]
	\[
		{{\bf Z}^{*}_1 }(t_2) \in \argmin \left\{ \act{ {\bf Q^k} , {\bf Z_1}}: {\bf Z_1} \in \R^{K\times N}\,, \norm{{\bf Z_1}}_F^2 \leq t_2 \right\}\,.
	\]
	Then we obtain  $({\bf U}^{{\bf k+1}}, {\bf Z}^{{\bf k+1}}) =  ({{\bf U}^{{\bf *}}_1}(t_1^*),{{\bf Z}^{*}_1 }(t_2^*))$, where $t_1^*$ and $t_2^*$ are obtained by solving the following two dimensional subproblem
	\begin{align*}
	  (t_1^*,t_2^*) \in  \underset{t_1\geq 0, t_2\geq 0}{\mathrm{argmin}}   &\left\{  \min_{{\bf U_1} \in \R^{M \times K}} \left\{  \act{ {\bf P^k}, {\bf U_1}  } : \norm{{\bf U_1}}_F^2 \leq t_1\right\} \right. \\
	& \left. + \min_{{\bf Z_1} \in \R^{K\times N} } \left\{ \act{ {\bf Q^k} , {\bf Z_1}}: \norm{{\bf Z_1}}_F^2 \leq t_2\right\} \right. \\
	&\left. + c_1\left(\frac{t_1^2 + t_2^2}{2}\right)^2 + c_2\frac{t_2^2}{2}  + (c_2+\lambda_0)\frac{ t_1^2}{2}\right\}\,.
	\end{align*}

	Note that inner minimization subproblems can be trivially solved once we obtain ${{\bf U}^{{\bf *}}_1}(t_1)$ and ${{\bf Z}^{*}_1 }(t_2)$ via Lemma~\ref{lem:helper-1}. Then the solution to the subproblem in each iteration as follows:
	\begin{align*}
	{\bf U}^{{\bf k+1}} 
	&= \left.
	\begin{cases}
	    t_1^*\frac{-{\bf P^k}}{\norm{{\bf P^k}}_F}, & \text{for } \norm{{\bf P^k}}_F \neq 0\,, \\
	    {\bf 0} &  otherwise\,.
	\end{cases}\right.\\
	{\bf Z}^{{\bf k+1}} 
	&= \left.
	\begin{cases}
	    t_2^*\frac{-{\bf Q^k}}{\norm{{\bf Q^k}}_F}, & \text{for } \norm{{\bf Q^k}}_F \neq 0\,, \\
	    {\bf 0} &  otherwise\,.
	\end{cases}
	 \right.
	\end{align*}
	We solve for $t_1^*$ and $t_2^*$ with the following two dimensional minimization problem
	\begin{align*}
	\underset{t_1\geq 0, t_2\geq 0}{\mathrm{argmin}} & \left\{  -t_1\norm{{\bf P^k}}_F  - t_2{\norm{{\bf Q^k}}_F}  + c_1\left(\frac{t_1^2 + t_2^2}{2}\right)^2 + c_2\frac{t_2^2}{2}  + (c_2+\lambda_0)\frac{ t_1^2}{2}\right\}\,.
	\end{align*}
	Thus, the solutions $t_1^*$ and $t_2^*$ are the non-negative real roots of the following equations
	\begin{align}
	&-\norm{{\bf P^k}}_F  + c_1(t_1^2 + t_2^2)t_1 + (c_2+\lambda_0)t_1 = 0\,,\\
	&-\norm{{\bf Q^k}}_F + c_1(t_1^2 + t_2^2)t_2 + c_2t_2 = 0\,.
	\end{align}
	Further simplifications with $t_1 = \frac{\norm{{\bf P^k}}_F}{r_1 + \lambda_0}$ and $t_2 = \frac{\norm{{\bf Q^k}}_F}{r_1}$ denoting  $r_1 = c_1(t_1^2 + t_2^2) + c_2$, then we have
	\[
		r_1 = c_1\left( \left(\frac{\norm{{\bf P^k}}_F}{r_1 + \lambda_0}\right)^2 + \left(\frac{\norm{{\bf Q^k}}_F}{r_1}\right)^2\right) +c_2
	\]
	This will result in following $5^{th}$ order equation,
	\begin{align*}
	 &c_1\left( \norm{{\bf P^k}}_F^2 r_1^2 + \norm{{\bf Q^k}}_F^2 (r_1+\lambda_0)^2\right)+ c_2r_1^2(r_1+\lambda_0)^2-r_1^3(r_1+\lambda_0)^2 = 0 \,.
	\end{align*}
	\end{proof}
	\subsection{Conversion to Cubic Equation}\label{ssec:conv-cubic}
	We set ${\mathcal R}_1({\bf U}) = \frac{\lambda_0}{2} \norm{{\bf U}}_F^2$, ${\mathcal R}_2({\bf Z}) = 0$ and $g = \frac 12\norm{{\bf A} - {\bf U}{\bf Z}}^2_{F}$. Denote $f({\bf U},{\bf Z})= \frac{\lambda_0}{2} \norm{{\bf U}}_F^2$, $h({\bf U},{\bf Z}) = h_a({\bf U},{\bf Z}) + \frac{\lambda_0}{2}\norm{{\bf Z}}_F^2$. Note that such a $g$ satisfies $L$-smad property with respect to $h$ satisfies $L$-smad trivially since only a quadratic term is added to $h_a$.
	\begin{proposition}\label{prop:closed-form-5a}
	In BPG-MF, with the above defined $g,f,h$ the update steps in each iteration are given by ${\bf U}^{{\bf k+1}} = -r\,{\bf P^k}$, ${\bf Z}^{{\bf k+1}} = -r\,{\bf Q^k}$ where $r$ is the non-negative real root of
	\begin{align}
		&c_1\left(\norm{{\bf Q^k}}_F^2 +  \norm{{\bf P^k}}_F^2  \right)r^3 + (c_2+\lambda_0)r-1 = 0\label{alg:simple-0}\,,
	\end{align}
	with $c_1 = 3$ and $c_2 = \norm{{\bf A}}_F$.
	\end{proposition}
	\begin{proof}
	The resulting subproblem is 
	\begin{align*}
	  ({\bf U}^{{\bf k+1}}, {\bf Z}^{{\bf k+1}}) \in \underset{({\bf U},{\bf Z}) \in \R^{M \times K}\times \R^{K\times N} }{\argmin} &\left\{ \act{ {\bf P^k}, {\bf U}} + \act{ {\bf Q^k} , {\bf Z}} \right.\\
	& \left. + c_1 \left(\frac{\norm{\bf U}_F^2 +\norm{\bf Z}_F^2 }{2} \right)^2 + \left(c_2 + \lambda_0\right)\left( \frac{\norm{\bf U}_F^2 +\norm{\bf Z}_F^2 }{2}\right) \right\}\,.
	\end{align*}
	The rest of the proof is similar to Proposition~\ref{prop:closed-form-1}.
	\end{proof}
	\subsection{Extensions to Mixed Regularization Terms}\label{ssec:mixed-reg}
	Let $\lambda_0 >0$ and we consider the following problem
	\begin{equation}\label{eq:mat-fac-ex-1cd}
	\min_{{\bf U}\in \R^{M \times K}, {\bf Z}\in \R^{K \times N}}\, \left\{ \Psi({\bf U},{\bf Z}) :=  \frac 12\norm{{\bf A} - {\bf U}{\bf Z}}^2_{F}  + \frac{\lambda_0}{2} \norm{{\bf U}}_F^2  +  \lambda_1  \norm{{\bf Z}}_1  \right\}\,.
	\end{equation}
	Note that the regularizer is a mixture of L1 and L2 regularization. The usual strategy with $h = h_a$ would result in a fifth order polynomial. In order to generate a cubic equation, we use the same strategy as given Section~\ref{ssec:conv-cubic}. We set $h({\bf U},{\bf Z}) = h_a({\bf U},{\bf Z}) + \frac{\lambda_0}{2}\norm{{\bf Z}}_F^2$, $g = \frac 12\norm{{\bf A} - {\bf U}{\bf Z}}^2_{F}$ and  $f({\bf U},{\bf Z})= \frac{\lambda_0}{2} \norm{{\bf U}}_F^2 + \lambda_1  \norm{{\bf Z}}_1$. 
	\begin{proposition}\label{prop:closed-form-5ab}
	In BPG-MF, with the above defined $g,f,h$ the update steps in each iteration are given by ${\bf U}^{{\bf k+1}} = -r\,{\bf P^k}$, ${\bf Z}^{{\bf k+1}} = r \SSS_{\lambda \lambda_1}\left(-\,{\bf Q^k}\right)$ where $r$ is the non-negative real root of
	\begin{align}
		&c_1\left( \norm{{\bf P^k}}_F^2 + \norm{\SSS_{\lambda \lambda_1}\left(-\,{\bf Q^k}\right)}_F^2   \right)r^3 + (c_2+\lambda_0)r-1 = 0\label{alg:simple-0}\,,
	\end{align}
	with $c_1 = 3$ and $c_2 = \norm{{\bf A}}_F$.
	\end{proposition}
	The proof is similar to Proposition~\ref{prop:closed-form-1} and Proposition~\ref{prop:closed-form-3}.

\section{Technical Lemmas and Proofs}\label{sec:technical-proofs}
	Before we proceed to the proof of Proposition~\ref{prop:l-smad-1} we require the following technical lemma.
	\begin{lemma}\label{lem:technical-1}
	Let $g := \frac 12\norm{{\bf A} - {\bf U}{\bf Z}}^2_{F}$, then we have the following
	\[
	\nabla g({\bf A},{\bf U}{\bf Z}) = \left(-({\bf A} - {\bf U}{\bf Z}){\bf Z}^T, -{\bf U}^T({\bf A} - {\bf U}{\bf Z})\right)
	\]
	\[
	\act{({\bf H}_{{\bf 1}},{\bf H}_{{\bf 2}}), \nabla^2 g({\bf A},{\bf U}{\bf Z}) ({\bf H}_{{\bf 1}}, {\bf H}_{{\bf 2}})} = -2\act{{\bf A} - {\bf U}{\bf Z}, {\bf H}_{{\bf 1}}{\bf H}_{{\bf 2}}} + \act{{\bf U}{\bf H}_{{\bf 2}} +  {\bf H}_{{\bf 1}}{\bf Z}, {\bf U}{\bf H}_{{\bf 2}} +  {\bf H}_{{\bf 1}}{\bf Z}}\,.
	\]
	\end{lemma}
	\begin{proof}
	With the Forbenius dot product, we have
	\begin{align*}
	\norm{{\bf A} - {\bf U}{\bf Z}}^2_{F} &= \act{{\bf A} - {\bf U}{\bf Z}, {\bf A} - {\bf U}{\bf Z}}\,.
	\end{align*}
	In the above expression by substituting ${\bf U}$ with ${\bf U}+{\bf H}_{{\bf 1}}$ and ${\bf Z}$ with ${\bf Z}+{\bf H}_{{\bf 2}}$, we obtain 
	\begin{align*}
	 &\act{{\bf A} - ({\bf U}+{\bf H}_{{\bf 1}})({\bf Z} + {\bf H}_{{\bf 2}}), {\bf A} -  ({\bf U}+{\bf H}_{{\bf 1}})({\bf Z} + {\bf H}_{{\bf 2}})}\,, \\
	 &= \act{{\bf A} - {\bf U}{\bf Z} - {\bf U}{\bf H}_{{\bf 2}} -{\bf H}_{{\bf 1}}{\bf Z} - {\bf H}_{{\bf 1}}{\bf H}_{{\bf 2}},{\bf A} - {\bf U}{\bf Z} - {\bf U}{\bf H}_{{\bf 2}} -{\bf H}_{{\bf 1}}{\bf Z} - {\bf H}_{{\bf 1}}{\bf H}_{{\bf 2}} }\,,\\
	 &= \act{{\bf A},{\bf A}} - \act{{\bf A}, {\bf U}{\bf Z}} - \act{{\bf A}, {\bf U}{\bf H}_{{\bf 2}}} - \act{{\bf A}, {\bf H}_{{\bf 1}}{\bf Z}} - \act{{\bf A}, {\bf H}_{{\bf 1}}{\bf H}_{{\bf 2}}} \,,\\
	& -\act{{\bf U}{\bf Z},{\bf A}} + \act{{\bf U}{\bf Z}, {\bf U}{\bf Z}} + \act{{\bf U}{\bf Z}, {\bf U}{\bf H}_{{\bf 2}}} + \act{{\bf U}{\bf Z}, {\bf H}_{{\bf 1}}{\bf Z}} + \act{{\bf U}{\bf Z}, {\bf H}_{{\bf 1}}{\bf H}_{{\bf 2}}} \\
	& -\act{{\bf U}{\bf H}_{{\bf 2}},{\bf A}} + \act{{\bf U}{\bf H}_{{\bf 2}}, {\bf U}{\bf Z}} + \act{{\bf U}{\bf H}_{{\bf 2}}, {\bf U}{\bf H}_{{\bf 2}}} + \act{{\bf U}{\bf H}_{{\bf 2}}, {\bf H}_{{\bf 1}}{\bf Z}} + \act{{\bf U}{\bf H}_{{\bf 2}}, {\bf H}_{{\bf 1}}{\bf H}_{{\bf 2}}} \\
	& -\act{{\bf H}_{{\bf 1}}{\bf Z},{\bf A}} + \act{{\bf H}_{{\bf 1}}{\bf Z}, {\bf U}{\bf Z}} + \act{{\bf H}_{{\bf 1}}{\bf Z}, {\bf U}{\bf H}_{{\bf 2}}} + \act{{\bf H}_{{\bf 1}}{\bf Z}, {\bf H}_{{\bf 1}}{\bf Z}} + \act{{\bf H}_{{\bf 1}}{\bf Z}, {\bf H}_{{\bf 1}}{\bf H}_{{\bf 2}}} \\
	& -\act{{\bf H}_{{\bf 1}}{\bf H}_{{\bf 2}},{\bf A}} + \act{{\bf H}_{{\bf 1}}{\bf H}_{{\bf 2}}, {\bf U}{\bf Z}} + \act{{\bf H}_{{\bf 1}}{\bf H}_{{\bf 2}}, {\bf U}{\bf H}_{{\bf 2}}} + \act{{\bf H}_{{\bf 1}}{\bf H}_{{\bf 2}}, {\bf H}_{{\bf 1}}{\bf Z}} + \act{{\bf H}_{{\bf 1}}{\bf H}_{{\bf 2}}, {\bf H}_{{\bf 1}}{\bf H}_{{\bf 2}}} \,.
	\end{align*}
	Collecting all the first order terms we have
	\begin{align*}
	&- \act{{\bf A}, {\bf U}{\bf H}_{{\bf 2}}}- \act{{\bf A}, {\bf H}_{{\bf 1}}{\bf Z}} + \act{{\bf U}{\bf Z}, {\bf U}{\bf H}_{{\bf 2}}} + \act{{\bf U}{\bf Z}, {\bf H}_{{\bf 1}}{\bf Z}} \\
	&-\act{{\bf U}{\bf H}_{{\bf 2}},{\bf A}} + \act{{\bf U}{\bf H}_{{\bf 2}}, {\bf U}{\bf Z}} -\act{{\bf H}_{{\bf 1}}{\bf Z},{\bf A}} + \act{{\bf H}_{{\bf 1}}{\bf Z}, {\bf U}{\bf Z}} \\
	& =- \act{{\bf A}, {\bf H}_{{\bf 1}}{\bf Z}} + \act{{\bf U}{\bf Z}, {\bf H}_{{\bf 1}}{\bf Z}} -\act{{\bf H}_{{\bf 1}}{\bf Z},{\bf A}} + \act{{\bf H}_{{\bf 1}}{\bf Z}, {\bf U}{\bf Z}} \\
	& - \act{{\bf A}, {\bf U}{\bf H}_{{\bf 2}}} +  \act{{\bf U}{\bf Z}, {\bf U}{\bf H}_{{\bf 2}}} -\act{{\bf U}{\bf H}_{{\bf 2}},{\bf A}} + \act{{\bf U}{\bf H}_{{\bf 2}}, {\bf U}{\bf Z}}\,,\\
	& = -2\act{{\bf A}, {\bf H}_{{\bf 1}}{\bf Z}} -2\act{{\bf A}, {\bf U}{\bf H}_{{\bf 2}}} + 2\act{{\bf U}{\bf Z}, {\bf H}_{{\bf 1}}{\bf Z}} +  2\act{{\bf U}{\bf Z}, {\bf U}{\bf H}_{{\bf 2}}}\,,\\
	&= -2{\bf tr}(({\bf A} - {\bf U}{\bf Z}){\bf Z}^T{\bf H}_{{\bf 1}}^T) - 2{\bf tr}(({\bf A} - {\bf U}{\bf Z}){\bf H}_{{\bf 2}}^T{\bf U}^T)\,,\\
	&= -2{\bf tr}(({\bf A} - {\bf U}{\bf Z}){\bf Z}^T{\bf H}_{{\bf 1}}^T) - 2{\bf tr}({\bf U}^T({\bf A} - {\bf U}{\bf Z}){\bf H}_{{\bf 2}}^T)\,,
	\end{align*}
	and similarly collecting all the second order terms we have
	\begin{align*}
	&- \act{{\bf A}, {\bf H}_{{\bf 1}}{\bf H}_{{\bf 2}}}+ \act{{\bf U}{\bf Z}, {\bf H}_{{\bf 1}}{\bf H}_{{\bf 2}}}+ \act{{\bf U}{\bf H}_{{\bf 2}}, {\bf U}{\bf H}_{{\bf 2}}} + \act{{\bf U}{\bf H}_{{\bf 2}}, {\bf H}_{{\bf 1}}{\bf Z}}\\
	& + \act{{\bf H}_{{\bf 1}}{\bf Z}, {\bf U}{\bf H}_{{\bf 2}}} + \act{{\bf H}_{{\bf 1}}{\bf Z}, {\bf H}_{{\bf 1}}{\bf Z}}  -\act{{\bf H}_{{\bf 1}}{\bf H}_{{\bf 2}},{\bf A}}+ \act{{\bf H}_{{\bf 1}}{\bf H}_{{\bf 2}}, {\bf U}{\bf Z}} \\
	& = -2\act{{\bf A} - {\bf U}{\bf Z}, {\bf H}_{{\bf 1}}{\bf H}_{{\bf 2}}} + \act{{\bf U}{\bf H}_{{\bf 2}} +  {\bf H}_{{\bf 1}}{\bf Z}, {\bf U}{\bf H}_{{\bf 2}} +  {\bf H}_{{\bf 1}}{\bf Z}}\,.
	\end{align*}
	Thus the statement follows using the second order Taylor expansion.
	\end{proof}
	\begin{lemma}\label{lem:l-smad-helper-1}
	Given $h_1 := \left(\frac{\norm{\bf U}_F^2 +\norm{\bf Z}_F^2 }{2} \right)^2$, then we have the following
	\[
		\nabla h_1({\bf U},{\bf Z}) = \left(\left(\norm{\bf U}_F^2+\norm{\bf Z}_F^2\right){\bf U}, \left(\norm{\bf U}_F^2+\norm{\bf Z}_F^2\right){\bf Z}\right)\,,
	\]
	\[
	\act{({\bf H}_{{\bf 1}},{\bf H}_{{\bf 2}}), \nabla^2 h_1({\bf U},{\bf Z}) ({\bf H}_{{\bf 1}}, {\bf H}_{{\bf 2}})} = (\norm{{\bf H}_{{\bf 1}}}_F^2+\norm{{\bf H}_{{\bf 2}}}_F^2)(\norm{\bf U}_F^2+\norm{\bf Z}_F^2)  + 2\norm{{\bf H}_{{\bf 1}}{\bf U}^T + {\bf Z}{\bf H}_{{\bf 2}}^T }_F^2
	\]
	\end{lemma}
	\begin{proof}
	By the definition of Forbenius dot product, we have
	\begin{align*}
	&\frac{1}{4}\norm{\bf U}_F^4 + \frac{1}{4}\norm{\bf Z}_F^4 + \frac{1}{2}\norm{\bf U}_F^2\norm{\bf Z}_F^2 =\frac{1}{4} \act{{\bf U},{\bf U}}^2 + \frac{1}{4}\act{{\bf Z},{\bf Z}}^2 + \frac{1}{2}\act{{\bf U},{\bf U}}\act{{\bf Z},{\bf Z}}
	\end{align*}
	Now, considering $h_1({\bf U}+{\bf H}_{{\bf 1}}, {\bf Z}+{\bf H}_{{\bf 2}})$ we have
	\begin{align*}
	&\frac{1}{4} \act{{\bf U}+{\bf H}_{{\bf 1}},{\bf U}+{\bf H}_{{\bf 1}}}^2 + \frac{1}{4}\act{{\bf Z} + {\bf H}_{{\bf 2}},{\bf Z} + {\bf H}_{{\bf 2}}}^2 + \frac{1}{2}\act{{\bf U}+{\bf H}_{{\bf 1}},{\bf U}+{\bf H}_{{\bf 1}}}\act{{\bf Z} + {\bf H}_{{\bf 2}},{\bf Z} + {\bf H}_{{\bf 2}}}\\
	& =\frac{1}{4} \left( \act{{\bf U},{\bf U}}+ 2\act{{\bf H}_{{\bf 1}}, {\bf U}} + \act{{\bf H}_{{\bf 1}}, {\bf H}_{{\bf 1}}} \right)^2 + \frac{1}{4} \left( \act{{\bf Z},{\bf Z}} + 2\act{{\bf Z}, {\bf H}_{{\bf 2}}} + \act{{\bf H}_{{\bf 2}}, {\bf H}_{{\bf 2}}} \right)^2 \\
	&+ \frac{1}{2}\left(\act{{\bf U},{\bf U}}+ 2\act{{\bf H}_{{\bf 1}}, {\bf U}} + \act{{\bf H}_{{\bf 1}}, {\bf H}_{{\bf 1}}}\right)\left(\act{{\bf Z},{\bf Z}} + 2\act{{\bf Z}, {\bf H}_{{\bf 2}}} + \act{{\bf H}_{{\bf 2}}, {\bf H}_{{\bf 2}}} \right)\\
	& =\frac{1}{4}\left(\act{{\bf U},{\bf U}}^2 + 4\act{{\bf H}_{{\bf 1}}, {\bf U}}^2 + \act{{\bf H}_{{\bf 1}}, {\bf H}_{{\bf 1}}}^2 + 2\act{{\bf H}_{{\bf 1}}, {\bf H}_{{\bf 1}}}\act{{\bf U},{\bf U}} \right.\\
	&\left. + 4\act{{\bf U},{\bf U}}\act{{\bf H}_{{\bf 1}}, {\bf U}} + 4\act{{\bf H}_{{\bf 1}}, {\bf U}} \act{{\bf H}_{{\bf 1}}, {\bf H}_{{\bf 1}}}   \right)\\
	& + \frac{1}{4} \left( \act{{\bf Z},{\bf Z}}^2 + 4\act{{\bf Z}, {\bf H}_{{\bf 2}}}^2 + \act{{\bf H}_{{\bf 2}}, {\bf H}_{{\bf 2}}}^2 + 2\act{{\bf H}_{{\bf 2}}, {\bf H}_{{\bf 2}}}\act{{\bf Z},{\bf Z}} \right.\\
	&\left. + 4\act{{\bf Z}, {\bf H}_{{\bf 2}}}\act{{\bf Z},{\bf Z}} + 4\act{{\bf Z}, {\bf H}_{{\bf 2}}}\act{{\bf H}_{{\bf 2}}, {\bf H}_{{\bf 2}}} \right) \\
	&+ \frac{1}{2}\left(\act{{\bf U},{\bf U}}\act{{\bf Z},{\bf Z}} + 2\act{{\bf U},{\bf U}}\act{{\bf Z}, {\bf H}_{{\bf 2}}} + \act{{\bf U},{\bf U}}\act{{\bf H}_{{\bf 2}}, {\bf H}_{{\bf 2}}} \right)\\
	&+ \frac{1}{2}\left( 2\act{{\bf H}_{{\bf 1}}, {\bf U}}\act{{\bf Z},{\bf Z}} + 4\act{{\bf H}_{{\bf 1}}, {\bf U}}\act{{\bf Z}, {\bf H}_{{\bf 2}}} + 2\act{{\bf H}_{{\bf 1}}, {\bf U}}\act{{\bf H}_{{\bf 2}}, {\bf H}_{{\bf 2}}} \right)\\
	& + \frac{1}{2}\left(\act{{\bf H}_{{\bf 1}}, {\bf H}_{{\bf 1}}}\act{{\bf Z},{\bf Z}} + 2\act{{\bf H}_{{\bf 1}}, {\bf H}_{{\bf 1}}}\act{{\bf Z}, {\bf H}_{{\bf 2}}} + \act{{\bf H}_{{\bf 1}}, {\bf H}_{{\bf 1}}}\act{{\bf H}_{{\bf 2}}, {\bf H}_{{\bf 2}}} \right)
	\end{align*}
	Collecting all the first order terms, we have
	\begin{align*}
	\act{{\bf U},{\bf U}}\act{{\bf H}_{{\bf 1}}, {\bf U}} + \act{{\bf Z}, {\bf H}_{{\bf 2}}}\act{{\bf Z},{\bf Z}} + \act{{\bf U},{\bf U}}\act{{\bf Z}, {\bf H}_{{\bf 2}}}+ \act{{\bf H}_{{\bf 1}}, {\bf U}}\act{{\bf Z},{\bf Z}} \,,
	\end{align*}
	and similarly collecting all the second order terms we have
	\begin{align*}
	& \frac{1}{4}\left( 4\act{{\bf H}_{{\bf 1}}, {\bf U}}^2 + 2\act{{\bf H}_{{\bf 1}}, {\bf H}_{{\bf 1}}}\act{{\bf U},{\bf U}}  +  4\act{{\bf Z}, {\bf H}_{{\bf 2}}}^2  + 2\act{{\bf H}_{{\bf 2}}, {\bf H}_{{\bf 2}}}\act{{\bf Z},{\bf Z}}  \right)\\
	&+ \frac{1}{2}\left( \act{{\bf U},{\bf U}}\act{{\bf H}_{{\bf 2}}, {\bf H}_{{\bf 2}}}  +  4\act{{\bf H}_{{\bf 1}}, {\bf U}}\act{{\bf Z}, {\bf H}_{{\bf 2}}}  +\act{{\bf H}_{{\bf 1}}, {\bf H}_{{\bf 1}}}\act{{\bf Z},{\bf Z}}  \right)\,, \\
	& = \frac{1}{2}\left(2\act{{\bf H}_{{\bf 1}}, {\bf U}}^2 + (\act{{\bf H}_{{\bf 1}}, {\bf H}_{{\bf 1}}}+\act{{\bf H}_{{\bf 2}}, {\bf H}_{{\bf 2}}})(\act{{\bf U},{\bf U}} +\act{{\bf Z}, {\bf Z}}) \right.\\ 
	&\left. +  2\act{{\bf Z}, {\bf H}_{{\bf 2}}}^2 + 4\act{{\bf H}_{{\bf 1}}, {\bf U}}\act{{\bf Z}, {\bf H}_{{\bf 2}}} \right)\,, \\
	& = \frac{1}{2} \left( (\act{{\bf H}_{{\bf 1}}, {\bf H}_{{\bf 1}}}+\act{{\bf H}_{{\bf 2}}, {\bf H}_{{\bf 2}}})(\act{{\bf U},{\bf U}} +\act{{\bf Z}, {\bf Z}})  + 2 (\act{{\bf H}_{{\bf 1}}, {\bf U}} + \act{{\bf Z}, {\bf H}_{{\bf 2}}})^2 \right)\,.
	\end{align*}
	Thus the statement follows.
	\end{proof}
	\begin{lemma}\label{lem:l-smad-helper-2}
	Given $h_2({\bf U},{\bf Z}) :=  \frac{\norm{\bf U}_F^2 +\norm{\bf Z}_F^2 }{2}$,  then we have the following
	\[
	\nabla h_2({\bf U},{\bf Z}) = \left({\bf U},{\bf Z}\right),
	\]
	\[
	\act{({\bf H}_{{\bf 1}},{\bf H}_{{\bf 2}}), \nabla^2 h_2({\bf U},{\bf Z}) ({\bf H}_{{\bf 1}}, {\bf H}_{{\bf 2}})} = \norm{{\bf H}_{{\bf 1}}}_F^2 + \norm{{\bf H}_{{\bf 2}}}_F^2 \,.
	\]
	\end{lemma}
	\begin{proof}
	Considering $h_2({\bf U}+{\bf H}_{\bf 1},{\bf Z}+{\bf H}_{\bf 2})$, we have
	\begin{align*}
	& \frac{1}{2} \act{{\bf U}+{\bf H}_{{\bf 1}},{\bf U}+{\bf H}_{{\bf 1}}} + \frac{1}{2}\act{{\bf Z} + {\bf H}_{{\bf 2}},{\bf Z} + {\bf H}_{{\bf 2}}}\\
	&= \frac{1}{2} \left( \act{{\bf U},{\bf U}} + 2\act{{\bf U}, {\bf H}_{{\bf 1}}} + \act{{\bf H}_{{\bf 1}},{\bf H}_{{\bf 1}}} \right) + \frac{1}{2}\left(\act{{\bf Z},{\bf Z} } + 2\act{{\bf Z}, {\bf H}_{{\bf 2}}} + \act{{\bf H}_{{\bf 2}},{\bf H}_{{\bf 2}}}\right)\,.
	\end{align*}
	Collecting all the first order terms we have
	\begin{align*}
	\act{{\bf U}, {\bf H}_{{\bf 1}}} + \act{{\bf Z}, {\bf H}_{{\bf 2}}} \,,
	\end{align*}
	and similarly collecting all the second order terms we have
	\begin{align*}
	\frac{1}{2} \left( \act{{\bf H}_{{\bf 1}},{\bf H}_{{\bf 1}}}+ \act{{\bf H}_{{\bf 2}},{\bf H}_{{\bf 2}}}\right)\,.
	\end{align*}
	Thus the statement holds.
	\end{proof}

\subsection{Proof of Proposition~\ref{prop:l-smad-1}}\label{ssec:main-lsmad-proof-sec}
\begin{proof} We prove here the convexity of $Lh_a-g$ for a certain constant $L\geq 1$. 	With  Lemma~\ref{lem:technical-1} we obtain  
	 \begin{align*}
	& \act{({\bf H}_1,{\bf H}_2), \nabla^2 g({\bf A},{\bf U}{\bf Z}) ({\bf H}_1,{\bf H}_2)} \\
	 &=  \norm{{\bf H}_1{\bf Z} + {\bf U}{\bf H}_2}_F^2 - 2 \act{{\bf A} - {\bf U}{\bf Z}, {\bf H}_1{\bf H}_2}\,,\\
	& \leq 2 \norm{{\bf H}_1{\bf Z}}_F^2 + 2\norm{{\bf U}{\bf H}_2}_F^2  + 2 \norm{{\bf A}}_F\norm{{\bf H}_1{\bf H}_2}_F + 2\norm{{\bf U}{\bf Z}}_F\norm{{\bf H}_1{\bf H}_2}_F\,,\\
	& \leq 2 \norm{{\bf H}_1}_F^2\norm{\bf Z}_F^2 + 2\norm{\bf U}_F^2 \norm{{\bf H}_2}_F^2  + 2 \norm{{\bf A}}_F\norm{{\bf H}_1}_F\norm{{\bf H}_2}_F + 2\norm{\bf U}_F\norm{\bf Z}_F\norm{{\bf H}_1}_F\norm{{\bf H}_2}_F\,.
	 \end{align*}
	 With AM-GM inequality, for non-negative real numbers $a,b$ we have $2\sqrt{ab} \leq a + b $, we have
	 \begin{align*}
	  2\norm{\bf U}_F\norm{\bf Z}_F\norm{{\bf H}_1}_F\norm{{\bf H}_2}_F \leq \norm{{\bf H}_1}_F^2\norm{\bf Z}_F^2 + \norm{\bf U}_F^2 \norm{{\bf H}_2}_F^2\,,
	 \end{align*}
	 and similarly we have
	 \begin{align*}
	 2 \norm{{\bf A}}_F\norm{{\bf H}_1}_F\norm{{\bf H}_2}_F \leq \norm{{\bf A}}_F \norm{{\bf H}_1}_F^2 + \norm{{\bf A}}_F\norm{{\bf H}_2}_F^2\,.
	 \end{align*}
	Using the above two inequalities, we obtain
	 \begin{align}
	 \act{({\bf H}_1,{\bf H}_2), \nabla^2 g(A,{\bf U}{\bf Z}) ({\bf H}_1,{\bf H}_2)} \leq (3\norm{\bf Z}_F^2 + \norm{{\bf A}}_F)\norm{{\bf H}_1}_F^2 + (3\norm{\bf U}_F^2 + \norm{{\bf A}}_F)\norm{{\bf H}_2}_F^2\label{eq:proof-l-smad}\,.
	 \end{align}
	 Now, considering the kernel generating distances, via Lemma~\ref{lem:l-smad-helper-1} and \ref{lem:l-smad-helper-2} we obtain
	 \begin{align*}
	 & \act{({\bf H}_1,{\bf H}_2), \nabla^2 h_1({\bf U},{\bf Z}) ({\bf H}_1,{\bf H}_2)} \\
	 &=  2\norm{{\bf H}_1{\bf U} + {\bf H}_2{\bf Z}}_F^2  + (\norm{\bf U}_F^2 +\norm{\bf Z}_F^2)\norm{{\bf H}_1}_F^2 +  (\norm{\bf U}_F^2 +\norm{\bf Z}_F^2)\norm{{\bf H}_2}_F^2\\
	 &\geq \norm{\bf Z}_F^2\norm{{\bf H}_1}_F^2 + \norm{\bf U}_F^2\norm{{\bf H}_2}_F^2\,,
	 \end{align*}
	 and
	  \begin{align*}
	  &\act{({\bf H}_1,{\bf H}_2), \nabla^2 h_2({\bf U},{\bf Z}) ({\bf H}_1,{\bf H}_2)} = \norm{{\bf H}_1}_F^2 +\norm{{\bf H}_2}_F^2\,.
	 \end{align*}
	Now, it is easy to see that
	  \[
	 \act{({\bf H}_1,{\bf H}_2), \nabla^2 h_a({\bf U},{\bf Z}) ({\bf H}_1,{\bf H}_2)} \geq \act{({\bf H}_1,{\bf H}_2), \nabla^2 g({\bf A},{\bf U}{\bf Z}) ({\bf H}_1,{\bf H}_2)}\,.
	 \]
	 A similar proof holds for the convexity of $Lh_a +g$, however the choice of $L$ here need not be the same as it is for $Lh_a - g$ (see \cite[Remark 2.1]{BSTV2018}).
	\end{proof}
	
\section{Additional Experiments and Implementation Details}\label{sec:additional-experiments}
\subsection{Double Backtracking Implementation}
This subsection where we provide certain crucial implementation details of CoCaIn BPG-MF algorithm, is largely based on \cite[Section 5.4]{MOPS2019}. Note that CoCaIn BPG-MF is a sequential algorithm in the sense one can compute $Y_{\bf U}^{{\bf k}},Y_{\bf Z}^{{\bf k}}$ first via the steps \eqref{eq:cocain-0}, \eqref{eq:cocain-1} and \eqref{eq:cocain-2}. Then, the updates can be done exactly like BPG-MF, where step-size depends on the parameter  ${\bar L}_k$ obtained via \eqref{eq:cocain-4}. In  \eqref{eq:cocain-2} it is required to find ${\underline L}_k$ such that the following holds 
	\begin{equation}\label{eq:cocain-temp-2}
		D_{g}\left({\bf U}^{{\bf k+1}},{\bf Z}^{{\bf k+1}},Y_{\bf U}^{{\bf k}},Y_{\bf Z}^{{\bf k}}\right) \geq -{\underline L}_k D_h\left({\bf U}^{{\bf k+1}},{\bf Z}^{{\bf k+1}},Y_{\bf U}^{{\bf k}},Y_{\bf Z}^{{\bf k}}\right)\,,
	\end{equation}
similarly in  \eqref{eq:cocain-4}  it is required to find ${\bar L}_k$ such that
	\begin{equation}\label{eq:cocain-temp-4}
		D_{g}\left({\bf U}^{{\bf k+1}},{\bf Z}^{{\bf k+1}},Y_{\bf U}^{{\bf k}},Y_{\bf Z}^{{\bf k}}\right) \leq {\bar L}_k D_{h}\left({\bf U}^{{\bf k+1}},{\bf Z}^{{\bf k+1}},Y_{\bf U}^{{\bf k}},Y_{\bf Z}^{{\bf k}}\right)\,.
	\end{equation}
The above mentioned steps can be solved via the classical backtracking strategy for ${\underline L}_k$  and ${\bar L}_k$ individually, hence the name "double backtracking". We describe the backtracking procedure for ${\underline L}_k$ and it is easy to extend to ${\bar L}_k$. The backtracking strategy involves a scaling parameter $\nu \geq 1$ and an initialization point ${\underline L}_{k,0}>0$ (preferably small) both chosen by the user and the parameter ${\underline L}_k$ is set to the smallest element from the set $\left\{ {\underline L}_{k,0},\nu {\underline L}_{k,0}, \nu^2 {\underline L}_{k,0},\ldots \right\}$ such that \eqref{eq:cocain-2} holds. For ${\bar L}_k$ one requires to use \eqref{eq:cocain-4} and also due to the additional restriction that ${\bar L}_k \geq {\bar L}_{k-1}$ in CoCaIn BPG-MF it is required to start the initialization  ${\bar L}_{k,0} = {\bar L}_{k-1}$.

\subsection{Non-negative Matrix Factorization}
We consider the same setting as the simple matrix factorization problem considered in \ref{sec:exps}, however we set $\mathcal U = \R^{M \times K}_{+} \text{ and } \mathcal Z = \R^{K \times N}_{+}$. We consider Medulloblastoma dataset \cite{BTGM2004} dataset with matrix $A \in \R^{5893 \times 34}$. As evident from Figure~\ref{fig:nmf-data} PALM based methods outpeform BPG methods here. This raises new open questions and  hints at potential variants of BPG which are better suited for constrained problems.  \ifpaper\else\\\fi
\begin{figure}[hbt!]
	\begin{subfigure}{0.32\textwidth}
		\centering
		\includegraphics[width=1\textwidth]{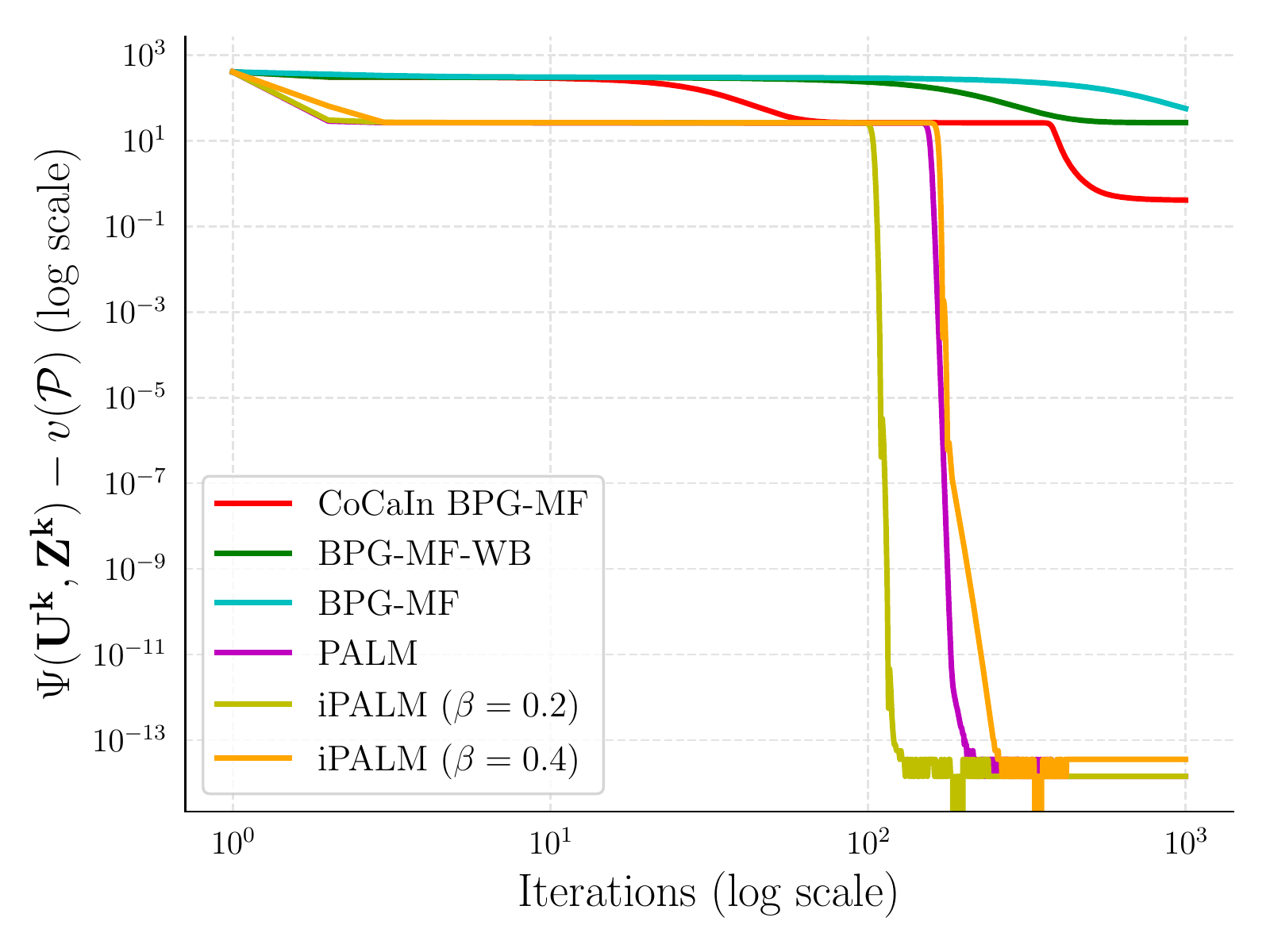}
		\caption{No-Regularization}
	\end{subfigure}
	\begin{subfigure}{0.32\textwidth}
		\centering
		\includegraphics[width=1\textwidth]{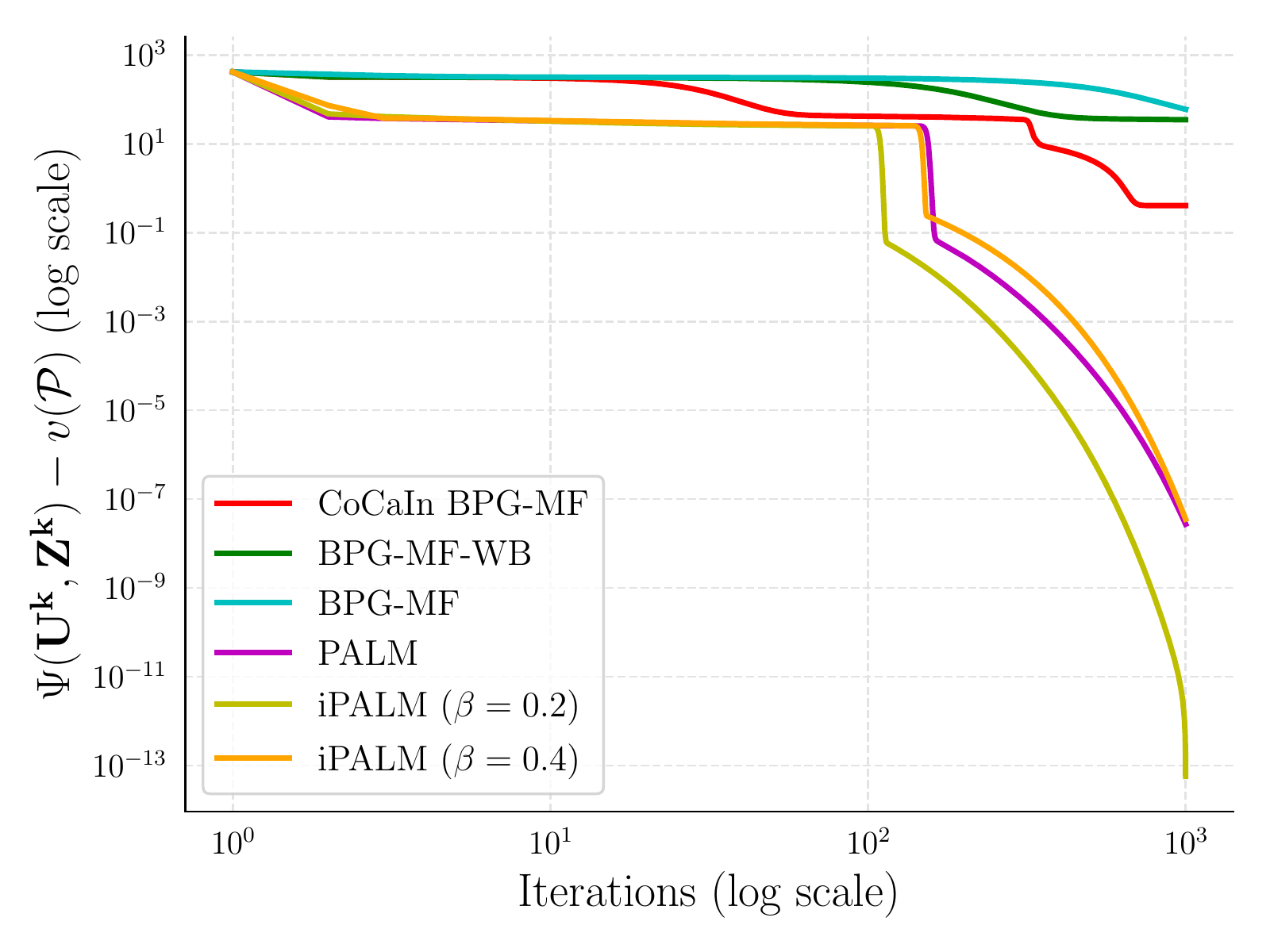}
		\caption{L2-Regularization}
	\end{subfigure}
	\begin{subfigure}{0.32\textwidth}
		\centering
		\includegraphics[width=1\textwidth]{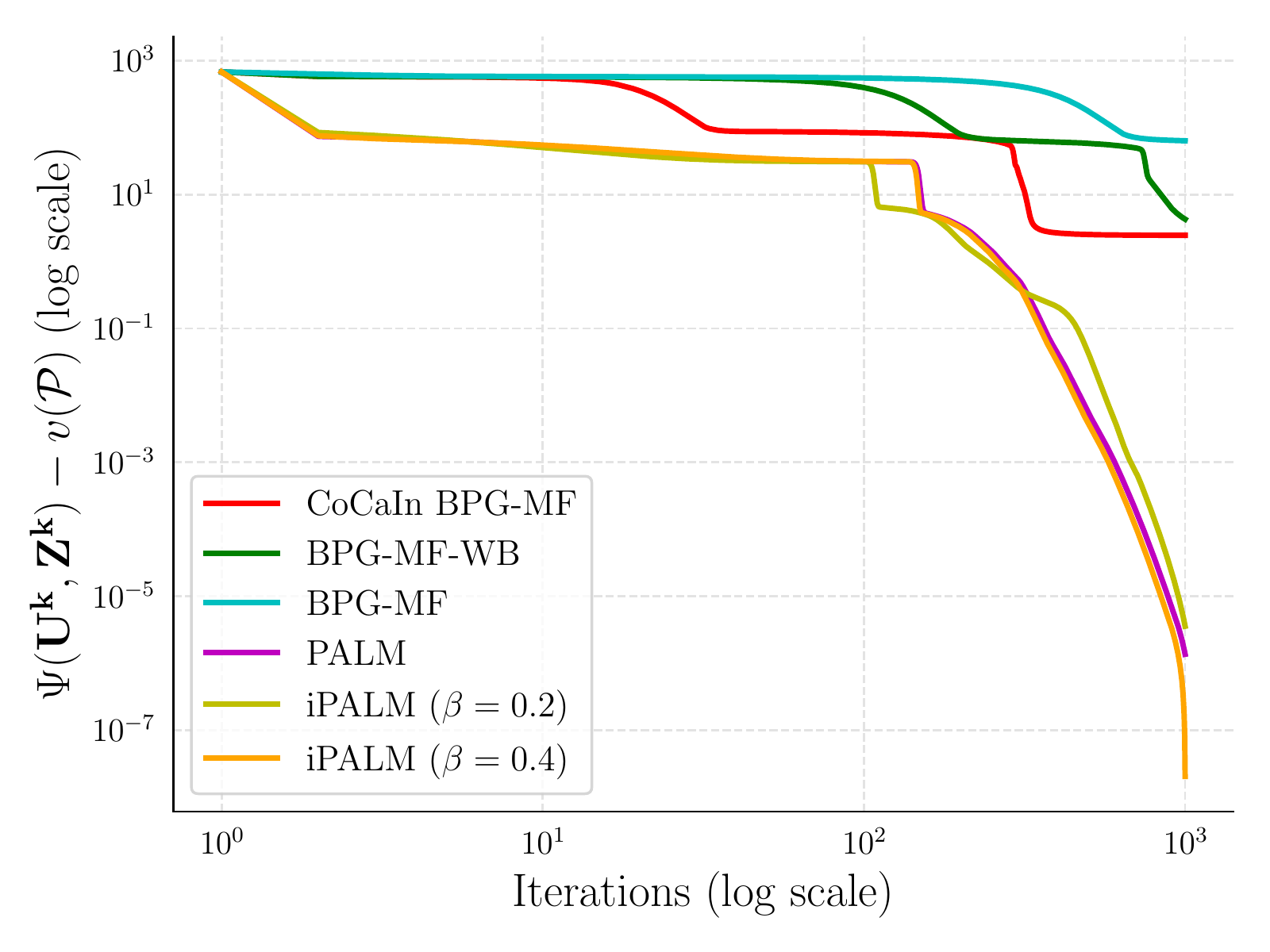}
		\caption{L1-Regularization}
	\end{subfigure}
	\caption{\textbf{Non-negative Matrix Factorization on  Medulloblastoma Dataset \cite{BTGM2004}.} }
	\label{fig:nmf-data}
\end{figure}

\subsection{Matrix Completion}

The MovieLens datasets are essentially a matrix $A \in \R^{M \times N}$, where $M$ denotes the number of users  and $N$ denotes the number of movies. Only a few non-zero entries are given and the entries denote the ratings which the user has provided for a particular movie. The ratings can take the value between 1 and 5, which we refer to as scale. The exact statistics of all the MovieLens datasets are given below.
\begin{center}
 \begin{tabular}{|c c c c c|}
 \hline
 \textbf{Dataset} & \textbf{Users} & \textbf{Movies} & \textbf{Non-zero entries} & \textbf{Scale} \\ [0.5ex] 
 \hline\hline
 MovieLens100K & 943 & 1682 & 100000 & 1-5 \\ 
 \hline
 MovieLens1M & 6040 & 3952 & 1000209 & 1-5\\
 \hline
 MovieLens10M & 71567 & 10681 & 10000054 & 1-5\\
 \hline
\end{tabular}
\end{center}
The plots provided for the matrix completion problem in Section~\ref{sec:exps} uses only 80\% of the data and we use the remaining 20\% as test data in order to obtain the generalization performance to unseen matrix entries with the resulting factors ${\bf U} \in \R^{M \times K}$ and ${\bf Z} \in \R^{K \times N}$ where we use $K=5$. The predicted rating to a particular $i \in \{1,2,\ldots, M\}$  and  $j \in \{1,2,\ldots,N\}$ is given by $({\bf UZ})_{ij}$. The test data is comprised of matrix indices with unseen entries and we denote this set of indices as $\Omega_T$. A popular measure for the test data is the Test RMSE, which is given by the following entity
\begin{align*}
\text{Test RMSE} = \sqrt{\frac{1}{|\Omega_T|} \sum_{i=1}^{M}\sum_{j=1}^{N} {\bf I}_{(i,j) \in \Omega_T}  \left({\bf A}_{ij} - ({\bf UZ})_{ij}\right)^2}
\end{align*}
where $|\Omega_T|$ denotes the cardinality of the set $\Omega_T$ and  ${\bf I}_{(i,j) \in \Omega_T} = 1$  if the index pair $(i,j)$ lies in the set $\Omega_T$ else it is zero. The Test RMSE comparisons for the MovieLens Dataset are given below in Figure~\ref{fig:test-rmse-plot}.
\begin{figure}[htb]
\begin{subfigure}{0.325\textwidth}
			\centering
			\includegraphics[width=0.9\textwidth]{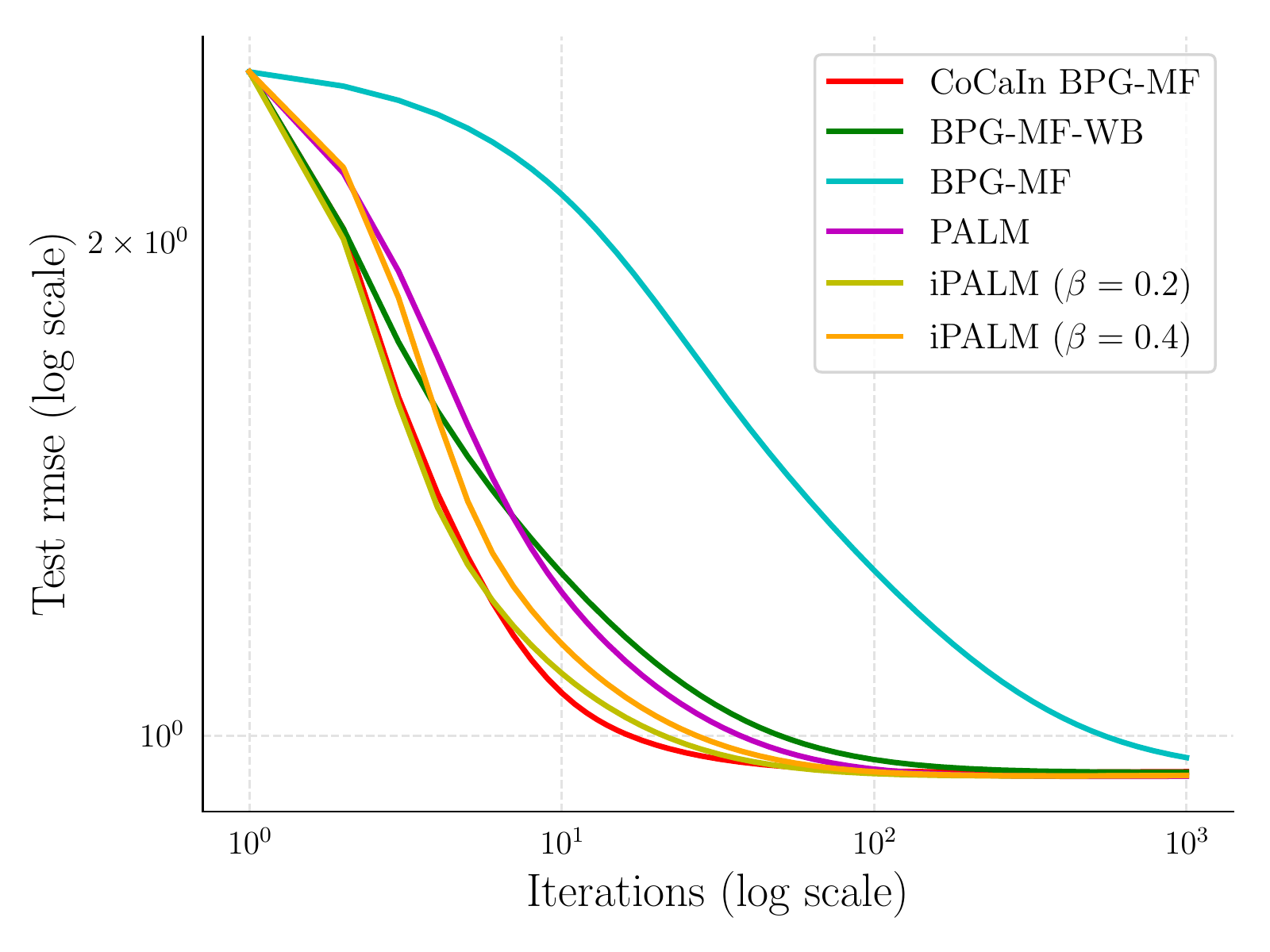}
			\caption{MovieLens-100K}
		\end{subfigure}
		\begin{subfigure}{0.325\textwidth}
			\centering
			\includegraphics[width=0.9\textwidth]{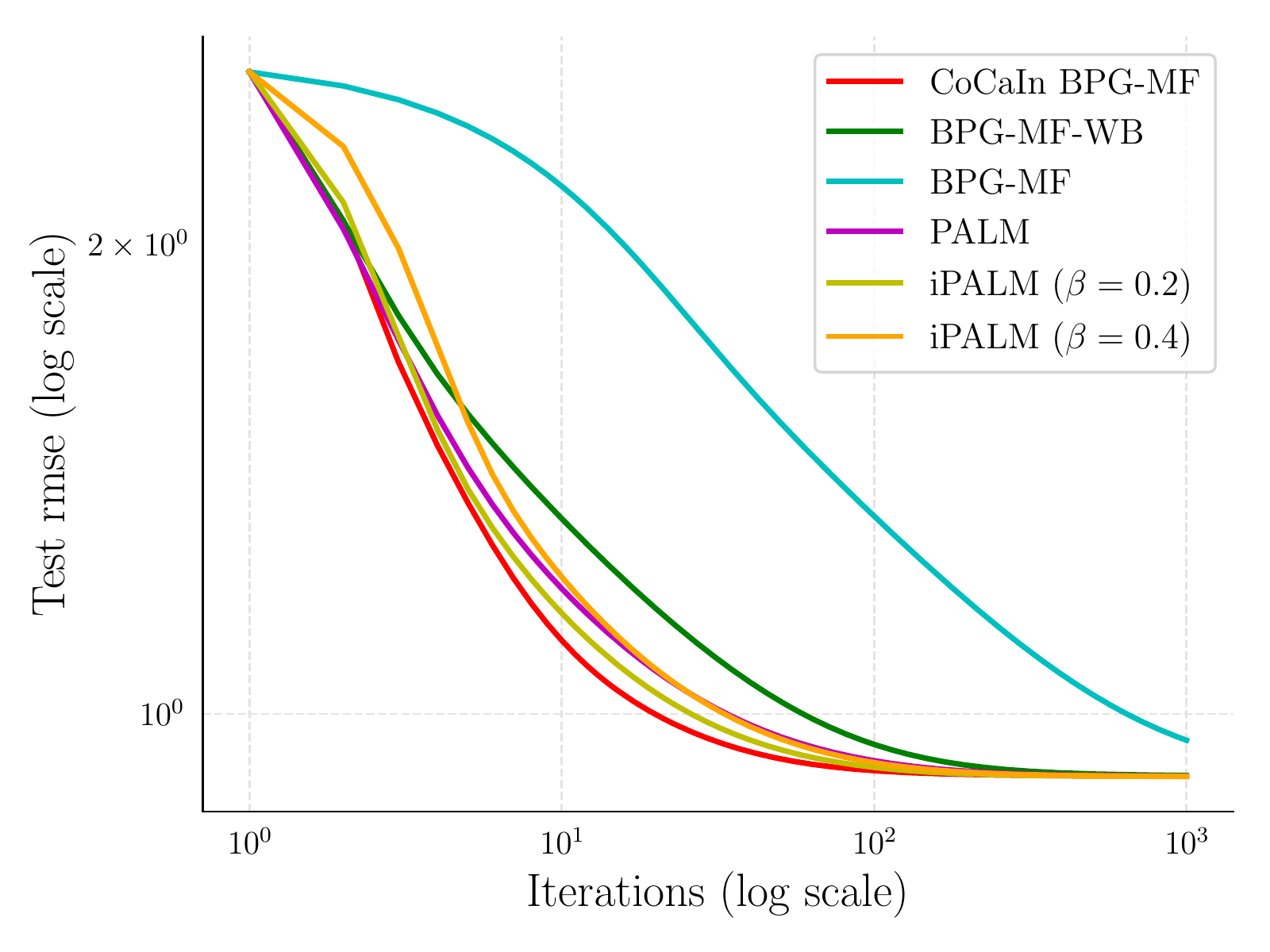}
			\caption{MovieLens-1M}
		\end{subfigure}
		\begin{subfigure}{0.325\textwidth}
			\centering
			\includegraphics[width=0.9\textwidth]{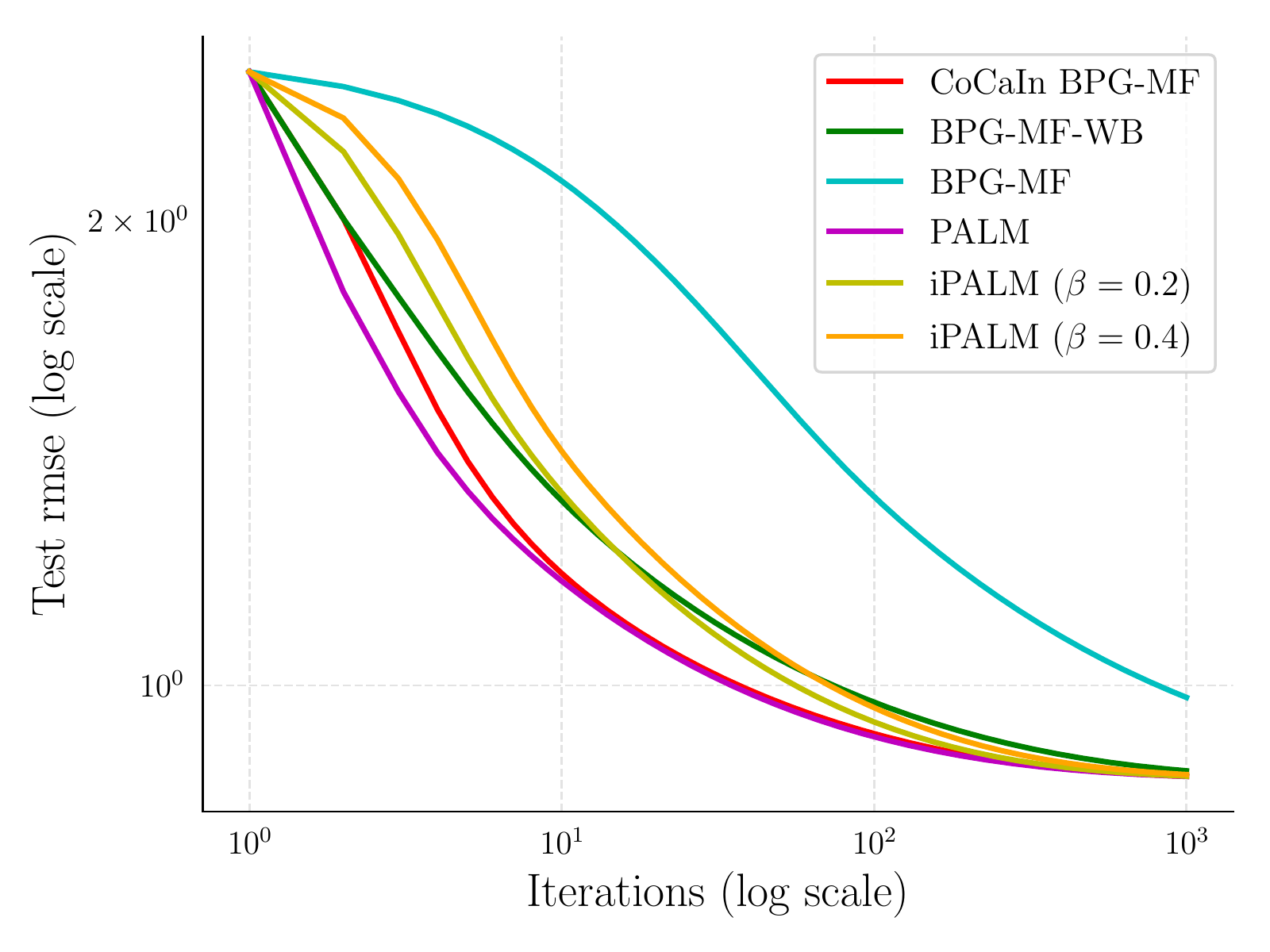}
			\caption{MovieLens-10M}
		\end{subfigure}
		\caption{\textbf{Test RMSE plot on MovieLens Datasets \cite{HK2016}.}}
		\label{fig:test-rmse-plot}
\end{figure}\\
The above given figures show that the proposed methods BPG-MF-WB and CoCaIn BPG-MF are competitive to PALM and iPALM. BPG-MF is slow in the beginning, however it is competitive to other methods towards the end.
\subsection{Time Comparisons}
We provide time comparisons in Figures~\ref{fig:time-synthetic}, \ref{fig:time-nmf}, \ref{fig:time-movielens} for all the experimental settings mentioned in Section~\ref{sec:exps}, where we mention the dataset in the caption. Since, we used logarithmic scaling, we used an offset of $10^{-2}$ for all algorithms for better visualization.\ifpaper\else\\\fi

\begin{figure}[htb!]
		\centering
		\begin{subfigure}{0.32\textwidth}
			\centering
			\includegraphics[width=1\textwidth]{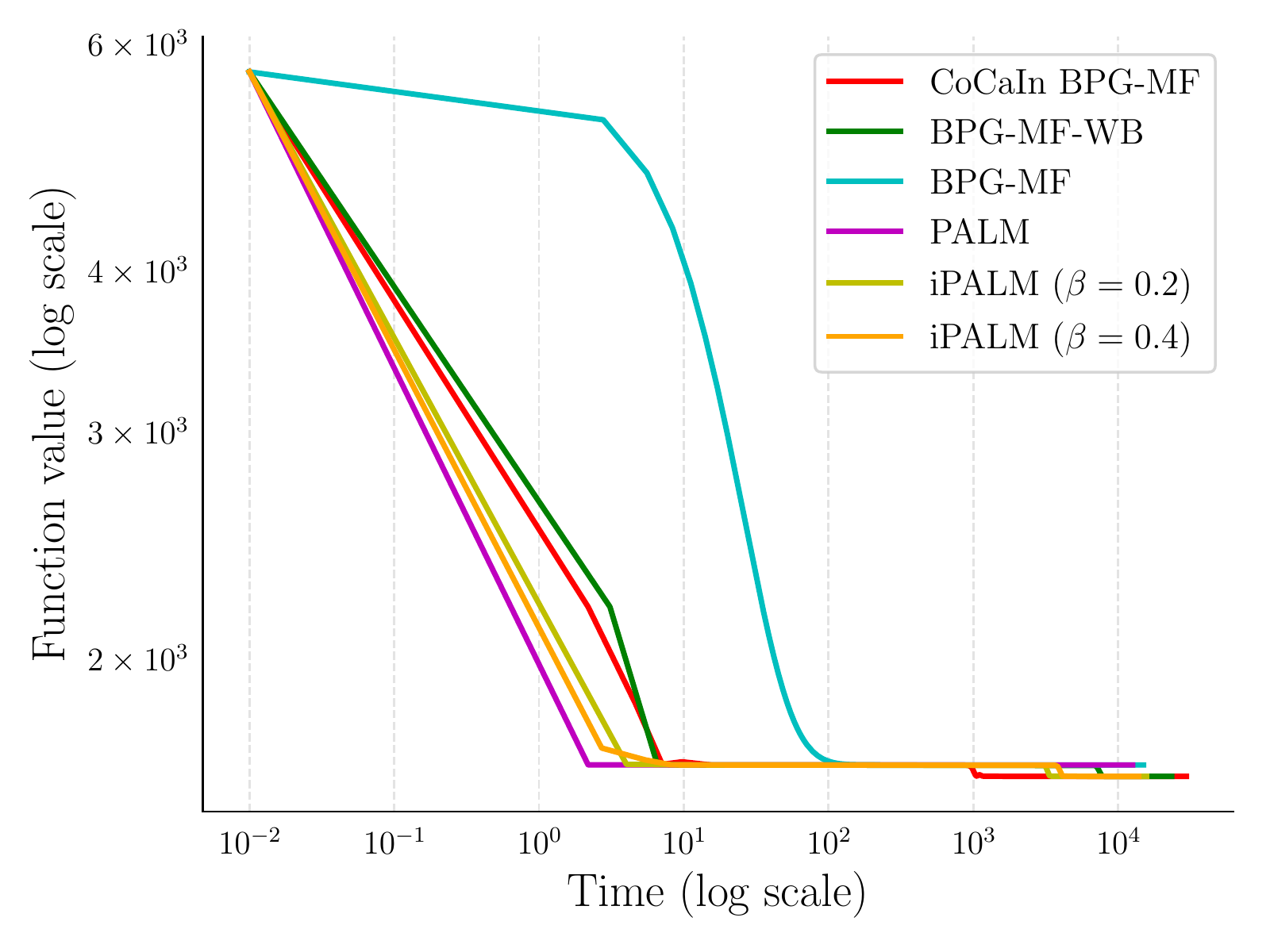}
			\caption{No Regularization}
		\end{subfigure}
		\begin{subfigure}{0.32\textwidth}
			\centering
			\includegraphics[width=1\textwidth]{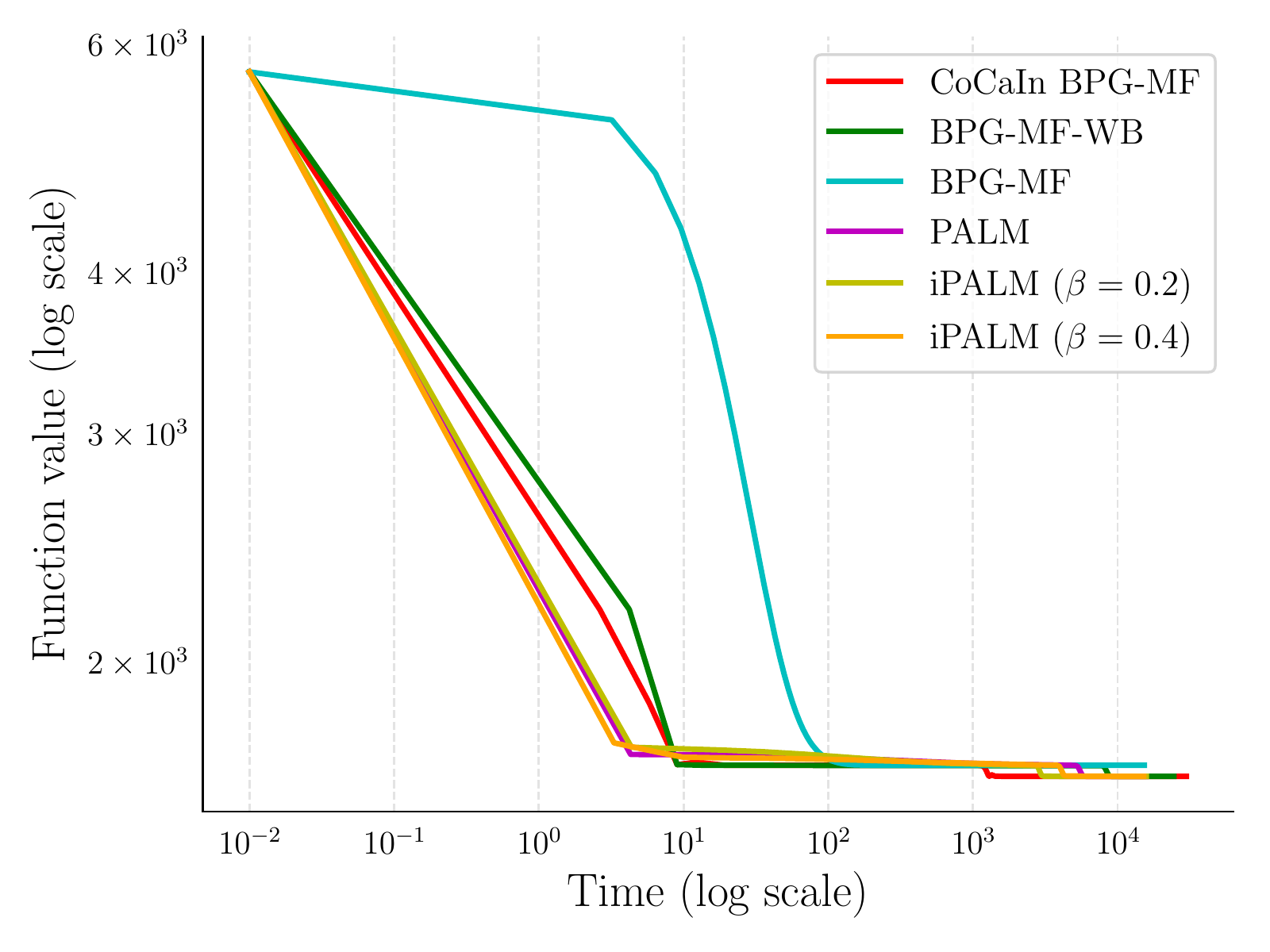}
			\caption{L2-Regularization}
		\end{subfigure}
		\begin{subfigure}{0.32\textwidth}
			\centering
			\includegraphics[width=1\textwidth]{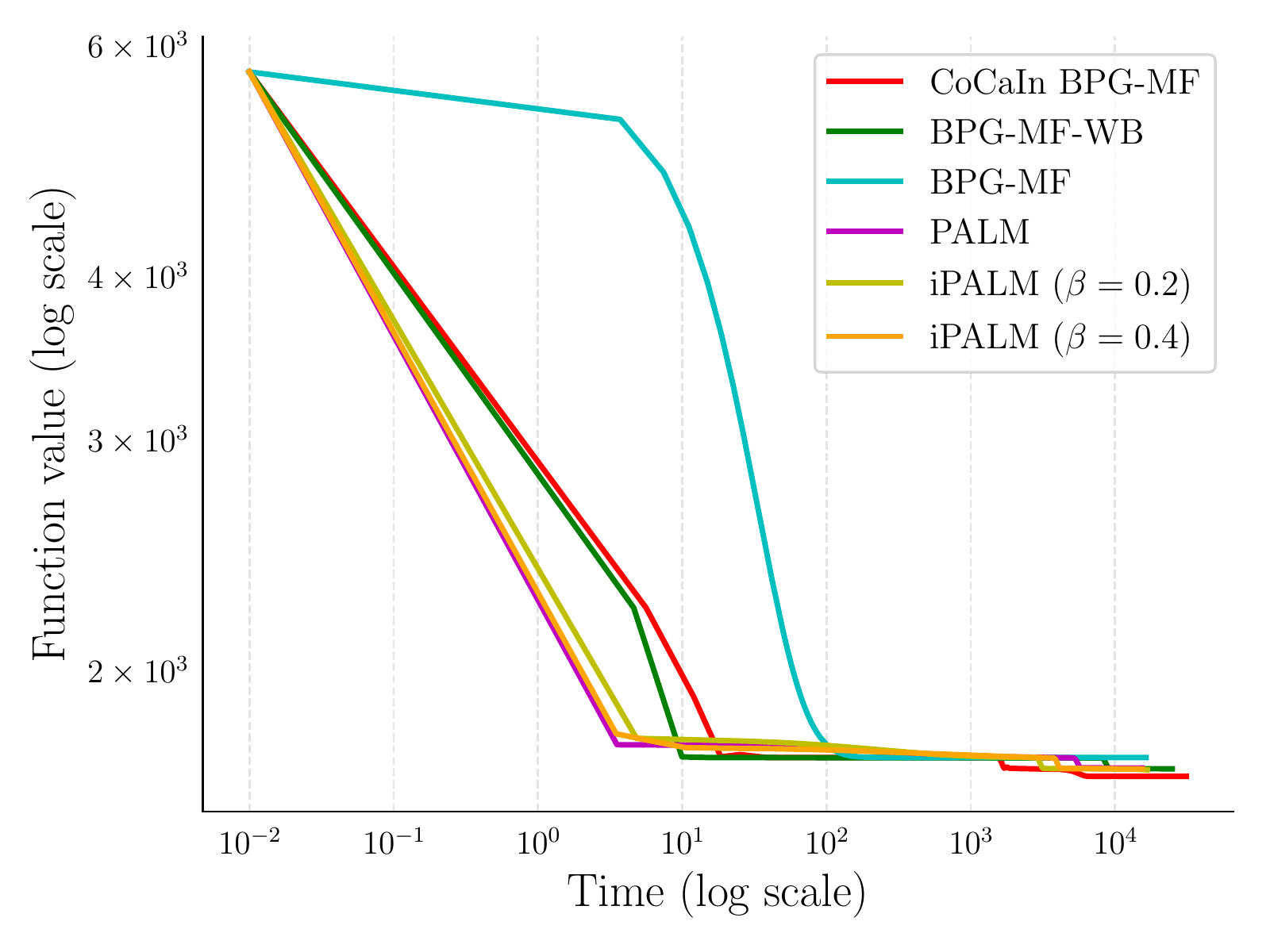}
			\caption{L1-Regularization}
		\end{subfigure}
		\caption{\textbf{Time plots for Simple Matrix Factorization on Synthetic Dataset.}}
		\label{fig:time-synthetic}
\end{figure}
\begin{figure}[htb!]
		\begin{subfigure}{0.32\textwidth}
			\centering
			\includegraphics[width=1\textwidth]{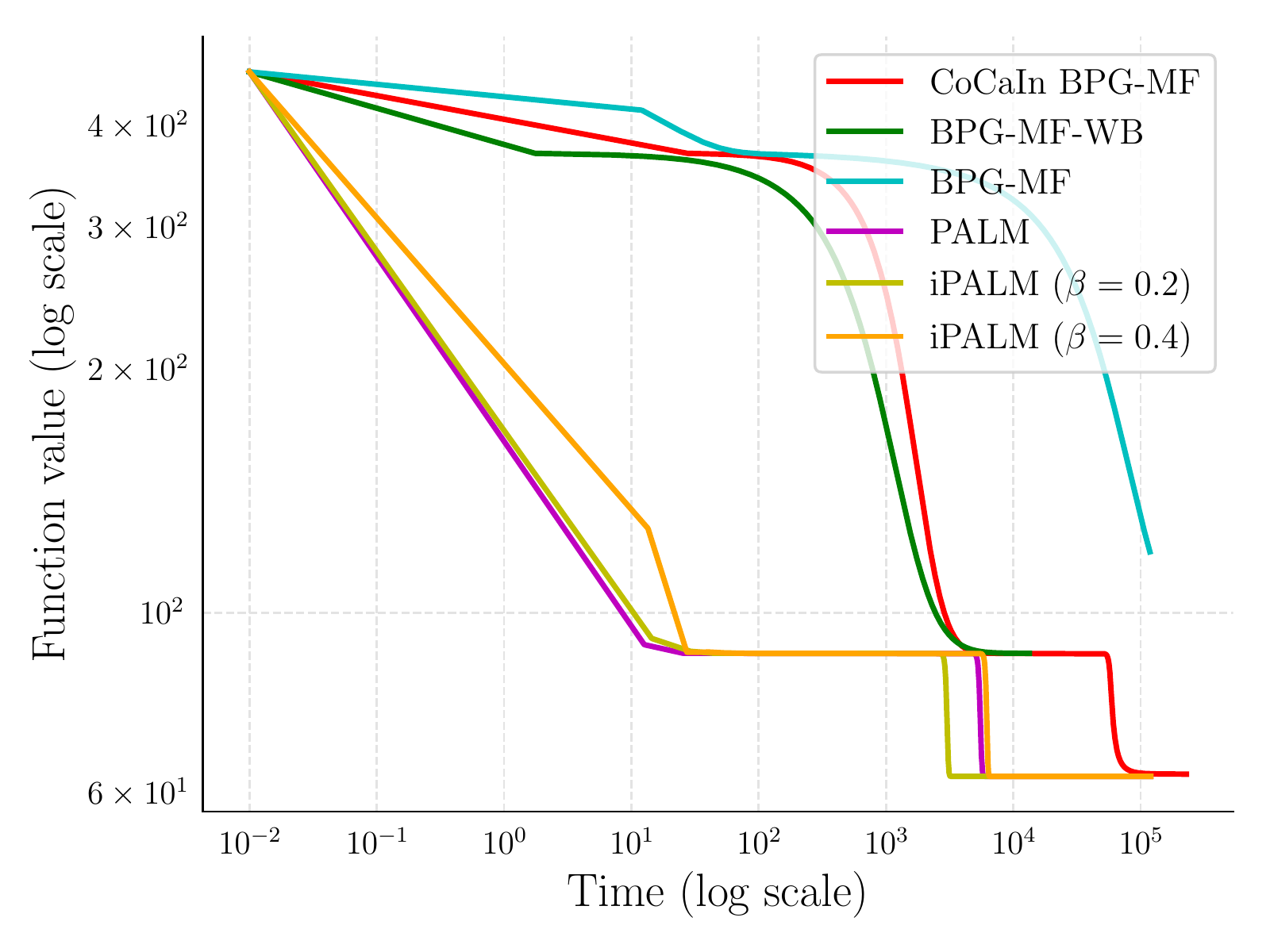}
			\caption{No-Regularization}
		\end{subfigure}
		\begin{subfigure}{0.32\textwidth}
			\centering
			\includegraphics[width=1\textwidth]{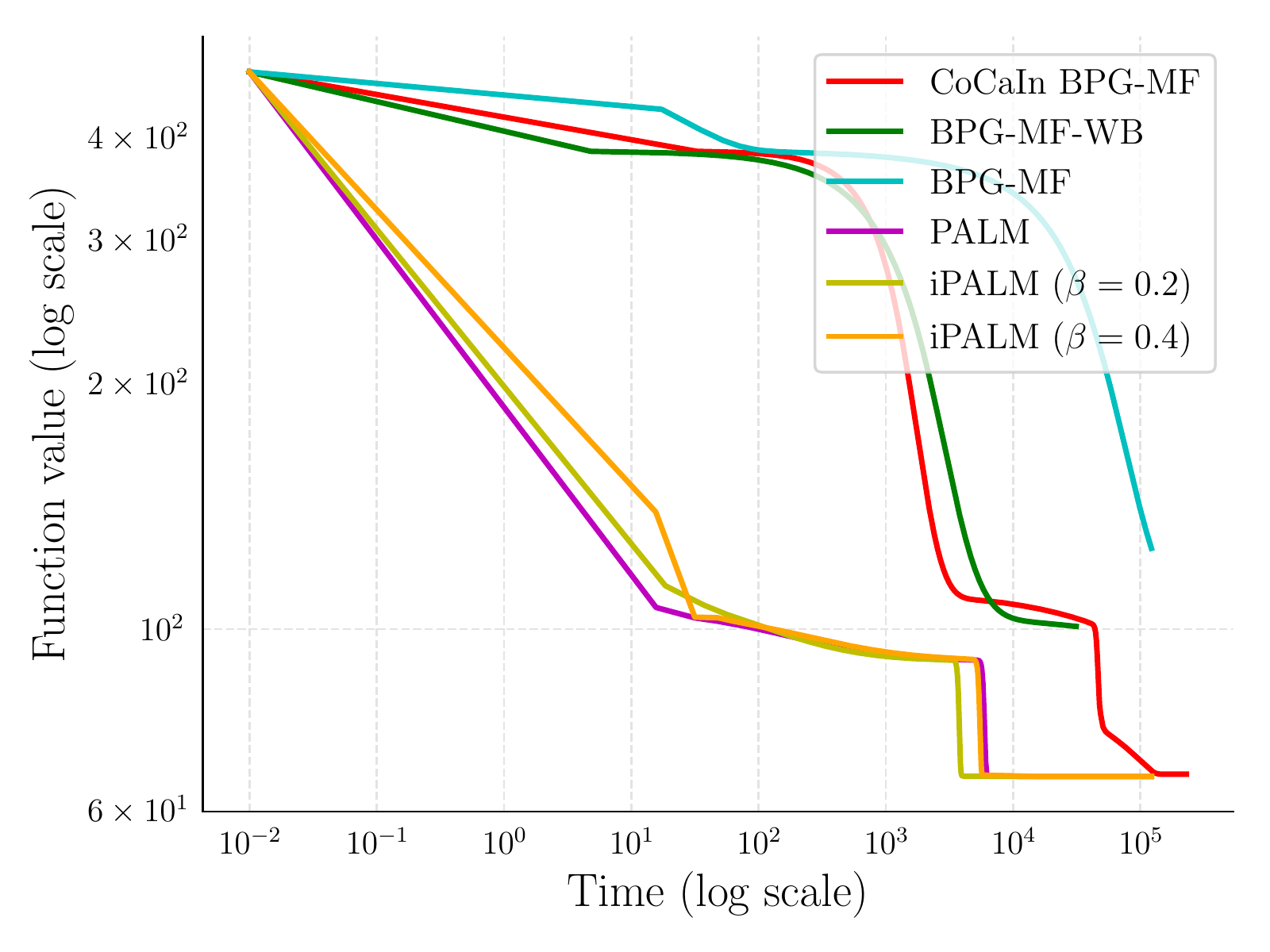}
			\caption{L2-Regularization}
		\end{subfigure}
		\begin{subfigure}{0.32\textwidth}
			\centering
			\includegraphics[width=1\textwidth]{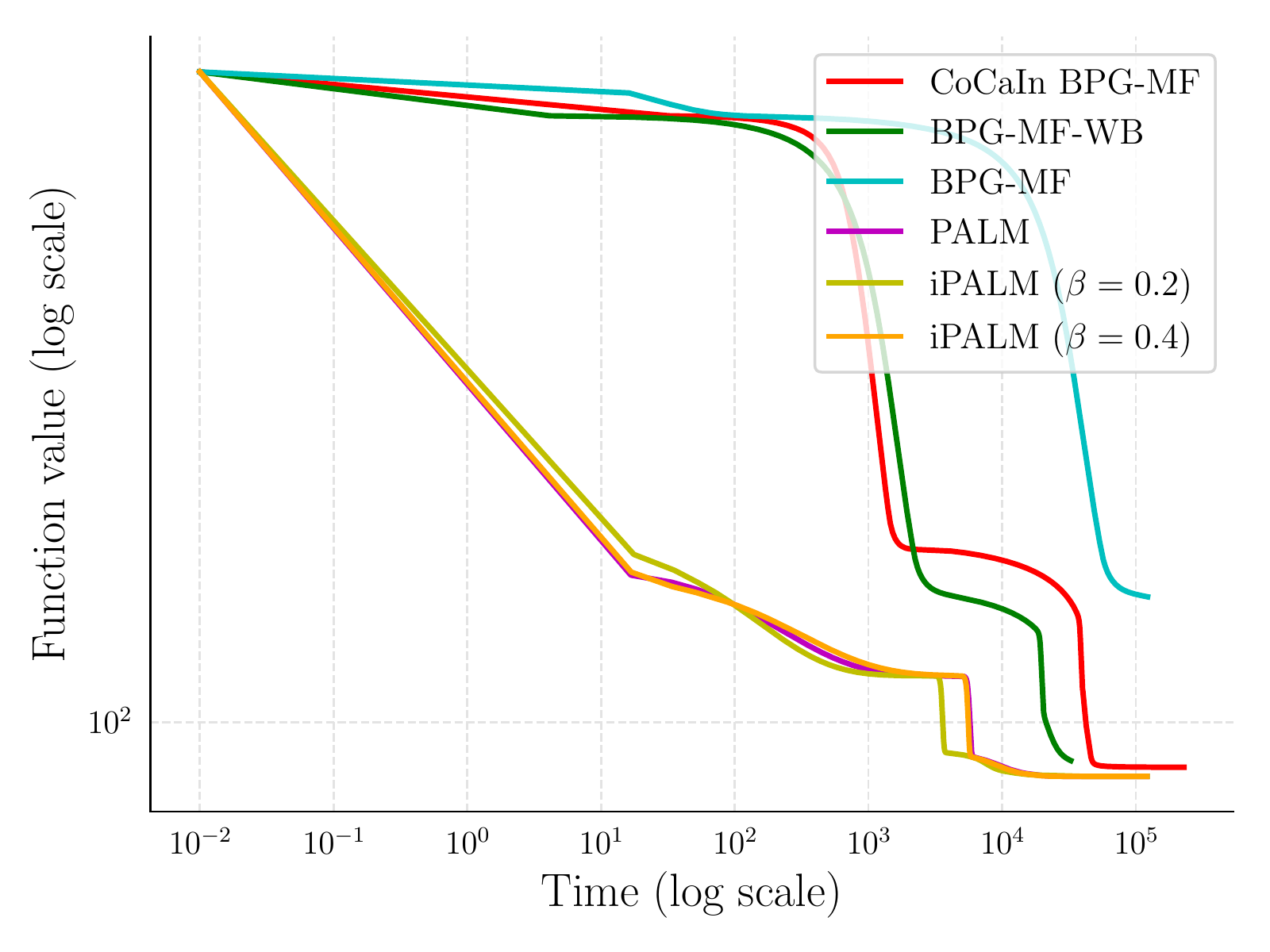}
			\caption{L1-Regularization}
		\end{subfigure}
		\caption{\textbf{Time plots for Non-negative Matrix Factorization on  Medulloblastoma dataset \cite{BTGM2004}.} }
		\label{fig:time-nmf}
\end{figure}
\begin{figure}[htb!]
		\begin{subfigure}{0.325\textwidth}
			\centering
			\includegraphics[width=0.9\textwidth]{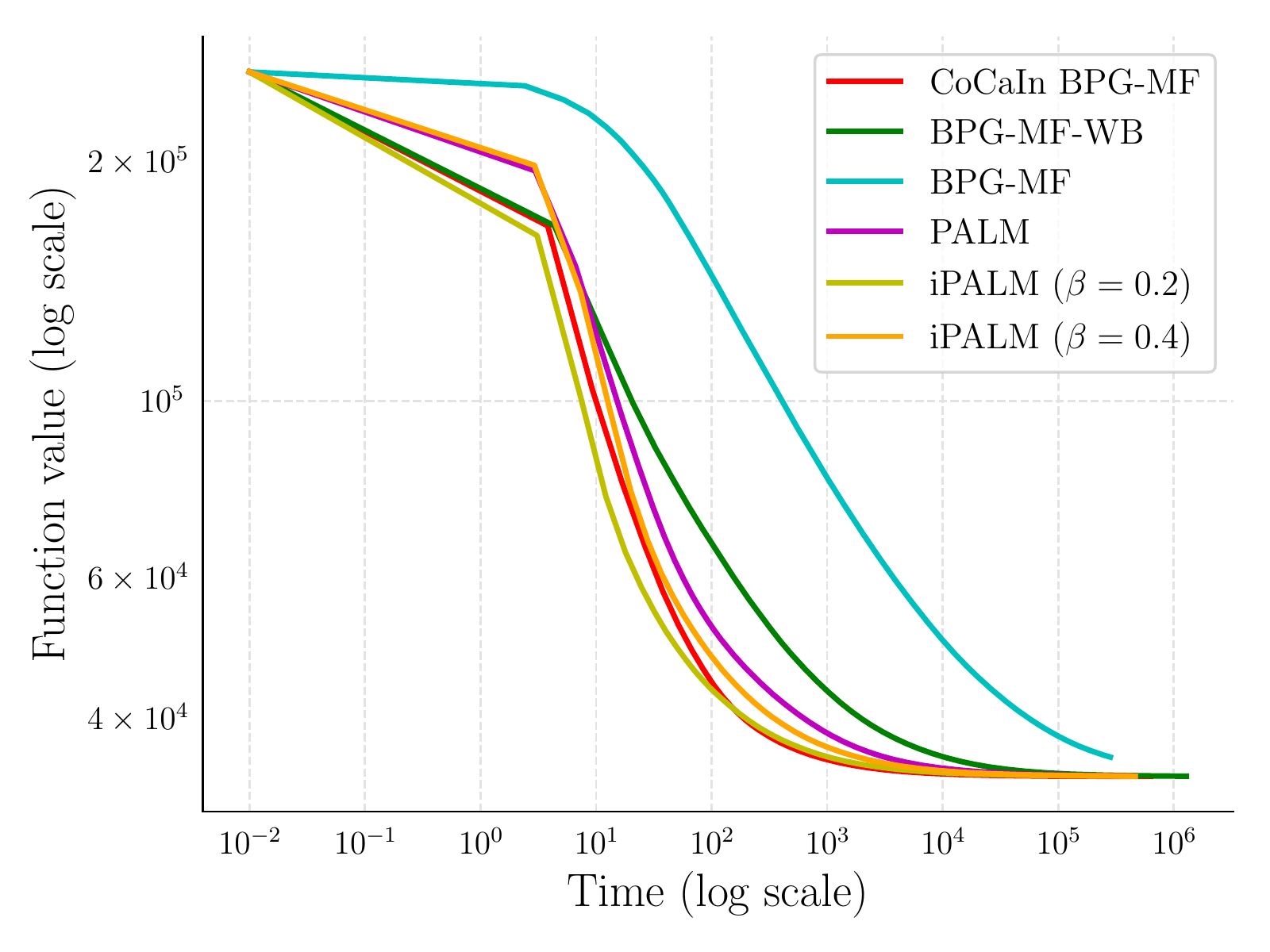}
			\caption{MovieLens-100K}
		\end{subfigure}
		\begin{subfigure}{0.325\textwidth}
			\centering
			\includegraphics[width=0.9\textwidth]{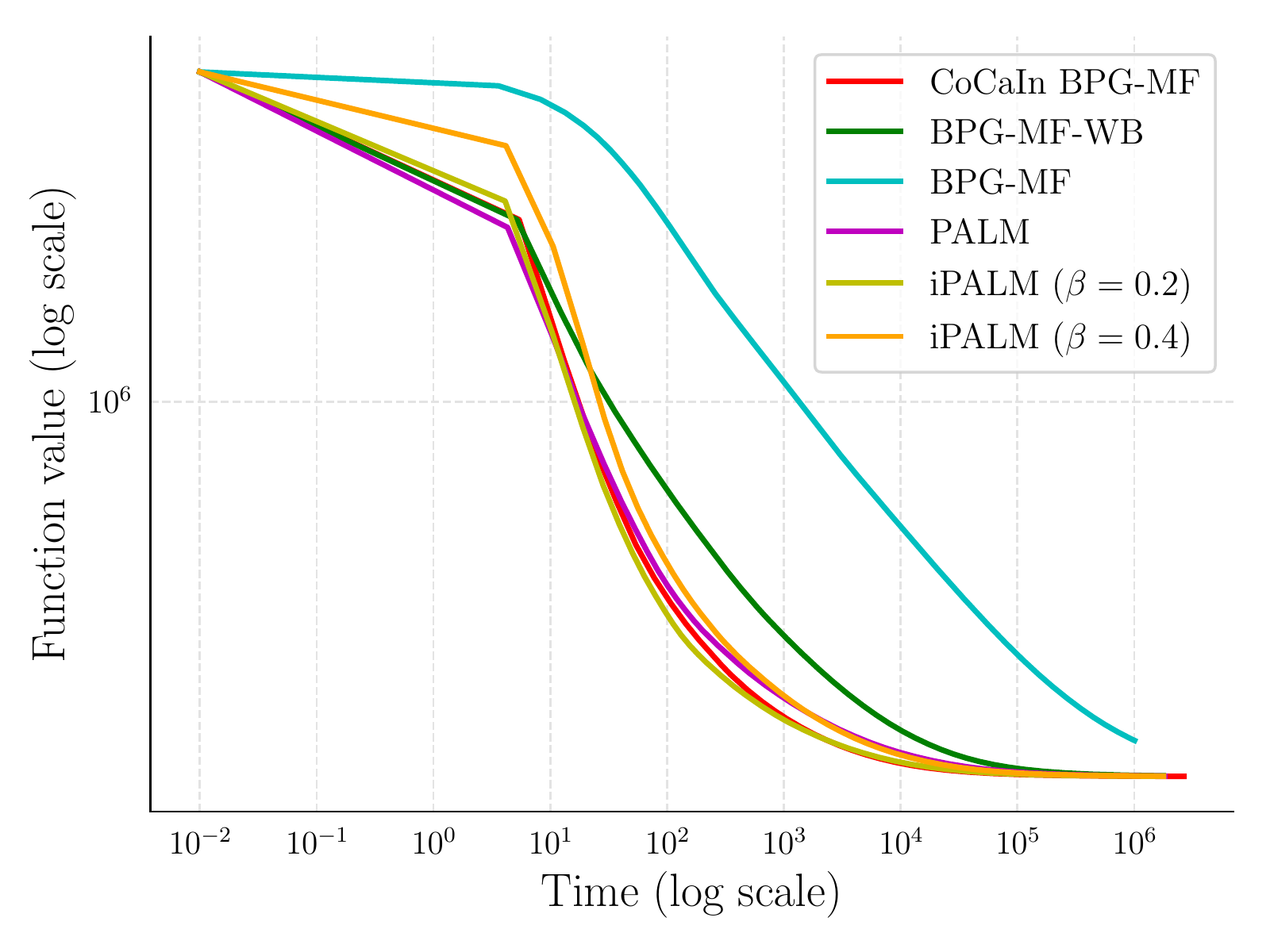}
			\caption{MovieLens-1M}
		\end{subfigure}
		\begin{subfigure}{0.325\textwidth}
			\centering
			\includegraphics[width=0.9\textwidth]{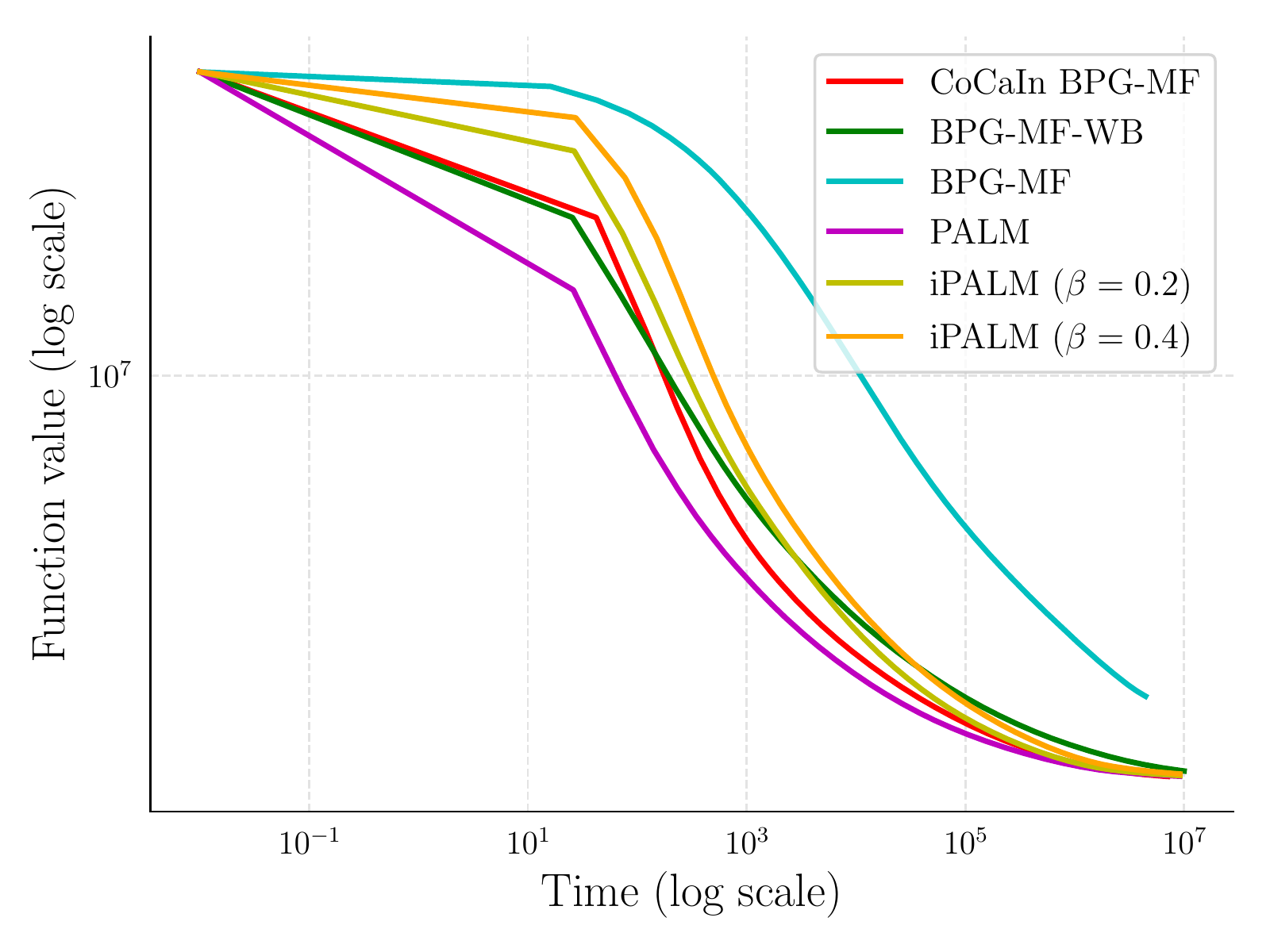}
			\caption{MovieLens-10M}
		\end{subfigure}
		\caption{\textbf{Time plots for Matrix Completion on MovieLens Datasets \cite{HK2016}.} }
		\label{fig:time-movielens}
	\end{figure}
 As evident from the plots, the proposed variants BPG-MF-WB and CoCaIn BPG-MF are competitive that PALM and iPALM. And, BPG-MF is mostly slow,  due to constant step-size, which can be potentially helpful when backtracking is computationally expensive.\ifpaper\else\\\fi

\ifpaper
\else
\bibliographystyle{plain}
\bibliography{notes}
\fi
\end{document}